\documentclass[a4paper,11pt]{article}

\usepackage{amsfonts}
\usepackage{amssymb}
\usepackage{amsthm}
\usepackage{amsmath}
\usepackage{a4wide}
\usepackage{mathrsfs}
\usepackage{epsfig}
\usepackage{palatino}
\usepackage{tikz,fp,ifthen}
\usepackage{float}
\usepackage{nicefrac}

\newcommand{\pdfgraphics}{\ifpdf\DeclareGraphicsExtensions{.pdf,.jpg}\else\fi}
\usepackage{graphicx}

\usepackage{color}
\definecolor{citegreen}{rgb}{0,0.6,0}
\definecolor{refred}{rgb}{0.8,0,0}
\usepackage[colorlinks, citecolor=citegreen,linkcolor=refred]{hyperref}

\usepackage{mathtools}
\numberwithin{equation}{section}

\theoremstyle{plain}

\newtheorem*{theorem*}{Theorem}

\newtheorem{teo}{Theorem}[section]
\newtheorem{lemma}[teo]{Lemma}
\newtheorem{prop}[teo]{Proposition}
\newtheorem{cor}[teo]{Corollary}

\theoremstyle{definition}
\newtheorem{dfnz}[teo]{Definition}

\theoremstyle{remark}
\newtheorem{rem}[teo]{Remark}

\numberwithin{equation}{section}

\newcommand{\intbar}{\etaathop{\int\etaakebox(-13.5,0){\rule[4pt]{.7em}{0.3pt}}
		\kern-6pt}\nolimits}

\newcommand{\be}{\begin{equation}}
\newcommand{\ee}{\end{equation}}
\newcommand{\bea}{\begin{equation*}}
\newcommand{\eea}{\end{equation*}}

\def\dert{\partial_t}
\def\ders{\partial_s}

\def\pol{{\mathfrak{p}}}

\def\be{\begin{equation}}
\def\ee{\end{equation}}
\def\bea{\begin{eqnarray*}}
	\def\bean{\begin{eqnarray}}
	\def\eean{\end{eqnarray}}
	\def\eea{\end{eqnarray*}}

\newcommand{\vertiii}[1]{{\left\vert\kern-0.25ex\left\vert\kern-0.25ex\left\vert #1 
		\right\vert\kern-0.25ex\right\vert\kern-0.25ex\right\vert}}
\newcommand{\vertiiii}[1]{{\left\vert\kern-0.25ex\left\vert\kern-0.25ex\left\vert #1 
		\right\vert\kern-0.25ex\right\vert\kern-0.25ex\right\vert}}

\begin{document}
\pdfgraphics % Use this command right after \begin{document}

\title{Long Time Existence of Solutions to an Elastic Flow of Networks}

\author{Harald Garcke 
\footnote{Fakult\"at f\"ur Mathematik, Universit\"at Regensburg, Universit\"atsstrasse 31, 
93053 Regensburg, Germany}
 \and Julia Menzel
\footnotemark[1]
\and Alessandra Pluda
\footnotemark[1]
}

\maketitle

\begin{abstract}
\noindent The $L^2$--gradient flow of the elastic energy of networks leads to a Willmore type evolution law with non-trivial nonlinear boundary conditions. We show local in time existence and uniqueness for this elastic flow of networks in a Sobolev space setting under natural boundary conditions. In addition we show a regularisation property and geometric existence and uniqueness. The main result is a long time existence result using energy methods.
\end{abstract}


\section{Introduction}
In this paper we show long time existence of solutions to the elastic flow of planar networks 
and extend a previous result which showed that the flow 
was locally well posed in parabolic H\"{o}lder spaces.

A network $\mathcal{N}$ is a connected \emph{set} in the plane composed of a finite union 
of regular curves $\mathcal{N}^i$ that meet in triple junctions 
and may have endpoints fixed in the plane.

\begin{figure}[h]
\centering
\includegraphics{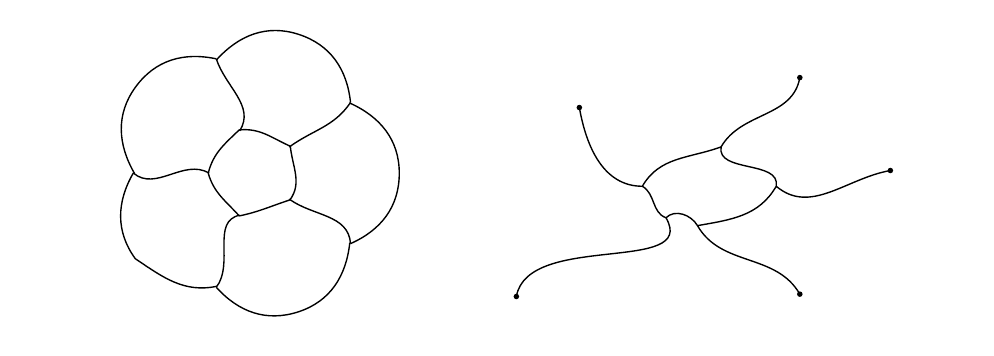}
\caption{Two examples of networks}	
\end{figure}

We let evolve a suitable initial network $\mathcal{N}_0$ 
by a linear combination of the Willmore and the Mean Curvature Flow:
each curve of the network moves with \emph{normal velocity} (at any time and point)
\begin{equation*}
V^\perp=- 2k_{ss}-k^3+\mu k\,,
\end{equation*}
where $k$ is the curvature, $s$ the arclength parameter 
and $\mu$ a positive constant. 
We vary also in tangential direction and allow the triple junctions to move.
It is shown in~\cite{bargarnu} that this motion equation coupled with 
suitable (geometric) conditions at the junctions and at the fixed endpoints
can be understood as the 
(formal) $L^2$--gradient flow of the Elastic Energy
\begin{equation}
E_\mu\left(\mathcal{N}\right):=\int_{\mathcal{N}}\left(k^{2}+\mu\right)ds
=\sum_{i}\int_{\mathcal{N}^{i}}\left((k^i)^{2}+\mu\right)ds^i\,.
\end{equation}
In our preceding work~\cite{garmenzplu} we have established 
well--posedness
for various sets of conditions at the boundary points.
More precisely, starting with a geometrically admissible initial network 
(a network satisfying all boundary conditions that are asked to be valid for the flow, 
see~\cite[Definition 3.2]{garmenzplu})
of class $C^{4+\alpha}$ with $\alpha\in(0,1)$ 
there exists a \emph{unique} evolution of networks 
in a possibly short time interval $[0,T]$ with $T>0$, 
that can be described by one parametrisation of class 
$C^{\frac{4+\alpha}{4},4+\alpha}\left([0,T]\times[0,1]\right)$.
Further, the initial network needs to be \emph{non--degenerate} in the sense that 
at each triple junction the three tangents are not linearly dependent. 

\begin{figure}[h]
\centering	 
\includegraphics{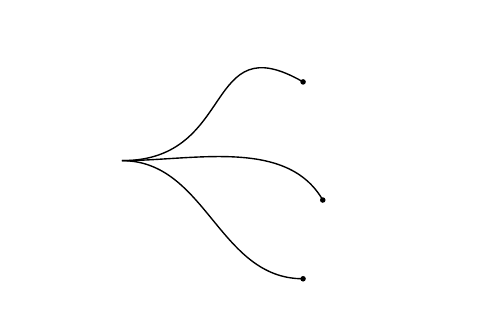}
 	\caption{A non--admissible initial network}	
 \end{figure}
 In this paper we lower the regularity of the initial datum 
 allowing admissible non--degenerate initial networks
 of class $W_p^{4-\nicefrac{4}{p}}$ for $p\in(5,\infty)$ 
 (see Definition~\ref{admg1}). 
 In contrast to the above mentioned classical case the initial datum has to satisfy 
 merely the boundary conditions of the flow and no
compatibility conditions of fourth order (see~\cite{solonnikov2}).
For this class of initial data we obtain a unique solution of the problem 
in a Sobolev setting.
\begin{teo}\label{shorttime}
Given an admissible and non--degenerate initial network $\mathcal{N}_0$ of regularity $W_p^{4-\nicefrac{4}{p}}$ with $p\in(5,\infty)$, there exists a positive time $T$ such that within the interval $[0,T]$ the elastic flow with natural boundary conditions, c.f.~\eqref{Triod0},
admits a geometrically unique solution $\mathcal{N}$ that
can be  
parametrised by maps of class 
$$
W_p^1\left((0,T);
L_p\left(0,1\right)\right)\cap L_p\left((0,T);W_p^4\left(0,1\right)\right)
\,.
$$
\end{teo}
We emphasise that the elastic flow is a \emph{geometric} evolution equation of \emph{sets} in the plane
and that all the results of this paper hold in this geometric sense. Nevertheless,
we prove Theorem~\ref{shorttime}
using an auxiliary system of non--degenerate parabolic
partial differential equations for maps parametrising the networks that we call Analytic Problem.
The crucial difficulty is to find a tangential velocity turning system~\eqref{Triod0} into a well--posed parabolic boundary value problem without changing the geometric nature of the evolution. To this end one further needs to impose additional conditions on the parametrisations of the curves at the boundary points.

We solve the Analytic Problem using a linearisation procedure and a fixed point argument. 

One important feature of the parabolic structure of~\eqref{TriodC^0} is that solutions are \emph{smooth} for positive times. Indeed, using a cut-off function in time and an auxiliary linear parabolic system, the classical theory in~\cite{solonnikov2} implies that the solution starting in any admissible initial datum of regularity $W_p^{4-\nicefrac{4}{p}}$, $p\in(5,\infty)$, lies in
$C^{\infty}\left([\varepsilon,T]\times[0,1]\right)$ for all positive $\varepsilon$.

This allows us to improve the local existence Theorem~\ref{shorttime}.
\begin{teo}\label{parabolicsmoothing}
There exists a positive time $T$ such that within the interval $[0,T]$ the elastic flow
admits a unique solution $\mathcal{N}$ that is \emph{smooth} for positive times.
\end{teo}

These local results are not only the foundation to any further analysis of the flow but also a key tool in proving the following:

\begin{teo}\label{main}
Let $p\in(5,10)$ and $\mathcal{N}_0$ be a geometrically admissible initial network. 
Suppose that $\left(\mathcal{N}(t)\right)_{t\in[0,T_{\max})}$ is a 
maximal solution to
the elastic flow with initial datum $\mathcal{N}_0$ 
in the maximal time interval $[0,T_{\max})$ with $T_{\max}\in (0,\infty)\cup\{\infty\}$. 
Then $$T_{max}=\infty$$ or at least one of the following happens:
\begin{itemize}
\item[(i)] the inferior limit of the length of at least one curve of $\mathcal{N}(t)$ is zero as 
$t\nearrow T_{max}$.
\item[(ii)]
at one of the triple junctions 
$\liminf_{t\nearrow T_{max}}\max\left\{\left\vert\sin\alpha^1(t)\right\vert,\left\vert\sin\alpha^2(t)\right\vert,\left\vert\sin\alpha^3(t)\right\vert\right\}=0$, 
where $\alpha^1(t)$, $\alpha^2(t)$ and $\alpha^3(t)$ are the angles 
at the respective triple junction.
\end{itemize}
\end{teo}
At this point we would like to say a few words about the strategy of the proof.
Thanks to our short time existence result to contradict the finiteness of $T_{max}$ 
it is enough to find parametrisations admitting uniform bounds 
in the norm $W_p^{4-\nicefrac{4}{p}}$ on $[0,T_{max})$. This simplifies the required energy type estimates.
Indeed to obtain such a bound it suffices to find 
a uniform in time estimate on the $L^2$--norm of the second arclength derivative of the curvature. 
These can be derived under the assumption that during the evolution 
all the lengths are uniformly bounded away from zero 
and that the network remains non--degenerate in a uniform sense 
(see condition~\eqref{nondegeneracy}).
It is important to underline that only estimates on geometric quantities, namely the curvature, are needed. In particular, the proof itself is independent of the choice of tangential velocity which corresponds to the very definition of the flow, where only the normal velocity is prescribed.
We notice that the possible behaviours $(i)$ and $(ii)$ are not merely technical assumptions but also quite realistic scenarios for the nature of potential singularities.
%(see the example of minimizing sequence discussed in~\cite[\S~3]{dallapluda}).
Moreover, none of the listed possibilities excludes the others:
as the time approaches $T_{max}$, both the lengths of one or several curves of the network and the angles between the curves at one or more triple junctions can 
go to zero, regardless of $T_{\max}$ being finite or infinite. 

\medskip

It is shown both in~\cite{dziukkuwertsch} and~\cite{poldenthesis} 
that the evolution of closed curves with respect to Willmore Flow
(both with a length penalization and fixed length) exists globally in time (for the Helfrich flow see also~\cite{glen}). 
The same result is obtained 
for several problems related to open curves with fixed endpoints or asymptotic
to a line at infinity  (see also~\cite{lin2, dalpoz,lin1, novok,okabe}).
While the geometric evolution of submanifolds has been extensively studied in the recent years, the situation is different when we consider motions of generalized possibly singular submanifolds. The simplest example of such objects are networks.
Short time existence results for geometric flows of triple junction clusters
have been obtained by several
authors (see~\cite{Bronsardreitich, DepnerGarckeKohsaka, Freire, GarckeItoKohsaka, michaelgoesswein, schulzewhite}). 
The motion of networks evolving by curvature has been extensively studied by Mantegazza-Novaga-et.al. (see~\cite{Carlo2, Carlo3, mannovtor, pluda}). Here
each curve of the
network moves with normal velocity equal to its curvature and 
at the triple junctions the angles are fixed to be equal.
The analysis of the 
long time behaviour
shows that as $t\to T_{\max}$ the curvature is unbounded 
or the length of one (or more) of the curves of the network goes to zero. In our case, the curvature can not become unbounded as the Willmore energy is decreasing during the evolution.
The second possibility is precisely one possible behaviour stated in our result. Since a non--degeneracy condition on the angles is needed for the well posedness
of our flow, it is not surprising that the evolution of the angles plays a role
in Theorem~\ref{main}.

An important aspect of the elastic flow of networks
that we have not considered yet
is the definition of the flow past singularities.
To attack this problem a short time existence result for
networks with junctions of order higher than three needs to be established.

Finally we refer to~\cite{bargarnu} for numerical analysis and simulations 
for the elastic flow of networks.

\medskip

After the completion of this paper we got aware of the work ``Flow of elastic networks: long--time existence result"  by Dall'Acqua, Lin, Pozzi, see~\cite{dallacqualinpozzi} where the authors give a long time existence result under the hypothesis that smooth solutions exist for a uniform short time interval.

%Relate the paper with the previous literature
%- long time behavior analysis -- mean curvature flow, explain the differences ecc ecc
%- numerics (Harald) -- all the thing he did with triple junctions
%- minimization problems 
%- short time with non regular data 

%%%%%%%%%%%%%%%%%%%%%%%%%%%%%%%%%%%%%%%%%%%%%%%%%%%
%%%%%%%%%%%%%%%%%%%%%%%%%%%%%%%%%%%%%%%%%%%%%%%%%%%
\section{Elastic flow of networks}
%%%%%%%%%%%%%%%%%%%%%%%%%%%%%%%%%%%%%%%%%%%%%%%%%%%
%%%%%%%%%%%%%%%%%%%%%%%%%%%%%%%%%%%%%%%%%%%%%%%%%%%

In the following we will denote by $\left\lvert x\right\rvert$ the euclidean norm of a vector $x\in\mathbb{R}^d$ with $d\geq 1$. All norms of function spaces are taken with respect to the euclidean norm.

\begin{dfnz}
A planar network $\mathcal{N}$ is a connected set in the plane consisting 
of a finite union of regular curves $\mathcal{N}^i$ that meet at their end--points in junctions.
Each curve $\mathcal{N}^i$ admits a regular $C^1$--parametrisation, that is, a map $\gamma^i:[0,1]\to\mathbb{R}^2$ of class $C^1$ with $\vert\gamma^i_x\vert\neq 0$ on $[0,1]$ and $\gamma^i\left([0,1]\right)=\mathcal{N}^i$.
\end{dfnz}

\begin{dfnz}
Let $k>1$ and $1\leq p\leq\infty$ with $p>\frac{1}{k-1}$. A network $\mathcal{N}$ is of class $C^k$ (or $W_p^{k}$, respectively) if it admits
a regular parametrisation 
of class $C^k$ (or $W_p^{k}$, respectively).
\end{dfnz}

\begin{dfnz}
Consider two networks $\mathcal{N}$ and $\widehat{\mathcal{N}}$ with regular parametrisation $\gamma$ and $\widehat{\gamma}$
of class $C^k$
(or $W_p^{k}$), respectively.
We say $\mathcal{N}=\widehat{\mathcal{N}}$ if 
they coincide as sets in $\mathbb{R}^2$ and if there exists 
a reparametrisation $\sigma^i:[0,1]\to[0,1]$ with $\sigma^i(0)=0$, $\sigma^i(1)=1$
such that $\gamma^i\circ\sigma^i=\widehat{\gamma}^i$.
\end{dfnz}

We restrict to networks with \emph{triple} junctions.
In particular we focus on two different topologies of planar networks
(triods and Theta--networks) that can be regarded as prototypes of more general configurations.

\begin{dfnz}\label{deftheta}
	A Theta--network $\Theta$ is a network in $\mathbb{R}^2$ composed of three regular 
	curves that intersect each other at their endpoints in two triple junctions.
\end{dfnz}

\begin{figure}[H]
\centering
\includegraphics{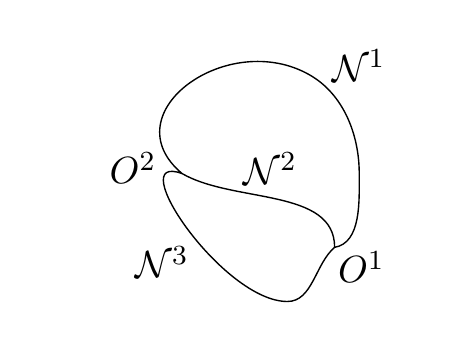}
 \begin{caption}{A Theta--network}
\end{caption}
\end{figure}

\begin{dfnz}
A triod $\mathbb{T}$ is a planar network composed of three regular curves  $\mathbb{T}^i$
that meet at one triple junction and have the other endpoints fixed in the plane. 
\end{dfnz}

Moreover we will adapt the following convention. Every regular parametrisation $\gamma$ of a triod is such that the triple junction denoted by $O$
coincides with $\gamma^1(0)=\gamma^2(0)=\gamma^3(0)$
and $\gamma^i(1)=P^i$ with $i\in\{1,2,3\}$ are the three fixed \textit{distinct} endpoints.
Denoting by $\tau_1, \tau_2,\tau_3$ the unit tangent vectors to the three curves
of a triod,
we call  $\alpha_3,\alpha_1$ and $\alpha_2$ the angle at the triple junctions
between $\tau^1$ and  
 $\tau^2$, $\tau^2$ and  $\tau^3$, and $\tau^3$ and $\tau^1$, respectively.

\begin{figure}[H]
\centering
\includegraphics{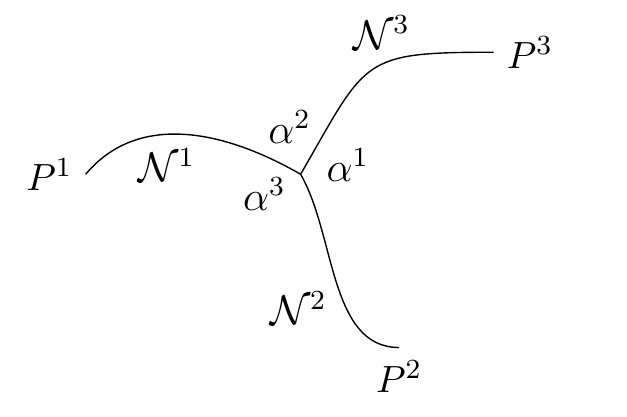}
\begin{caption}{A Triod}
\end{caption}
\end{figure}

In the following we will consider a time dependent family of networks
$\left(\mathcal{N}(t)\right)_{t\in[0,T)}$ for some $T>0$ that are parametrised by 
$\gamma=(\gamma^1,\ldots,\gamma^j)$
in a suitable regularity class
and evolve according to the $L^2$--gradient flow of the elastic energy
\begin{equation}\label{eef}
E_\mu\left(\mathcal{N}\right):=\int_{\mathcal{N}}\left(k^{2}+\mu\right)ds
=\sum_{i}\int_{\mathcal{N}^{i}}\left((k^i)^{2}+\mu\right)ds^i\,
\end{equation}
with $\mu>0$. Here $k^i$ and $s^i$ denote the curvature 
and arc--length parameter of $\gamma^i$, respectively.

\medskip

While the normal velocity
\begin{equation}
-A^i:=-\left(2k_{ss}^{i}+\left(k^{i}\right)^{3}-\mu k^{i}\right)
\end{equation}
of each curve is prescribed by the $L^2$--gradient of $E_\mu$ 
(at any point of the curve and at any time), 
we have some freedom in choosing the tangential velocity. 
In~\cite{garmenzplu} it is discussed in detail that 
\begin{align}\label{Tang}
T^{i}&:=\left\langle 2\frac{\gamma^i_{xxxx}}{\left|\gamma^i_{x}\right|^{4}}
-12\frac{\gamma^i_{xxx}\left\langle \gamma^i_{xx},\gamma^i_{x}\right\rangle }
{\left|\gamma^i_{x}\right|^{6}}
-5\frac{\gamma^i_{xx}\left|\gamma^i_{xx}\right|^{2}}{\left|\gamma^i_{x}\right|^{6}}
\right. \nonumber \\
&\left. -8\frac{\gamma^i_{xx}\left\langle \gamma^i_{xxx},\gamma^i_{x}\right\rangle }
{\left|\gamma^i_{x}\right|^{6}}
+35\frac{\gamma^i_{xx}\left\langle \gamma^i_{xx},\gamma^i_{x}\right\rangle ^{2}}
{\left|\gamma^i_{x}\right|^{8}}
-\mu\frac{\gamma_{xx}^{i}}{\left|\gamma_{x}^{i}\right|^{2}},\tau^{i}\right\rangle \,.
\end{align}
is a meaningful choice. Thus the motion we consider is given by
\begin{align}\label{motion}
\gamma_t^i=-A^i\nu^i-T^i\tau^i
\end{align}
where $\nu^i$ and $\tau^i$ denote the unit normal and tangent vector of $\gamma^i$, respectively.

In fact, the flow is given as (see~\cite{garmenzplu})
\begin{equation*}
\gamma_t^i=-2\frac{\gamma^i_{xxxx}}{\vert\gamma^i_x\vert^4}+l.o.t.\,,
\end{equation*}
where $l.o.t.$ contains terms of order lower than four in $\gamma^i$ and the highest order term $\frac{\gamma^i_{xxxx}}{\vert\gamma^i_x\vert^4}$ ensures that the equation is parabolic as long as $\vert\gamma^i_x\vert$ is uniformly bounded from above and below.
\medskip

The motion equations~\eqref{motion} can be coupled 
with different boundary conditions. In~\cite{garmenzplu}
we prove
(geometric) existence and uniqueness of the motion in several cases of constraints at the boundary provided that the initial datum is
of class $C^{4+\alpha}([0,1])$ and satisfies the respective compatibility conditions. The result also justifies the special choice of the tangential velocity~\eqref{Tang}.

\medskip

In this paper we focus our attention on
networks
evolving according to the elastic flow without restriction on the angles at the triple junction.
Our proof of the long time existence Theorem~\ref{main} relies on 
a short time existence result in a Sobolev setting.

%To discuss the long time behaviour of the evolution we need to prove short time existence of the problem in a Sobolev setting. 
As we are considering a fourth order parabolic problem, the natural solution space is given by the intersection of Bochner spaces
\begin{equation*}
W_p^1\left((0,T);L_p\left((0,1);\mathbb{R}^d\right)\right)\cap L_p\left((0,T);W_p^4\left((0,1);\mathbb{R}^d \right)\right)=:W_p^{1,4}\left((0,T)\times(0,1);\mathbb{R}^d\right)
\end{equation*}
for $T$ positive and  $p\in(1,\infty)$ large enough. We will often simply write $W_p^{1,4}\left((0,T)\times(0,1)\right)$.
The initial values need to be elements of the respective trace space which is given by the Sobolev Slobodeckij space $W_p^{4-\nicefrac{4}{p}}\left((0,1);\mathbb{R}^d \right)$. In the case that $4-\nicefrac{4}{p}$ is not an integer, this space coincides with the Besov space $B_{pp}^{4-\nicefrac{4}{p}}\left((0,1);\mathbb{R}^d\right)$.

\begin{dfnz}
	Given $d\geq 1$, $p\in [1,\infty)$ and $\theta\in(0,1)$ the \textit{Slobodeckij seminorm} of an element $f\in L_p\left((0,1);\mathbb{R}^d\right)$ is defined by
	\begin{equation*}
	\left[f\right]_{\theta,p}:=\left(\int_{0}^{1}\int_{0}^{1}\frac{\left\lvert f(x)-f(y)\right\rvert^p}{\lvert x-y\rvert^{\theta p+1}}\mathrm{d}x\mathrm{d}y\right)^{\nicefrac{1}{p}}\,.
	\end{equation*}
	Let $s\in(0,\infty)$ be a non--integer. The Sobolev--Slobodeckij space $W_p^s\left((0,1);\mathbb{R}^d\right)$ is defined by
	\begin{equation*}
	W_p^s\left((0,1);\mathbb{R}^d\right):=\left\{f\in W_p^{\lfloor s\rfloor}\left((0,1);\mathbb{R}^d\right):\left[f\right]_{s-\lfloor s\rfloor,p}<\infty\right\}\,.
	\end{equation*}
\end{dfnz}
\begin{rem}
	We will restrict to the case $p\in (5,\infty)$. Then by the Sobolev Embedding Theorem 
	\begin{equation*}
	W_p^{4-\nicefrac{4}{p}}\left((0,1);\mathbb{R}^2\right)\hookrightarrow C^{3+\alpha}\left([0,1];\mathbb{R}^2\right)
	\end{equation*}
	for $\alpha\in\left(0,1-\nicefrac{5}{p}\right)$
	(see for instance~\cite[Theorem 4.12]{AdamsFournier}). In particular, the boundary constraints for the initial network in Definition~\ref{admg1} and~\ref{adm} should be understood pointwise.
\end{rem}

For sake of presentation we first describe in details the case of the triod.

%%%%%%%%%%%%%%%%%%%%%%%%%%%%%%%%%%%%%%%%%%%%%%%%%%%
\subsection{Elastic flow of a triod}
%%%%%%%%%%%%%%%%%%%%%%%%%%%%%%%%%%%%%%%%%%%%%%%%%%%

\begin{dfnz}[Elastic flow of a triod]\label{geometricsolution1}
	Let $p\in (5,\infty)$ and $T>0$.
	Let $\mathbb{T}_0$ be a geometrically admissible initial triod with fixed endpoints $P^1$, $P^2$, $P^3$.
	A time dependent family of triods $\left(\mathbb{T}(t)\right)$ is a solution to
	the elastic flow with initial datum $\mathbb{T}_0$ 
	in  $[0,T]$ 
	if and only if there exists a collection of time dependent
	parametrisations
	\begin{equation*}
		\gamma^i_n\in  W^1_p(I_n;L_p((0,1);\mathbb{R}^2))\cap L_p(I_n;W^4_p((0,1);\mathbb{R}^2))\,,
	\end{equation*}
		with $n\in\{0,\dots, N\}$ for some $N\in\mathbb{N}$, $I_n:=(a_n,b_n)\subset \mathbb{R}$, $a_n\leq a_{n+1}$, $b_n\leq b_{n+1}$, $a_n<b_n$
	and $\bigcup_n (a_n,b_n)=(0,T)$ such that for all $n\in\{0,\dots,N\}$ and $t\in I_n$, $\gamma_n(t)=\left(\gamma^1(t),\gamma^2(t),\gamma^3(t)\right)$ is a regular parametrisation of $\mathbb{T}(t)$.
	Moreover each 
	$\gamma_n$
	needs to satisfy the following system
	\begin{equation}\label{Triod0}
	\begin{cases}
	\begin{array}{lll}
	\left\langle \gamma_t^{i}(t,x), \nu^i(t,x)\right\rangle  \nu^i(t,x)
	=-A^{i}(t,x)\nu^{i}(t,x)& &\text{motion,}\\
	\gamma^{1}\left(t,0\right)=\gamma^{2}\left(t,0\right)=\gamma^{3}\left(t,0\right)&
	&\text{concurrency condition,}\\
	k^{i}(t,0)=0 & &\text{curvature condition,}\\
	\sum_{i=1}^{3}\left(2k_{s}^{i}\nu^{i}-\mu\tau^{i}\right)(t,0)=0& &\text{third order condition,}\\
	\gamma^i(t,1)=P^i& &\text{fixed endpoints,}\\
	k^{i}(t,1)=0 & &\text{curvature condition,}
	\end{array}
	\end{cases}
	\end{equation}
	for every  $t\in  I_n,x\in\left(0,1\right)$ and for $i\in\{1,2,3\}$. Finally we ask that $\gamma_n(0,[0,1])=\mathbb{T}_0$ whenever $a_n=0$. 

\end{dfnz}

\begin{dfnz}\label{admg1}
Let $p\in(5,\infty)$. A triod $\mathbb{T}_0$ is a geometrically admissible initial network for 
 system~\eqref{Triod0}  if
\begin{itemize}
	\item[-] there exists a parametrisation $\sigma=(\sigma^1,\sigma^2,\sigma^3)$ 
	of $\mathbb{T}_0$ 
	such that every curve $\sigma^i$ is regular and
	\begin{equation*}
	\sigma^i\in W_p^{4-\nicefrac{4}{p}}((0,1);\mathbb{R}^2)\,.
	\end{equation*}	
\item[-]  the three curves meet in one triple junction
with $k^i_0=0$ ,

$\sum_{i=1}^{3}\left(2k_{0,s}^{i}\nu_0^{i}-\mu\tau^{i}_0\right)=0$ and at least two curves form a strictly positive angle;
\item[-] each curve has zero curvature at its fixed endpoint.
\end{itemize}
\end{dfnz}

\begin{dfnz}\label{smoothsol}
	A time dependent family of triods $\left(\mathbb{T}(t)\right)$ is a \emph{smooth}
	solution to	the elastic flow with initial datum $\mathbb{T}_0$ 
	in  $(0,T]$ 
	if and only if there exists a collection $\gamma_n$ of time dependent
	parametrisations with
 $n\in\{0,\dots, N\}$ for some $N\in\mathbb{N}$, $I_n:=(a_n,b_n)\subset \mathbb{R}$, $a_n\leq a_{n+1}$, $b_n\leq b_{n+1}$, $a_n<b_n$
	and $\bigcup_n (a_n,b_n)=(0,T)$. For all $n\in\{0,\dots,N\}$ and $t\in I_n$, $\gamma_n(t)=\left(\gamma^1(t),\gamma^2(t),\gamma^3(t)\right)$ is a regular parametrisation of $\mathbb{T}(t)$ satisfying~\eqref{Triod0}.
Moreover we require $a_1>0$ and
$$
\gamma^i_n\in C^{\infty}(\overline{I_n}\times[0,1];\mathbb{R}^2)
$$	
for all $n\in\{1,\dots,N\}$ and further for all $\varepsilon\in(0,b_0)$
$$
\gamma^i_0\in  W^1_p((a_0,b_0);L_p((0,1);\mathbb{R}^2))\cap L_p((a_0,b_0);W^4_p((0,1);\mathbb{R}^2))\cap C^\infty\left([\varepsilon,b_0]\times[0,1];\mathbb{R}^2\right)\,.
$$
\end{dfnz}

%%%%%%%%%%%%%%%%%%%%%%%%%%%%%%%%%%%%%%%%%%%%%%%%%%%
%%%%%%%%%%%%%%%%%%%%%%%%%%%%%%%%%%%%%%%%%%%%%%%%%%%

\section{Short time existence of the Analytic Problem in Sobolev spaces}\label{shorttimeexistence}
%%%%%%%%%%%%%%%%%%%%%%%%%%%%%%%%%%%%%%%%%%%%%%%%%%%
%%%%%%%%%%%%%%%%%%%%%%%%%%%%%%%%%%%%%%%%%%%%%%%%%%%

We now pass from 
the geometric problem to a system of PDEs that we call ``Analytic Problem".
The curves are now described in terms of parametrisations.
Moreover the tangential component of the velocity is fixed to be $T$ which is defined in~\eqref{Tang}.

%%%%%%%%%%%%%%%%%%%%%%%%%%%%%%%%%%%%%%%%%%%%%%%%%%%
%%%%%%%%%%%%%%%%%%%%%%%%%%%%%%%%%%%%%%%%%%%%%%%%%%%

\begin{dfnz}[Admissible initial parametrisation]\label{adm}
Let $p\in(5,\infty)$. A parametrisation $\varphi=(\varphi^1,\varphi^2,\varphi^3)$
of an initial triod $\mathbb{T}_0$
is an admissible initial parametrisation for system~\eqref{TriodC^0} if 
\begin{itemize}
\item[-] each curve $\varphi^i$ is regular and $\varphi^i\in W_p^{4-\nicefrac{4}{p}}((0,1);\mathbb{R}^2)$ for $i\in\{1,2,3\}$;
\item[-] the network is non degenerate in the sense that $\mathrm{span}\{\nu^1_0(0),\nu^2_0(0),\nu^3_0(0)\}=\mathbb{R}^2$;
\item[-] the compatibility conditions for system~\eqref{TriodC^0} are fulfilled, namely
the 
parametrisations $\varphi^i$
satisfy the concurrency, second and third order condition at $y=0$ and
the second order condition at $y=1$.
\end{itemize}
\end{dfnz}

\begin{dfnz}\label{solutionanalytic}
Let $T>0$ and $p\in(5,\infty)$. Given an admissible initial parametrisation 
$\varphi$
the time dependent parametrisation $\gamma=(\gamma^1,\gamma^2,\gamma^3)$
is a solution of the Analytic Problem with initial value $\varphi$ in $[0,T]$ if
\begin{equation}
\gamma\in W^1_p\left((0,T);L_p((0,1),(\mathbb{R}^2)^3))\cap L_p((0,T);W^4_p((0,1),(\mathbb{R}^2)^3))\right)
 =:\boldsymbol{E}_T\,,
\end{equation}
the curve $\gamma^i(t)$ is regular for all $t\in[0,T]$ and the following system is satisfied 
for almost every $t\in\left(0,T\right),x\in\left(0,1\right)$ and for $i\in\{1,2,3\}$.
\begin{equation}\label{TriodC^0}
\begin{cases}
\begin{array}{lll}
\gamma_t^{i}(t,x)=-A^{i}(t,x)\nu^{i}(t,x)-T^{i}(t,x)\tau^{i}(t,x)& &\text{motion,}\\
\gamma^{1}\left(t,0\right)=\gamma^{2}\left(t,0\right)=\gamma^{3}\left(t,0\right)&
&\text{concurrency condition,}\\
\gamma^i(t,1)=\varphi^i(1) & &\text{fixed endpoints,}\\
\gamma_{xx}^{i}(t,y)=0 &\text{for}\,y\in\{0,1\} &\text{second order condition,}\\
\sum_{i=1}^{3}\left(2k_{s}^{i}\nu^{i}-\mu\tau^{i}\right)(t,0)=0 &  &
\text{third order condition,}\\
\gamma^i(0,x)=\varphi^i(x)&  &
\text{initial condition}\,.
\end{array}
\end{cases}
\end{equation}
\end{dfnz}

To prove existence and uniqueness of a solution to~\eqref{TriodC^0} in $\boldsymbol{E}_T$ 
for some (possibly small) positive time $T$ we follow the same strategy 
we used in~\cite{garmenzplu}
based on a fixed point argument. Due to the different function space setting 
some arguments need to be revised.

\subsection{Existence of a unique solution to the linearised system}\label{linearisation}
Linearising the highest order term of the motion equation of system~\eqref{TriodC^0} 
around the fixed initial parametrisation $\varphi$ we get
\begin{align}
\gamma^i_t
+\frac{2}{\vert\varphi^i_x\vert^4}\gamma^i_{xxxx}
&=\left(\frac{2}{\vert\varphi^i_x\vert^4} -\frac{2}{\vert\gamma^i_x\vert^4}\right)\gamma^i_{xxxx} 
+\tilde{f}^i(\gamma^i_{xxx},\gamma^i_{xx},\gamma^i_x)=:f^i(\gamma^i_{xxxx},\gamma^i_{xxx},
\gamma^i_{xx},\gamma^i_{x})\,.
\end{align}
The linearised version of  the third order condition
takes the form 
\begin{align}
-\sum_{i=1}^3\frac{1}{\vert \varphi^i_x\vert^3}
\left\langle \gamma^i_{xxx},\nu_0^i\right\rangle \nu_0^i =
-\sum_{i=1}^3\frac{1}{\vert \varphi^i_x\vert^3}
\left\langle \gamma^i_{xxx},\nu_0^i\right\rangle \nu_0^i
+\sum_{i=1}^3\left(\frac{1}{\vert \gamma^i_x\vert^3}
\left\langle \gamma^i_{xxx},\nu^i\right\rangle \nu^i 
+h^i(\gamma^i_x)\right)=:b(\gamma)
\,.
\end{align}

All the other conditions are already linear.
The linearised system associated to~\eqref{TriodC^0} is 
\begin{equation}\label{lyntriod}
\begin{cases}
\begin{array}{lll}
\gamma^i_t(t,x)+\frac{2}{\vert\varphi^i_x\vert^4}\gamma^i_{xxxx}(t,x)&=f^i(t,x)
&\;\text{motion,}\\
\gamma^{1}(t,0)-\gamma^{2}(t,0)&=0 &\;\text{concurrency,}\\
\gamma^{1}(t,0)-\gamma^{3}(t,0)&=0 &\;\text{concurrency,}\\
\gamma^{i}(t,1)&=\varphi^i(1) &\;\text{fixed endpoints,}\\
\gamma^i_{xx}(t,y)&=0 &\;\text{second order,}\\
-\sum_{i=1}^3\frac{1}{\vert \varphi^i_x(0)\vert^3}
\left\langle \gamma^i_{xxx}(t,0),\nu_0^i(0)\right\rangle \nu_0^i(0)&=b(t,0)
&\;\text{third order,}\\
\gamma^{i}(0,x)&=\psi^{i}(x) &\;\text{initial condition}\\
\end{array}
\end{cases}
\end{equation}
for $i\in\{1,2,3\}$, $t\in (0,T)$, $x\in (0,1)$ and $y\in\{0,1\}$.
Here $(f,b,\psi)$ is a general right hand side with $\psi$ satisfying the linear compatibility conditions with respect to $b$.
\begin{rem}\label{identification}
	Integration with respect to the volume element on the $0$--dimensional manifold $\{0\}$ is given by integration with respect to the counting measure, see for example~\cite[page~406]{Lee}. 
Given $p\in[1,\infty)$ we will identify the space 
	\begin{equation*}
	L^p\left(\{0\};\mathbb{R}^2\right):=\left\{f:\{0\}\to\mathbb{R}^2:\int_{\{0\}}^{}\left\lvert f\right\rvert^p\,\mathrm{d}\sigma=\left\lvert f(0)\right\rvert^p <\infty\right\}
	\end{equation*}
	 with $\mathbb{R}^2$ via the isometric isomorphism $I:L_p\left(\{0\};\mathbb{R}^2\right)\to\mathbb{R}^2$, $f\mapsto f(0)$. As this operator restricts to $I:W_p^s\left(\{0\};\mathbb{R}^2\right)\to \mathbb{R}^2$ for every $s>0$, we have an isometric isomorphism
	 \begin{equation*}
	 W_p^{\nicefrac{1}{4}-\nicefrac{1}{4p}}\left((0,T);L_p\left(\{0\};\mathbb{R}^2\right)\right)\cap L_p\left((0,T); W_p^{1-\nicefrac{1}{p}}\left(\{0\};\mathbb{R}^2\right)\right) \simeq W_p^{\nicefrac{1}{4}-\nicefrac{1}{4p}}\left((0,T);\mathbb{R}^2\right)
	 \end{equation*}
	 via the map $f\mapsto \left(t\mapsto f(t,0)\right)$. We will use this identification in the following.
\end{rem}
\begin{dfnz}\label{lincompcond}[Linear compatibility conditions]
A function $\psi\in W^{4-\nicefrac{4}{p}}_p((0,1);(\mathbb{R}^2)^3)$ 
with $p\in (5,\infty)$ satisfies the linear compatibility conditions 
for system~\eqref{lyntriod} with respect to $b\in W_p^{\nicefrac{1}{4}-\nicefrac{1}{4p}}\left((0,T);\mathbb{R}^2\right)$ if for $i,j\in\{1,2,3\}$  it holds
$\psi^i(0)=\psi^j(0)$, $\psi^i_{xx}(0)=0$, $\psi^i(1)=\varphi^i(1)$, $\psi^i_{xx}(1)=0$
and
\begin{equation*}
-\sum_{i=1}^3\frac{1}{\vert \varphi^i_x(0)\vert^3}\left\langle \psi^i_{xxx}(0),\nu_0^i(0) \right\rangle 
\nu_0^i(0) =b(0,0)\,.
\end{equation*}
\end{dfnz}

\begin{teo}\label{linearexistence}
Let $p\in (5,\infty)$.
For every $T>0$ the 
system~\eqref{lyntriod} admits a unique solution $\gamma\in\boldsymbol{E}_T$

provided that
\begin{itemize}
\item[-] $f^i\in L_p((0,T); L_p((0,1);\mathbb{R}^2))$ 
for $i\in\{1,2,3\}\,$;
\item[-]  $b\in W_p^{\nicefrac{1}{4}-\nicefrac{1}{4p}}\left((0,T);\mathbb{R}^2\right)\,$;
\item[-]  $\psi\in W_p^{4-\nicefrac{4}{p}}((0,1);(\mathbb{R}^2)^3)$ and
\item[-]  $\psi$ satisfies the linear compatibility conditions~\ref{lincompcond} 
with respect to $b$.
\end{itemize}
Moreover, for all $T>0$ there exists a constant $C(T)>0$ such that all
solutions satisfy  the inequality
\begin{equation}
\Vert \gamma\Vert_{{\boldsymbol{E}}_T} 
\leq C(T)\left( 
\sum_{i=1}^3 \Vert f^i\Vert_{ L_p((0,T); L_p((0,1);\mathbb{R}^2))}+
\Vert b\Vert_{ W_p^{\nicefrac{1}{4}-\nicefrac{1}{4p}}\left((0,T);\mathbb{R}^2\right)}+
\Vert \psi\Vert_{ W_p^{4-\nicefrac{4}{p}}\left((0,1);\left(\mathbb{R}^2\right)^3\right)}
\right)\,.
\end{equation} 
\end{teo}

\begin{proof}
The theorem follows from~\cite[Theorem 5.4]{solonnikov2} 
as the system~\eqref{lyntriod} is parabolic and the Lopatinskii--Shapiro condition is satisfied.
These two properties are proven in~\cite[Section 3.3.2]{garmenzplu}.
\end{proof}

As a consequence of  Theorem~\ref{linearexistence} we have the following 
\begin{lemma}\label{Lisomorphism}
	Given $p\in(5,\infty)$ and $T$ positive we define
	\begin{align*}
	\mathbb{E}_T:=\{&\gamma\in 
	\boldsymbol{E}_T
	\;\text{such that for}\;i\in\{1,2,3\},   t\in(0,T)\,,y\in\{0,1\}\,\text{it holds}\,\\
	&
	\,\gamma^1(t,0)=\gamma^2(t,0)=\gamma^3(t,0)\,,\gamma^i(t,1)=\varphi^i(1)\,,\gamma^i_{xx}(t,y)=0\}
	\,,\\
	\mathbb{F}_T:=\{&(f,b,\psi)\in 
	L_p((0,T); L_p((0,1);(\mathbb{R}^2)^3))
	\times 
	W_p^{\nicefrac{1}{4}-\nicefrac{1}{4p}}\left((0,T);\mathbb{R}^2\right)
	\times  W_p^{4-\nicefrac{4}{p}}((0,1);(\mathbb{R}^2)^3)\\
	&\,\text{such that the linear compatibility conditions hold}
	\}
	\,.
	\end{align*}
	The map $L_{T}:\mathbb{E}_T\to \mathbb{F}_T$ defined by
	$$
	L_{T}(\gamma):=
	\begin{pmatrix}
	\left(\gamma^i_t+\frac{2}{\vert\varphi^i_x\vert^4}\gamma^i_{xxxx}\right)_{i\in\{1,2,3\}}\\
	-\left(\sum_{i=1}^3\frac{1}{\vert \varphi^i_x\vert^3}
	\left\langle \gamma^i_{xxx},\nu_0^i\right\rangle \nu_0^i\right)_{\vert x=0}\\
	\gamma_{\vert t=0}
	\end{pmatrix}
	$$
	is a continuous isomorphism.
\end{lemma}

\subsection{Uniform in time estimates}

To obtain uniform in time estimates we need to change norms.
\begin{teo}\label{embeddingBUC}
	Let $T$ be positive, $p\in(5,\infty)$ and $\alpha\in\left(0,1-\nicefrac{5}{p}\right]$. We have continuous embeddings
	\begin{equation*}
	W_p^{1,4}\left((0,T)\times(0,1)\right)\hookrightarrow BUC\left([0,T];W_p^{4-\nicefrac{4}{p}}\left((0,1)\right)\right)\hookrightarrow BUC\left([0,T];C^{3+\alpha}\left([0,1]\right)\right)\,.
	\end{equation*}
\end{teo}
\begin{proof}
The first embedding follows from~\cite[Chapter III, Theorem 4.10.2]{Amann}, the second is an immediate consequence of the Sobolev Embedding Theorem, see for example~\cite[Theorem 4.12]{AdamsFournier}.
\end{proof}
\begin{prop}\label{hölderspacetime}
	Let $T$ be positive, $p\in(5,\infty)$ and $\theta\in\left(\frac{1+\nicefrac{1}{p}}{4-\nicefrac{4}{p}},1\right)$. Then
	\begin{equation*}
	W_p^{1,4}\left((0,T)\times(0,1)\right)\hookrightarrow C^{(1-\theta)\left(1-\nicefrac{1}{p}\right)}\left([0,T];C^{1}\left([0,1]\right)\right)
	\end{equation*}
	with continuous embedding.
\end{prop}
\begin{proof}
By~\cite[Corollary~26]{Simon}, $W_p^1\left((0,T);L_p\left((0,1)\right)\right)\hookrightarrow C^{1-\nicefrac{1}{p}}\left([0,T];L_p\left((0,1)\right)\right)$. A direct calculation shows that for all Banach spaces $X_0$, $X_1$ and $Y$ such that $X_0\cap X_1\subset Y$ and $\left\lVert y\right\rVert_Y\leq C\left\lVert y\right\rVert_{X_0}^{1-\theta}\left\lVert y\right\rVert^{\theta}_{X_1}$
for all $y\in X_0\cap X_1$, one has the continuous embedding
\begin{equation*}
BUC\left([0,T];X_1\right)\cap C^{\alpha}\left([0,T];X_0\right)\hookrightarrow C^{(1-\theta)\alpha}\left([0,T];Y\right)\,.
\end{equation*}
For all $\theta\in(0,1)$ the real interpolation method gives 
\begin{equation*}
W_p^{\theta\left(4-\nicefrac{4}{p}\right)}\left((0,1)\right)=\left(L_p\left((0,1)\right);W_p^{4-\nicefrac{4}{p}}\left((0,1)\right)\right)_{\theta,p}
\end{equation*}
and hence using Theorem~\ref{embeddingBUC} the arguments above imply for all $\theta\in(0,1)$,
\begin{equation*}
		W_p^{1,4}\left((0,T)\times(0,1)\right)\hookrightarrow C^{(1-\theta)\left(1-\nicefrac{1}{p}\right)}\left([0,T];W_p^{\theta\left(4-\nicefrac{4}{p}\right)}\left((0,1)\right)\right)\,.
\end{equation*}
The assertion now follows from the Sobolev Embedding Theorem.
\end{proof}
\begin{cor}\label{equivalent norm $E_T$}
	Let $p\in(5,\infty)$. For every $T>0$,
	\begin{equation*}
	\vertiii{g}_{W_p^{1,4}\left((0,T)\times(0,1)\right)}:=\left\lVert g\right\rVert_{W_p^{1,4}\left((0,T)\times(0,1)\right)}+\left\lVert g(0)\right\rVert_{W_p^{4-\nicefrac{4}{p}}\left((0,1)\right)}
	\end{equation*}
	defines a norm on $W_p^{1,4}\left((0,T)\times(0,1)\right)$ that is equivalent to the usual one.
\end{cor}
\begin{lemma}\label{extensionE_T}
	Let $T$ be positive and $p\in(5,\infty)$. There exists a linear operator
	\begin{equation*}
	\boldsymbol{E}: W_p^{1,4}\left((0,T)\times(0,1)\right)\to W_p^{1,4}\left((0,\infty)\times(0,1)\right)
	\end{equation*}
	such that for all $g\in W_p^{1,4}\left((0,T)\times(0,1)\right)$, $\left(\boldsymbol{E}g\right)_{|(0,T)}=g$ and
	\begin{equation*}
	\left\lVert\boldsymbol{E}g\right\rVert_{W_p^{1,4}\left((0,\infty)\times(0,1)\right)}\leq C_p\left(\left\lVert  g\right\rVert_{W_p^{1,4}\left((0,T)\times(0,1)\right)}+\left\lVert  g(0)\right\rVert_{W_p^{4-\nicefrac{4}{p}}(0,1)}\right)=C_p\vertiii{g}_{W_p^{1,4}\left((0,T)\times(0,1)\right)}
	\end{equation*}
	with a constant $C_p$ depending only on $p$.
\end{lemma}
\begin{proof}
	In the case that $g(0)=0$, the function $g$ can be extended to $(0,\infty)$ by reflecting it with respect to the axis $t=T$. The general statement can be deduced from this case by solving a linear parabolic equation of fourth order and using results on maximal regularity as given in~\cite[Proposition 3.4.3]{Prusssimonett}.
\end{proof}
Applying $\boldsymbol{E}$ to every component we obtain an extension operator on $W_p^{1,4}\left((0,T)\times(0,1);\mathbb{R}^d\right)$ for $d\geq 1$. The following Lemma is an immediate consequence of the Sobolev Embedding 
Theorem~\cite[Theorem 4.12]{AdamsFournier}.
\begin{lemma}
	Let $p\in(5,\infty)$. For every positive $T$,
	\begin{equation*}
	\vertiii{b}_{W_p^{\nicefrac{1}{4}-\nicefrac{1}{4p}}\left((0,T);\mathbb{R}\right)}:=\left\lVert b\right\rVert_{W_p^{\nicefrac{1}{4}-\nicefrac{1}{4p}}\left((0,T);\mathbb{R}\right)}+\vert b(0)\vert
	\end{equation*}
	defines a norm on $W_p^{\nicefrac{1}{4}-\nicefrac{1}{4p}}\left((0,T);\mathbb{R}\right)$ that is equivalent to the usual one.
\end{lemma}
\begin{lemma}\label{extensionboundarydata}
	Let $T$ be positive and $p\in(5,\infty)$. There exists a linear operator
	\begin{equation*}
	E:W_p^{\nicefrac{1}{4}-\nicefrac{1}{4p}}\left((0,T);\mathbb{R}\right)\to W_p^{\nicefrac{1}{4}-\nicefrac{1}{4p}}\left((0,\infty);\mathbb{R}\right)
	\end{equation*}
	such that for all $b\in W_p^{\nicefrac{1}{4}-\nicefrac{1}{4p}}\left((0,T);\mathbb{R}\right)$, $\left(Eb\right)_{|(0,T)}=b$ and
	\begin{equation*}
	\left\lVert Eb\right\rVert_{W_p^{\nicefrac{1}{4}-\nicefrac{1}{4p}}\left((0,\infty);\mathbb{R}\right)}\leq C_p\left(\left\lVert b\right\rVert_{W_p^{\nicefrac{1}{4}-\nicefrac{1}{4p}}\left((0,T);\mathbb{R}\right)}+\lvert b(0)\rvert\right)=C_p	\vertiii{b}_{W_p^{\nicefrac{1}{4}-\nicefrac{1}{4p}}\left((0,T);\mathbb{R}\right)}
	\end{equation*}
	with a constant $C_p$ depending only on $p$.
\end{lemma}
\begin{proof}
In the case $b(0)=0$ the operator obtained by reflecting the function with respect to the axis $t=T$ has the desired properties. The general statement can be deduced from this case using surjectivity of the temporal trace
$
_{|t=0}:W_p^{\nicefrac{1}{4}-\nicefrac{1}{4p}}\left((0,\infty);\mathbb{R}\right)\to\mathbb{R}\,.
$
\end{proof}
As a consequence for every positive $T$ the spaces $\mathbb{E}_T$ and $\mathbb{F}_T$ endowed with the norms
\begin{equation*}
\vertiii{\gamma}_{\mathbb{E}_T}:=\vertiii{\gamma}_{W_p^{1,4}\left((0,T)\times(0,1);(\mathbb{R}^2)^3\right)}=\left\lVert\gamma\right\rVert_{W_p^{1,4}\left((0,T)\times(0,1);(\mathbb{R}^2)^3\right)}+\left\lVert\gamma(0)\right\rVert_{W_p^{4-\nicefrac{4}{p}}\left((0,1);(\mathbb{R}^2)^3\right)}
\end{equation*}
and
\begin{equation*}
\vertiii{(f,b,\psi)}_{\mathbb{F}_T}:=\left\lVert f\right\rVert_{L_p\left((0,T);L_p((0,1);(\mathbb{R}^2)^3)\right)}+\vertiii{b}_{W_p^{\nicefrac{1}{4}-\nicefrac{1}{4p}}\left((0,T);\mathbb{R}^2\right)}+\left\lVert\psi\right\rVert_{W_p^{4-\nicefrac{4}{p}}\left((0,1);(\mathbb{R}^2)^3\right)}\,,
\end{equation*}
respectively, are Banach spaces. Given a linear operator $A:\mathbb{F}_T\to\mathbb{E}_T$ we let
\begin{equation*}
\vertiii{A}_{\mathcal{L}\left(\mathbb{F}_T,\mathbb{E}_T\right)}:=\sup\{\vertiii{A(f,b,\psi)}_{\mathbb{E}_T}:(f,b,\psi)\in\mathbb{F}_T,\vertiii{(f,b,\psi)}_{\mathbb{F}_T}\leq 1\}\,.
\end{equation*}
\begin{lemma}\label{Luniformlybounded}
Let $p\in (5,\infty)$. For all $T_0>0$ there exists a constant $c(T_0,p)$ such that
\begin{equation*}
\sup_{T\in (0,T_0]}\vertiii{L_T^{-1}}_{\mathcal{L}(\mathbb{F}_T,\mathbb{E}_T)}\leq c(T_0,p)\,.
\end{equation*}
\end{lemma}
\begin{proof}
Let $T\in (0,T_0]$ be arbitrary, $\left(f,b,\psi\right)\in\mathbb{F}_T$ and $E_{T_0}b:=\left(Eb\right)_{|(0,T_0)}$ where $E$ is the extension operator defined in Lemma~\ref{extensionboundarydata}. Extending $f$ by $0$ to $E_{T_0}f\in L_p\left((0,T_0);L_p\left((0,1)\right)\right)$ we observe that $\left(E_{T_0}f,E_{T_0}b,\psi\right)$ lies in $\mathbb{F}_{T_0}$. As $L_T$ and $L_{T_0}$ are isomorphisms, there exist unique $\gamma\in\mathbb{E}_T$ and $\tilde{\gamma}\in\mathbb{E}_{T_0}$ such that $L_T\gamma=(f,b,\psi)$ and $L_{T_0}\widetilde{\gamma}=\left(E_{T_0}f,E_{T_0}b,\psi\right)$ satisfying
\begin{equation*}
L_T\gamma=(f,b,\psi)=\left(E_{T_0}f,E_{T_0}b,\psi\right)_{|(0,T)}=\left(L_{T_0}\widetilde{\gamma}\right)_{|(0,T)}=L_T\left(\widetilde{\gamma}_{|(0,T)}\right)
\end{equation*}
and thus $\gamma=\widetilde{\gamma}_{|(0,T)}$. Using Lemma~\ref{Lisomorphism} and the equivalence of norms on $\mathbb{E}_{T_0}$ this implies
\begin{align*}
\vertiii{L_T^{-1}\left(f,b,\psi\right)}_{\mathbb{E}_T}&=\vertiii{\left(L_{T_0}^{-1}\left(E_{T_0}f,E_{T_0}b,\psi\right)\right)_{|(0,T)}}_{\mathbb{E}_T}\leq \vertiii{L_{T_0}^{-1}\left(E_{T_0}f,E_{T_0}b,\psi\right)}_{\mathbb{E}_{T_0}}\\
&\leq c\left(T_0,p\right)\left\lVert L_{T_0}^{-1}\left(E_{T_0}f,E_{T_0}b,\psi\right)\right\rVert_{\mathbb{E}_{T_0}}\leq c\left(T_0,p\right)\left\lVert\left(E_{T_0}f,E_{T_0}b,\psi\right)\right\rVert_{\mathbb{F}_{T_0}}\\
&\leq c\left(T_0,p\right)\left(\left\lVert f\right\rVert_{L_p\left((0,T);L_p\left((0,1)\right)\right)}+\vertiii{b}_{W_p^{\nicefrac{1}{4}-\nicefrac{1}{4p}}\left((0,T);\mathbb{R}\right)}+\left\lVert\psi\right\rVert_{W_p^{4-\nicefrac{4}{p}}\left((0,1);\mathbb{R}\right)}\right)\,.
\end{align*}
\end{proof}
\begin{lemma}\label{embeddingBUCimeindependent}
	Let $p\in(5,\infty)$ and $T_0$ be positive. There exist constants $C(p)$ and $C\left(T_0,p\right)$ such that for all $T\in (0,T_0]$ and all $g\in W_p^{1,4}\left((0,T)\times(0,1)\right)$,
	\begin{equation*}
	\left\lVert g\right\rVert_{BUC\left([0,T];C^3\left([0,1]\right)\right)}\leq C(p)\left\lVert g\right\rVert_{BUC\left([0,T];W_p^{4-\nicefrac{4}{p}}\left((0,1)\right)\right)}\leq C\left(T_0,p\right)\vertiii{g}_{W_p^{1,4}\left((0,T)\times(0,1)\right)}\,.
	\end{equation*}
\end{lemma}
\begin{proof}
 Let $T\in (0,T_0]$ be arbitrary, $g\in W_p^{1,4}\left((0,T)\times(0,1)\right)$ and $\boldsymbol{E}g$ the extension according to Lemma~\ref{extensionE_T}. Then $\left(\boldsymbol{E}g\right)_{|\left(0,T_0\right)}$ lies in $W_p^{1,4}\left(\left(0,T_0\right)\times(0,1)\right)$ and Theorem~\ref{embeddingBUC} implies
 \begin{align*}
 \left\lVert g\right\rVert_{BUC\left([0,T];W_p^{4-\nicefrac{4}{p}}\left((0,1)\right)\right)}&\leq \left\lVert \left(\boldsymbol{E}g\right)_{|\left(0,T_0\right)}\right\rVert_{BUC\left(\left[0,T_0\right];W_p^{4-\nicefrac{4}{p}}\left((0,1)\right)\right)}\\
 &\leq C\left(T_0,p\right)\left\lVert\left(\boldsymbol{E}g\right)_{|\left(0,T_0\right)}\right\rVert_{W_p^{1,4}\left(\left(0,T_0\right)\times(0,1)\right)}\\
 &\leq C(T_0,p)\left\lVert\boldsymbol{E}g\right\rVert_{W_p^{1,4}\left((0,\infty)\times(0,1)\right)}\leq C\left(T_0,p\right)\vertiii{g}_{W_p^{1,4}\left((0,T)\times(0,1)\right)}\,.
 \end{align*}
 The first inequality follows from the Sobolev Embedding Theorem.
\end{proof}
\begin{lemma}\label{uniformcalphac1}
	Let $p\in(5,\infty)$, $\theta\in\left(\frac{1+\nicefrac{1}{p}}{4-\nicefrac{4}{p}},1\right)$ and $T_0$ be positive. There exists a constant $C\left(T_0,p,\theta\right)$ such that for all $T\in(0,T_0]$ and all $g\in W_p^{1,4}\left((0,T)\times(0,1)\right)$,
	\begin{equation*}
	\left\lVert g\right\rVert_{C^{(1-\theta)\left(1-\nicefrac{1}{p}\right)}\left([0,T];C^1\left([0,1]\right)\right)}\leq C\left(T_0,p,\theta\right)\vertiii{g}_{W_p^{1,4}\left((0,T)\times(0,1)\right)}\,.
	\end{equation*}
\end{lemma}
\begin{proof}
	Similarly to the previous proof it holds for $T\in (0,T_0]$ and $g\in W_p^{1,4}\left((0,T)\times(0,1)\right)$,
	\begin{align*}
		\left\lVert g\right\rVert_{C^{(1-\theta)\left(1-\nicefrac{1}{p}\right)}\left([0,T];C^1\left([0,1]\right)\right)}&\leq \left\lVert \left(\boldsymbol{E}g\right)_{|\left(0,T_0\right)}\right\rVert_{C^{(1-\theta)\left(1-\nicefrac{1}{p}\right)}\left(\left[0,T_0\right];C^1\left([0,1]\right)\right)}\\
		&\leq C\left(T_0,p,\theta\right)\left\lVert\left(\boldsymbol{E}g\right)_{|\left(0,T_0\right)}\right\rVert_{W_p^{1,4}\left(\left(0,T_0\right)\times(0,1)\right)}\\
		&\leq C\left(T_0,p,\theta\right)\vertiii{g}_{W_p^{1,4}\left((0,T)\times(0,1)\right)}\,.
	\end{align*}
\end{proof}

For $T$ positive and $p\in(5,\infty)$ we consider the complete metric spaces
\begin{align*}
\mathbb{E}^\varphi_T:=&\left\{\gamma\in 
\mathbb{E}_T\,\text{such that }\,\gamma_{\vert t=0}=\varphi\right\}\,,\\
\mathbb{F}^\varphi_T:=&\left\{(f,b)\,\text{such that }\, (f,b,\varphi)\in\mathbb{F}_{T} \right\}
\times\left\{\varphi\right\}
\,.
\end{align*}
Given a positive time $T$ and a radius $M$ we let 
\begin{equation*}
\overline{B_M}:=\left\{\gamma\in\boldsymbol{E}_T:\vertiii{\gamma}_{\mathbb{E}_T}\leq M\right\}\,.
\end{equation*}
\begin{lemma}\label{curveregular}
	Let $p\in(5,\infty)$ and $T_0$, $M$ be positive and
	\begin{equation*}
	\boldsymbol{c}:=\frac{1}{2}\min_{i\in\{1,2,3\}} \inf_{x\in[0,1]}\vert\varphi^i_x(x)\vert >0\,.
	\end{equation*} 
	There exists a time $\widetilde{T}\left(\boldsymbol{c},M\right)\in (0,T_0]$ 
	such that for all $\gamma\in \mathbb{E}_{T}^\varphi\cap\overline{B_M}$ with $T\in(0,\widetilde{T}\left(\boldsymbol{c},M\right)]$ and all $i\in\{1,2,3\}$ we have
	\begin{equation*}
	\inf_{t\in[0,T],x\in[0,1]}\left\lvert \gamma^i_x(t,x)\right\rvert \geq \boldsymbol{c}\,.
	\end{equation*}
	In particular the curves $\gamma^i(t)$ are regular for all $t\in [0,T]$. 
	Moreover, given a polynomial $\mathfrak{p}$ in $\vert\gamma^i_x\vert^{-1}$ 
	there exists a constant $C^i>0$ depending on $\boldsymbol{c}$ such that
	\begin{equation*}
	\sup_{t\in[0,T],x\in[0,1]}\left\lvert\mathfrak{p}\left(
	\frac{1}{\vert\gamma^i_x\vert}\right)\right\rvert\leq C^i\left(\boldsymbol{c}\right)\,.
	\end{equation*}
\end{lemma}
\begin{proof}
Let $\theta\in\left(\frac{1+\nicefrac{1}{p}}{4-\nicefrac{4}{p}},1\right)$ be fixed and $\alpha:=\left(1-\theta\right)\left(1-\nicefrac{1}{p}\right)$. Given $T\in (0,T_0]$ and $\gamma\in\mathbb{E}_T^\varphi\cap\overline{B_M}$ we have for $i\in\{1,2,3\}$, $t\in[0,T]$, $x\in[0,1]$,
\begin{equation*}
\vert\gamma^i_x(t,x)\vert\geq \vert\varphi^i_x(x)\vert-\vert\gamma^i_x(t,x)-\gamma^i_x(0,x)\vert\geq 2\boldsymbol{c}-\left\lVert\gamma^i(t)-\gamma^i(0)\right\rVert_{C^1\left([0,1]\right)}\,.
\end{equation*}
Lemma~\ref{uniformcalphac1} implies for all $t\in[0,T]$,
\begin{equation*}
\left\lVert\gamma^i(t)-\gamma^i(0)\right\rVert_{C^1\left([0,1]\right)}\leq t^\alpha\left\lVert\gamma^i\right\rVert_{C^\alpha\left([0,T];C^1\left([0,1]\right)\right)}\leq T^\alpha C\left(T_0,p\right)\vertiii{\gamma}_{\mathbb{E}_T}\leq T^\alpha M C\left(T_0,p\right)\,.
\end{equation*}
The claim follows considering $\widetilde{T}$ so small that $\widetilde{T}^\alpha M C\left(T_0,p\right)\leq \boldsymbol{c}$.
\end{proof}
\subsection{Contraction estimates}
In the following we let $T_0\equiv 1$ and $\widetilde{T}=\widetilde{T}(\boldsymbol{c},M)\in(0,1]$ be as in Lemma~\ref{curveregular}. Given $T\in(0,\widetilde{T}\left(\boldsymbol{c},M\right)]$ and $p\in(5,\infty)$ we consider the mappings
\begin{align*}
N_{T,1}:&
\begin{cases}
\mathbb{E}^\varphi_T\cap\overline{B_M} &\to L_p((0,T); L_p((0,1);(\mathbb{R}^2)^3))\,,\\
\gamma&\mapsto
f(\gamma):=(f^i(\gamma^i))_{i\in\{1,2,3\}}\,,
\end{cases}\\
N_{T,2}:&
\begin{cases}
\mathbb{E}^\varphi_T\cap\overline{B_M}&\to  W_p^{\nicefrac{1}{4}-\nicefrac{1}{4p}}\left((0,T);\mathbb{R}^2\right)\,, \\
\gamma&\mapsto b(\gamma)\,,
\end{cases}
\end{align*}
where the functions $f^i(\gamma^i):=f^i(\gamma^i_{xxxx},\gamma^i_{xxx},\gamma^i_{xx},\gamma^i_x)$ and $b(\gamma):=b(\gamma_{xxx},\gamma_x)$ 
are defined in \S~\ref{linearisation}.
\begin{prop}\label{N_1 contractive}
Let $p\in(5,\infty)$ and $M$ be positive. For every $T\in (0,\widetilde{T}\left(\boldsymbol{c},M\right)]$ the map $N_{T,1}$ is well defined and there exists a constant $\sigma\in (0,1)$ and a constant $C$ depending on $\boldsymbol{c}$ and $M$ such that for all $\gamma$, $\widetilde{\gamma}\in\mathbb{E}_T^\varphi\cap\overline{B_M}$,
\begin{equation*}
\left\lVert N_{T,1}(\gamma)-N_{T,1}\left(\widetilde{\gamma}\right)\right\rVert_{L_p\left((0,T);L_p\left((0,1);\left(\mathbb{R}^2\right)^3\right)\right)}\leq C\left(\boldsymbol{c},M\right)T^\sigma \vertiii{\gamma-\widetilde{\gamma}}_{\mathbb{E}_T}\,.
\end{equation*}
\end{prop}
\begin{proof}
Given $\gamma\in\mathbb{E}_T^\varphi\cap \overline{B_M}$ every component $f^i\left(\gamma^i\right)$ is of the form
\begin{equation*}
f^i\left(\gamma^i\right)=\left(\frac{2}{\vert\varphi^i_x\vert^4}-\frac{2}{\vert\gamma^i_x\vert^4}\right)\gamma^i_{xxxx}+\mathfrak{p}\left(\frac{1}{\vert\gamma^i_x\vert},\gamma^i_x,\gamma^i_{xx},\gamma^i_{xxx}\right)\,,
\end{equation*}
where $\mathfrak{p}$ is a polynomial. The precise formula is given in~\cite[Lemma 3.23]{garmenzplu}.
Denoting by $\mathfrak{p}$ any polynomial that may change from line to line the Lemmata~\ref{embeddingBUCimeindependent} (with $T_0\equiv1$) and~\ref{curveregular} imply for $\gamma\in\mathbb{E}_T^\varphi\cap\overline{B_M}$ with constants $C>0$ depending on $\boldsymbol{c}$ and some $k>0$,
\begin{align*}
&\left\lVert \mathfrak{p}\left(\frac{1}{\vert\gamma^i_x\vert},\gamma^i_x,\gamma^i_{xx},\gamma^i_{xxx}\right)\right\rVert_{L_p\left((0,T);L_p\left((0,1)\right)\right)}^p
=\int_0^T\int_0^1\left\lvert\mathfrak{p}\left(\frac{1}{\vert\gamma^i_x\vert},\gamma^i_x,\gamma^i_{xx},\gamma^i_{xxx}\right)(t,x)\right\rvert^p\,\mathrm{d}x\,\mathrm{d}t\\
&\leq T C(\boldsymbol{c})\left(\left\lVert\gamma^i\right\rVert_{BUC\left([0,T];C^3\left([0,1];\mathbb{R}^2\right)\right)}^k+1\right)
\leq TC\left(\boldsymbol{c}\right)\left(\vertiii{\gamma^i}_{W_p^{1,4}\left((0,T)\times(0,1);\mathbb{R}^2\right)}^k+1\right)\\
&\leq TC\left(\boldsymbol{c}\right)\left(\vertiii{\gamma}_{\mathbb{E}_T}^k+1\right)\leq TC\left(\boldsymbol{c}\right)\left(M^k+1\right)\,.
\end{align*}
And similarly,
\begin{align*}
&\left\lVert\left(\frac{2}{\vert\varphi^i_x\vert^4}-\frac{2}{\vert\gamma^i_x\vert^4}\right)\gamma^i_{xxxx}\right\rVert_{L_p\left((0,T);L_p\left((0,1);\mathbb{R}^2\right)\right)}^p
=\int_0^T\int_0^1 \left\lvert \frac{2}{\vert\varphi^i_x\vert^4}-\frac{2}{\vert\gamma^i_x\vert^4}\right\rvert^p\left\lvert \gamma^i_{xxxx}\right\rvert^p\,\mathrm{d}x\,\mathrm{d}t\\
&\leq C\left(\sup _{x\in[0,1]}\frac{1}{\left\lvert \varphi^i_x\right\rvert^{4p}}+\sup_{t\in[0,T],x\in[0,1]}\frac{1}{\left\lvert\gamma^i_x\right\rvert^{4p}}\right)\int_0^T\int_0^1\left\vert\gamma^i_{xxxx}\right\rvert^p\,\mathrm{d}x\,\mathrm{d}t\\
&\leq C\left(\boldsymbol{c}\right)\left\lVert \gamma^i_{xxxx}\right\rVert_{L_p\left((0,T);L_p\left((0,1);\mathbb{R}^2\right)\right)}^p\leq C\left(\boldsymbol{c}\right)\vertiiii{\gamma^i}_{W_p^{1,4}\left((0,T)\times(0,1);\mathbb{R}^2\right)}^p\leq C\left(\boldsymbol{c}\right)M^p\,.
\end{align*}
This shows that $N_{T,1}$ is well defined.

To prove that $N_{T,1}$ is Lipschitz continuous we let $\gamma$ and $\widetilde{\gamma}$ in $\mathbb{E}_T^\varphi\cap\overline{B_M}$ be fixed. All detailed formulas are given in~\cite[Proposition 3.28]{garmenzplu}. The highest order term is of the form
\begin{equation*}
\left(\left\lvert\varphi^i_x\right\rvert-\left\lvert\gamma^i_x\right\rvert\right)\left(\gamma^i_{xxxx}-\widetilde{\gamma}^i_{xxxx}\right) 
\mathfrak{p}\left(\frac{1}{\lvert\varphi^i_x\rvert},\frac{1}{\lvert\gamma^i_x\rvert}\right)
+\left(\left\lvert\widetilde{\gamma}^i_x\right\rvert-\left\lvert\gamma^i_x\right\rvert\right)\mathfrak{p}\left(\frac{1}{\lvert\widetilde{\gamma}^i_x\rvert},\frac{1}{\lvert\gamma^i_x\rvert}\right)\widetilde{\gamma}^i_{xxxx}
\end{equation*}
and can be estimated as follows using Lemma~\ref{curveregular} and Lemma~\ref{uniformcalphac1} with some fixed coefficient $\theta\in\left(\frac{1+\nicefrac{1}{p}}{4-\nicefrac{4}{p}},1\right)$ and $\alpha:=\left(1-\theta\right)\left(1-\nicefrac{1}{p}\right)$:
\begin{align*}
&\left\lVert\left(\left\lvert\varphi^i_x\right\rvert-\left\lvert\gamma^i_x\right\rvert\right)\left(\gamma^i_{xxxx}-\widetilde{\gamma}^i_{xxxx}\right) 
\mathfrak{p}\left(\frac{1}{\lvert\varphi^i_x\rvert},\frac{1}{\lvert\gamma^i_x\rvert}\right)\right\rVert_{L_p\left((0,T);L_p\left((0,1);\mathbb{R}^2\right)\right)}\\
&\leq C(\boldsymbol{c}) \sup_{t\in[0,T],x\in[0,1]}\left\lvert\varphi^i_x(x)-\gamma^i_x(t,x)\right\rvert \left\lVert\gamma^i_{xxxx}-\widetilde{\gamma}^i_{xxxx}\right\rVert_{L_p\left((0,T);L_p\left((0,1);\mathbb{R}^2\right)\right)}\\
&\leq C(\boldsymbol{c}) T^\alpha \sup_{t\in[0,T],x\in[0,1]}t^{-\alpha}\left\lvert\gamma^i_x(t,x)-\gamma^i_x(0,x)\right\rvert\vertiii{\gamma^i-\widetilde{\gamma}^i}_{W_p^{1,4}\left((0,T)\times(0,1);\mathbb{R}^2\right)}\\
&\leq C(\boldsymbol{c})T^\alpha \sup_{t\in[0,T]}t^{-\alpha}\left\lVert\gamma^i(t)-\gamma^i(0)\right\rVert_{C^1([0,1];\mathbb{R}^2)}\vertiii{\gamma-\widetilde{\gamma}}_{\mathbb{E}_T}\\
&\leq C(\boldsymbol{c})T^{\alpha}\left\lVert \gamma^i\right\rVert_{C^{\alpha}\left([0,T];C^1\left([0,1];\mathbb{R}^2\right)\right)}\vertiii{\gamma-\widetilde{\gamma}}_{\mathbb{E}_T}
\leq C(\boldsymbol{c})T^\alpha \vertiii{\gamma^i}_{W_p^{1,4}\left((0,T)\times(0,1);\mathbb{R}^2\right)}\vertiii{ \gamma-\widetilde{\gamma}}_{\mathbb{E}_T}\\
&\leq C(\boldsymbol{c})MT^\alpha\vertiii{\gamma-\widetilde{\gamma}}_{\mathbb{E}_T}\,.
\end{align*}
Similarly,
\begin{align*}
&\left\lVert\left(\left\lvert\widetilde{\gamma}^i_x\right\rvert-\left\lvert\gamma^i_x\right\rvert\right)
\mathfrak{p}\left(\frac{1}{\lvert\widetilde{\gamma}^i_x\rvert},\frac{1}{\lvert\gamma^i_x\rvert}\right)\widetilde{\gamma}^i_{xxxx}\right\rVert_{L_p((0,T);L_p((0,1);\mathbb{R}^2))}\leq C\left(\boldsymbol{c}\right)\sup_{t\in[0,T],x\in[0,1]}\left\lvert\gamma^i_x-\widetilde{\gamma}^i_x\right\rvert\vertiii{\widetilde{\gamma}}_{\mathbb{E}_T}\\
&\leq C\left(\boldsymbol{c}\right)M\sup_{t\in[0,T],x\in[0,1]}\left\lvert\left(\gamma^i_x(t,x)-\widetilde{\gamma}^i_x(t,x)\right)-\left(\gamma^i_x(0,x)-\widetilde{\gamma}^i_x(0,x)\right)\right\rvert\\
&\leq C\left(\boldsymbol{c}\right)MT^\alpha\sup_{t\in[0,T]}t^{-\alpha}\left\lVert\left(\gamma^i(t)-\widetilde{\gamma}^i(t)\right)-\left(\gamma^i(0)-\widetilde{\gamma}^i(0)\right)\right\rVert_{C^1\left([0,1];\mathbb{R}^2\right)}\\
&\leq C\left(\boldsymbol{c}\right)MT^\alpha\left\lVert\gamma^i-\widetilde{\gamma}^i\right\rVert_{C^\alpha\left([0,T];C^1\left([0,1];\mathbb{R}^2\right)\right)}\leq C\left(\boldsymbol{c}\right)MT^\alpha\vertiii{\gamma-\widetilde{\gamma}}_{\mathbb{E}_T}\,.
\end{align*}

By the proof of~\cite[Proposition 3.28]{garmenzplu} all the other terms of $f^i\left(\gamma^i\right)-f^i\left(\widetilde{\gamma}^i\right)$ are of the form
\begin{equation*}
\sum_{j=1}^2\mathfrak{p}_j\left(\gamma^i_x,\gamma^i_{xx},\gamma^i_{xxx},\left\lvert \gamma^i_x\right\rvert^{-1},\widetilde{\gamma}^i_x,\widetilde{\gamma}^i_{xx},\widetilde{\gamma}^i_{xxx},\left\lvert \widetilde{\gamma}^i_x\right\rvert^{-1}\right)\left(a-\tilde{a}\right)^j
\end{equation*}
where $\mathfrak{p}_j$ is a polynomial and $\left(a-\tilde{a}\right)^j$ is either equal to the $j$--th component of $\partial^k_x\gamma^i-\partial^k_x\widetilde{\gamma}^i$ for $k\in\{1,2,3\}$ or equal to $\left\lvert\gamma^i_x\right\rvert^{-l}-\left\lvert\widetilde{\gamma}^i_x\right\rvert^{-l}$ with a natural number $l\geq 1$. Using~\ref{embeddingBUCimeindependent} with $T\equiv 1$ the polynomial can be estimated in supremum norm with respect to time and space by
\begin{equation*}
C\left(\boldsymbol{c}\right)\left\lVert\gamma^i\right\rVert_{BUC\left([0,T];C^3\left([0,1];\mathbb{R}^2\right)\right)}^m\left\lVert\widetilde{\gamma}^i\right\rVert_{BUC\left([0,T];C^3\left([0,1];\mathbb{R}^2\right)\right)}^{\tilde{m}}\leq C\left(\boldsymbol{c}\right)M^{m+\tilde{m}}
\end{equation*}
for some natural numbers $m$ and $\tilde{m}\in\mathbb{N}$. Moreover for $k\in\{1,2,3\}$,
\begin{equation*}
\sup_{t\in[0,T], x\in[0,1]}\left\lvert \partial_x^k\gamma^i(t,x)-\partial_x^{k}\widetilde{\gamma}^i(t,x)\right\rvert\leq \left\lVert\gamma^i-\widetilde{\gamma}^i\right\rVert_{BUC\left([0,T];C^3\left([0,1];\mathbb{R}^2\right)\right)}\leq C\vertiii{\gamma-\widetilde{\gamma}}_{\mathbb{E}_T}
\end{equation*}
and by an identity given in the proof of~\cite[Proposition 3.28]{garmenzplu},
\begin{align*}
&\sup_{t\in[0,T],x\in[0,1]}\left(\left\lvert\gamma^i_x\right\rvert^{-l}-\left\lvert\widetilde{\gamma}^i_x\right\rvert^{-l}\right)=\sup_{t\in[0,T],x\in[0,1]}\left(\left\lvert\gamma^i_x\right\rvert-\left\lvert\tilde{\gamma}^i_x\right\rvert\right)\mathfrak{p}\left(\left\lvert\gamma^i_x\right\rvert^{-1},\left\lvert\widetilde{\gamma}^i_x\right\rvert^{-1}\right)\\
\leq &\sup_{t\in[0,T],x\in[0,1]}\left\lvert\gamma^i_x-\widetilde{\gamma}^i_x\right\rvert\mathfrak{p}\left(\left\lvert\gamma^i_x\right\rvert^{-1},\left\lvert\widetilde{\gamma}^i_x\right\rvert^{-1}\right)\leq C\left(\boldsymbol{c}\right)\vertiii{\gamma-\widetilde{\gamma}}_{\mathbb{E}_T}\,.
\end{align*}
\end{proof}
This shows the estimate
\begin{equation*}
\left\lVert \mathfrak{p}_j\left(\gamma^i_x,\gamma^i_{xx},\gamma^i_{xxx},\left\lvert \gamma^i_x\right\rvert^{-1},\widetilde{\gamma}^i_x,\widetilde{\gamma}^i_{xx},\widetilde{\gamma}^i_{xxx},\left\lvert \widetilde{\gamma}^i_x\right\rvert^{-1}\right)\left(a-\tilde{a}\right)^j\right\rVert_{L_p\left((0,T);L_p\left(0,1\right)\right)}\leq TC\left(\boldsymbol{c},M\right)\vertiii{\gamma-\widetilde{\gamma}}_{\mathbb{E}_T}\,.
\end{equation*}
To prove the analogous result for the boundary term we make use of the following Lemma.
\begin{lemma}\label{embeddingsobolevhölder}
Let $p\in(1,\infty)$, $d\geq 1$, $T$ be positive and $0<\alpha<\beta<1$. Then 
\begin{equation*}
C^\beta\left([0,T];\mathbb{R}^d\right)\hookrightarrow W_p^{\alpha}\left((0,T);\mathbb{R}^d\right)
\end{equation*} 
and there exist positive constants $\sigma=\sigma(\alpha,\beta)$, $C\left(p,\alpha,\beta\right)$ such that for all $f\in C^\beta\left([0,T];\mathbb{R}^d\right)$,
\begin{equation*}
\left\lVert f \right\rVert_{W_p^{\alpha}\left((0,T);\mathbb{R}^d\right)}\leq C\left(p,\alpha,\beta\right)T^{\sigma(\alpha,\beta)}\left\lVert f\right\rVert_{C^\beta\left([0,T];\mathbb{R}^d\right)}\,.
\end{equation*}
\end{lemma}
\begin{proof}
	The assertion follows directly by estimating the respective expression.
\end{proof}
\begin{lemma}\label{compositionhölder}
	Let $d\geq 1$, $T$ be positive and $\beta\in(0,1)$. Given $F\in C^2\left(\mathbb{R}^d;\mathbb{R}\right)$ and $f\in C^\beta\left([0,T];\mathbb{R}^d\right)$ with $f\left([0,T]\right)\subset K$ for some compact convex subset $K\subset\mathbb{R}^d$, the composition $F\circ f$ lies in $C^\beta\left([0,T];\mathbb{R}\right)$ and satisfies the estimate
	\begin{equation*}
	\left[F\circ f\right]_{C^\beta\left([0,T];\mathbb{R}\right)}\leq \left\lVert F\right\rVert_{C^1\left(K;\mathbb{R}\right)}\left\lVert f\right\rVert_{C^\beta\left([0,T];\mathbb{R}^d\right)}\,.
	\end{equation*}
	If $g$ is another function in $C^\beta\left([0,T];\mathbb{R}^d\right)$ with $g\left([0,T]\right)\subset K$, it holds
	\begin{align*}
	&\left\lVert F\circ f-F\circ g\right\rVert_{C^\beta\left([0,T];\mathbb{R}\right)}\\
	&\leq \max\left\{1,\left[f\right]_{C^\beta\left([0,T];\mathbb{R}^d\right)}+\left[g\right]_{C^\beta\left([0,T];\mathbb{R}^d\right)}\right\} \left\lVert F\right\rVert_{C^2\left(K;\mathbb{R}\right)}\left\lVert f-g\right\rVert_{C^\beta\left([0,T];\mathbb{R}^d\right)}\,.
	\end{align*}
\end{lemma}
\begin{proof}
	The first assertion follows directly using the fundamental theorem of calculus and convexity of the compact set $K$.
	Given $t\in[0,T]$ we define the function $H\in C^0\left([0,T];\mathbb{R}^d\right)$ by
	\begin{equation*}
	H(t):=\int_0^1 \left(\nabla F\right)\left(\tau f(t)+(1-\tau)g(t)\right)\,\mathrm{d}\tau
	\end{equation*}
and observe that $\left\lVert H\right\rVert_{C^0\left([0,T];\mathbb{R}^d\right)}\leq \left\lVert F\right\rVert_{C^1\left(K;\mathbb{R}\right)}$. As the function $\nabla F\in C^1\left(\mathbb{R}^d;\mathbb{R}^d\right)$ is Lipschitz continuous on the compact set $K$ with Lipschitz constant $0\leq L\leq \left\lVert F\right\rVert_{C^2\left(K;\mathbb{R}\right)}$, we obtain $H\in C^\beta\left([0,T];\mathbb{R}^d\right)$ and
\begin{align*}
&\left[H\right]_{C^\beta\left([0,T];\mathbb{R}^d\right)}=\sup_{s,t\in[0,T]}\left\lvert t-s\right\rvert^{-\beta}\left\lvert H(t)-H(s)\right\rvert\\
\leq&\sup_{s,t\in[0,T]}\left\lvert t-s\right\rvert^{-\beta}\int_0^1\left\lvert\left(\nabla F\right)\left(\tau f(t)+(1-\tau)g(t)\right)-\left(\nabla F\right)\left(\tau f(s)+(1-\tau)g(s)\right)\right\rvert\mathrm{d}\tau\\
\leq &\sup_{s,t\in[0,T]}\left\lvert t-s\right\rvert^{-\beta}\left\lVert F\right\rVert_{C^2\left(K;\mathbb{R}\right)}\int_0^1\tau\left\lvert f(t)-f(s)\right\rvert+(1-\tau)\left\lvert g(t)-g(s)\right\rvert\mathrm{d}\tau\\
\leq &\left\lVert F\right\rVert_{C^2\left(K;\mathbb{R}\right)}\left[f\right]_{C^\beta\left([0,T];\mathbb{R}^d\right)}\left[g\right]_{C^\beta\left([0,T];\mathbb{R}^d\right)}\,.
\end{align*}
In particular using $\left(F\circ f\right)(t)- \left(F\circ g\right)(t)=H(t)\cdot \left(f(t)-g(t)\right)$ we conclude that
\begin{align*}
\left\lVert F\circ f-F\circ g\right\rVert_{C^0\left([0,T];\mathbb{R}\right)}&\leq \left\lVert H\right\rVert_{C^0\left([0,T];\mathbb{R}^d\right)}\left\lVert f-g\right\rVert_{C^0\left([0,T];\mathbb{R}^d\right)}\\
&\leq  \left\lVert F\right\rVert_{C^2\left(K;\mathbb{R}\right)}\left\lVert f-g\right\rVert_{C^\beta\left([0,T];\mathbb{R}^d\right)}
\end{align*}
and
\begin{align*}
&\left[F\circ f-F\circ g\right]_{C^\beta\left([0,T];\mathbb{R}^d\right)}=\sup_{s,t\in[0,T]}\left\lvert t-s\right\rvert^{-\beta}\left\lvert H(t)\cdot\left(f(t)-g(t)\right)-H(s)\cdot\left(f(s)-g(s)\right)\right\rvert\\
\leq &\left[H\right]_{C^\beta\left([0,T];\mathbb{R}^d\right)}\left\lVert f-g\right\rVert_{C^0\left([0,T];\mathbb{R}^d\right)}+\left\lVert H\right\rVert_{C^0\left([0,T]\right)}\left[ f-g\right]_{C^\beta\left([0,T];\mathbb{R}^d\right)}\,.
\end{align*}
\end{proof}

\begin{lemma}\label{derivativesol}
	Let $T$ be positive and $p\in (1,\infty)$. Then the operator
	\begin{equation*}
	W_p^{1,4}\left((0,T)\times(0,1);\mathbb{R}\right)\to W_p^{\nicefrac{1}{4}-\nicefrac{1}{4p}}\left((0,T);\mathbb{R}\right)\,,\quad \gamma\mapsto \left(\gamma_{xxx}\right)_{|x=0}
	\end{equation*}
	is continuous. 
\end{lemma}
\begin{proof}
	This follows from~\cite[Theorem 5.1]{solonnikov2} and Remark~\ref{identification}.
\end{proof}
\begin{prop}\label{N_2 contractive}
Let $p\in (5,\infty)$ and $M$ be positive and $\mathfrak{c}:=\vert\varphi_x(0)\vert$. 
For every $T\in (0,\widetilde{T}\left(\boldsymbol{c},M\right)]$ the map $N_{T,2}$ is well defined and there exist constants $\sigma\in(0,1)$ and $C\left(\boldsymbol{c},\mathfrak{c},M\right)>0$ such that for all $\gamma$, $\widetilde{\gamma}\in\mathbb{E}_T^\varphi\cap\overline{B_M}$,
\begin{equation*}
\vertiii{ N_{T,2}(\gamma)-N_{T,2}\left(\widetilde{\gamma}\right)}_{W_p^{\nicefrac{1}{4}-\nicefrac{1}{4p}}\left((0,T);\mathbb{R}^2\right)}\leq C\left(\boldsymbol{c},\mathfrak{c},M\right)T^\sigma \vertiii{\gamma-\widetilde{\gamma}}_{\mathbb{E}_T}\,.
\end{equation*}
\end{prop}
\begin{proof}
	Given $T\in(0,\widetilde{T}\left(\boldsymbol{c},M\right)]$ and $\gamma\in\mathbb{E}_T^\varphi\cap\overline{B_M}$ the expression $b(\gamma)=N_{T,2}\left(\gamma\right)$ is given by
	\begin{align*}
	b(\gamma)=-\sum_{i=1}^3\frac{1}{\vert \varphi^i_x\vert^3}
	\left\langle \gamma^i_{xxx},\nu_0^i\right\rangle \nu_0^i
	+\sum_{i=1}^3\frac{1}{\vert \gamma^i_x\vert^3}
	\left\langle \gamma^i_{xxx},\nu^i\right\rangle \nu^i -\frac{\mu}{2}\sum_{i=1}^{3}\frac{\gamma^i_x}{\vert\gamma^i_x\vert}\,,
	\end{align*}
	where all functions on the right hand side are evaluated in $x=0$.
	Lemma~\ref{uniformcalphac1} with $T_0\equiv 1$ and some fixed $\theta\in\left(\frac{1+\nicefrac{1}{p}}{4-\nicefrac{4}{p}},1\right)$ imply $\gamma^i\in C^{\beta}\left([0,T];C^1\left([0,1];\mathbb{R}^2\right)\right)$	with $\beta:=\left(1-\theta\right)\left(1-\nicefrac{1}{p}\right) >\nicefrac{1}{4}-\nicefrac{1}{4p}=:\alpha$ and
	\begin{equation*}
	\left\lVert\gamma^i\right\rVert_{C^\beta\left([0,T];C^1\left([0,1];\mathbb{R}^2\right)\right)}\leq C\vertiii{ \gamma^i}_{W_p^{1,4}\left((0,T)\times(0,1);\mathbb{R}^2\right)}\leq C\vertiii{\gamma}_{\mathbb{E}_T}\leq C(M)\,.
	\end{equation*}
	This implies in particular $\gamma^i_x(0)\in C^\beta\left([0,T];\mathbb{R}^2\right)$ and with Lemma~\ref{curveregular} we conclude $\gamma^i_x(t,0)\in \overline{B_{\boldsymbol{c}}\left(\varphi^i_x(0)\right)}=:K$  for all $t\in[0,T]$. For $j\in\{1,3\}$ the function $x\mapsto \left\lvert x\right\rvert^{-j}=\left(x_1^2+x_2^2\right)^{-\nicefrac{j}{2}}$ is smooth on $\mathbb{R}^2\setminus\left\{0\right\}$ and can be extended to a function $F^j\in C^2\left(\mathbb{R}^2;\mathbb{R}\right)$ such that $F^j_{|K}=\left\lvert \cdot\right\rvert^{-j}$.
	Lemma~\ref{compositionhölder} implies $t\mapsto \vert\gamma^i_x(t,0)\vert ^{-j}\in C^\beta\left([0,T];\mathbb{R}\right)$ for $j\in\{1,3\}$. Since $\nu^i=R\left(\frac{\gamma^i_x}{\vert\gamma^i_x\vert}\right)$ with $R$ the rotation matrix to the angle $\frac{\pi}{2}$, we may conclude $t\mapsto \nu^i(t,0)\in C^\beta\left([0,T];\mathbb{R}^2\right)$ as H\"{o}lder spaces are stable under products.
	As $W_p^{\nicefrac{1}{4}-\nicefrac{1}{4p}}\left((0,T);\mathbb{R}^2\right)$ is a Banach algebra the preceding arguments combined with Lemma~\ref{derivativesol} imply that $N_{T,2}$ is well defined. 
	To derive the estimate we let $\gamma$ and $\widetilde{\gamma}$ in $\mathbb{E}_T^\varphi\cap\overline{B_M}$ be fixed and note that $\gamma_{|t=0}=\widetilde{\gamma}_{|t=0}=\varphi$ implies $N_{T,2}\left(\gamma\right)_{|t=0}-N_{T,2}\left(\widetilde{\gamma}\right)_{|t=0}= 0$. Thus the terms need to be estimated in the usual sub multiplicative norm on $W_p^{\nicefrac{1}{4}-\nicefrac{1}{4p}}\left((0,T);\mathbb{R}^2\right)$.
In~\cite[Proposition~3.28]{garmenzplu}  it is shown that  $b\left(\gamma\right)-b\left(\widetilde{\gamma}\right)$ can be written as
	\begin{align}
	&\sum_{i=1}^3\frac{1}{\vert \gamma^i_x\vert^3}
	\left\langle \tilde{\gamma}^i_{xxx},\nu^i-\tilde{\nu}^i\right\rangle \nu^i
	+\sum_{i=1}^3\frac{1}{\vert \gamma^i_x\vert^3}
	\left\langle \tilde{\gamma}^i_{xxx},\tilde{\nu}^i\right\rangle\left(\nu^i- \tilde{\nu}^i \right)\label{lineone}\\
	+&\sum_{i=1}^3\frac{1}{\vert \varphi^i_x\vert^3}
	\left\langle (\tilde{\gamma}^i_{xxx}-\gamma^i_{xxx}),\nu_0^i-\nu^i\right\rangle \nu_0^i+\sum_{i=1}^3\frac{1}{\vert \varphi^i_x\vert^3}
	\left\langle (\tilde{\gamma}^i_{xxx}-\gamma^i_{xxx}),\nu^i\right\rangle \left(\nu_0^i-\nu^i\right)\label{linetwo}\\
	+&\sum_{i=1}^3
	\left(\frac{1}{\vert \gamma^i_x\vert^3}-\frac{1}{\vert \tilde{\gamma}^i_x\vert^3}\right)
	\left\langle \tilde{\gamma}^i_{xxx},\tilde{\nu}^i\right\rangle \tilde{\nu}^i 
	+\sum_{i=1}^3\left(\frac{1}{\vert \varphi^i_x\vert^3}-\frac{1}{\vert \gamma^i_x\vert^3}\right)
	\left\langle \tilde{\gamma}^i_{xxx}-\gamma^i_{xxx},\nu^i\right\rangle \nu^i 
	\,\label{linethree}
	\end{align}
	evaluated at $x=0$. Observe that by Lemma~\ref{extensionE_T} and Lemma~\ref{derivativesol} with $T_0\equiv 1$,
	\begin{align*}
	&\left\lVert\gamma^i_{xxx}(0)-\widetilde{\gamma}^i_{xxx}(0)\right\rVert_{W_p^{\alpha}\left((0,T);\mathbb{R}^2\right)}=\left\lVert\left(\boldsymbol{E}\gamma^i\right)_{xxx}(0)-\left(\boldsymbol{E}\widetilde{\gamma}^i\right)_{xxx}(0)\right\rVert_{W_p^{\alpha}\left((0,T);\mathbb{R}^2\right)}\\
	\leq &\left\lVert\left(\boldsymbol{E}\left(\gamma^i-\widetilde{\gamma}^i\right)\right)_{xxx}(0)\right\rVert_{W_p^{\alpha}\left((0,T_0);\mathbb{R}^2\right)}\leq C(T_0)\left\lVert \boldsymbol{E}\left(\gamma^i-\widetilde{\gamma}^i\right)\right\rVert_{W_p^{1,4}\left((0,T_0)\times(0,1);\mathbb{R}^2\right)}\\
	\leq &\,C\left\lVert \boldsymbol{E}\left(\gamma^i-\widetilde{\gamma}^i\right)\right\rVert_{W_p^{1,4}\left((0,\infty)\times(0,1);\mathbb{R}^2\right)}\leq C(p)\vertiii{\gamma^i-\widetilde{\gamma}^i}_{W_p^{1,4}\left((0,T)\times(0,1);\mathbb{R}^2\right)}\leq C(p)\vertiii{\gamma-\widetilde{\gamma}}_{\mathbb{E}_T}\,.
	\end{align*}
	The same estimate shows $\left\lVert \gamma^i_{xxx}(0)\right\rVert_{W_p^\alpha\left((0,T);\mathbb{R}^2\right)}\leq C(p,M)$. Moreover, Lemma~\ref{compositionhölder} implies
	\begin{align*}
	&\left\lVert\left\lvert\gamma^i_x(0)\right\rvert^{-j}-\left\lvert\widetilde{\gamma}^i_x(0)\right\rvert^{-j} \right\rVert_{C^\beta\left([0,T];\mathbb{R}\right)}=\left\lVert F^j\left(\gamma^i_x(0)\right)-F^j\left(\widetilde{\gamma}^i_x(0)\right)\right\rVert_{C^\beta\left([0,T];\mathbb{R}\right)}\\
	\leq&\,\max\{1,\left\lVert\gamma^i_x(0)\right\rVert_{C^\beta\left([0,T];\mathbb{R}^2\right)}+\left\lVert\widetilde{\gamma}^i_x(0)\right\rVert_{C^\beta\left([0,T];\mathbb{R}^2\right)}\}\left\lVert F^j\right\rVert_{C^2\left(K;\mathbb{R}\right)}\left\lVert\gamma^i_x(0)-\widetilde{\gamma}^i_x(0)\right\rVert_{C^\beta\left([0,T];\mathbb{R}^2\right)}\\
	\leq & \,C(M)\left\lVert F^j\right\rVert_{C^2\left(K;\mathbb{R}\right)}\vertiii{\gamma-\widetilde{\gamma}}_{\mathbb{E}_T}\leq C\left(\boldsymbol{c},\mathfrak{c},M\right) \vertiii{\gamma-\widetilde{\gamma}}_{\mathbb{E}_T}
	\end{align*}
	for $j\in\{1,3\}$, and similarly,
	\begin{equation*}
	\left\lVert\left\lvert\gamma^i_x(0)\right\rvert^{-j}\right\rVert_{C^\beta\left([0,T];\mathbb{R}\right)}\leq C\left(\boldsymbol{c},\mathfrak{c},M\right)\,.
	\end{equation*}
	As the H\"{o}lder norm is sub multiplicative, we obtain in particular
	\begin{align*}
	\left\lVert\nu^i(0)-\widetilde{\nu}^i(0)\right\rVert_{C^\beta}&\leq \left\lVert\left\lvert\gamma^i_x(0)\right\rvert^{-1}-\left\lvert\widetilde{\gamma}^i_x(0)\right\rvert^{-1}\right\rVert_{C^\beta}\left\lVert\gamma^i_x(0)\right\rVert_{C^\beta}+\left\lVert\left\lvert\widetilde{\gamma}^i_x(0)\right\rVert_{C^\beta}\right\rVert\left\lVert\gamma^i_x(0)-\widetilde{\gamma}^i_x(0)\right\rVert_{C^\beta}\\
	&\leq C\left(\boldsymbol{c},\mathfrak{c},M\right) \vertiii{\gamma-\widetilde{\gamma}}_{\mathbb{E}_T}
	\end{align*}
	and similarly $\left\lVert\nu^i(0)\right\rVert_{C^\beta\left([0,T];\mathbb{R}^2\right)}\leq C\left(\boldsymbol{c},\mathfrak{c},M\right)$. Combining these estimates with Lemma~\ref{embeddingsobolevhölder} we conclude that there exists $\sigma\in(0,1)$ such that each summand of the expression $b\left(\gamma\right)-b\left(\widetilde{\gamma}\right)$ is bounded by 
	\begin{equation*}
	 C\left(\boldsymbol{c},\mathfrak{c},M\right) T^\sigma \vertiii{\gamma-\widetilde{\gamma}}_{\mathbb{E}_T}\,.
	\end{equation*}
\end{proof}
\begin{cor}\label{K_T contractive}
	Let $p\in(5,\infty)$ and $M$ be positive. There exists $T\left(\boldsymbol{c},\mathfrak{c},M\right)\in(0,\widetilde{T}\left(\boldsymbol{c},M\right)]$ such that for every $T\in \left(0,T\left(\boldsymbol{c},\mathfrak{c},M\right)\right]$ the map 
	\begin{equation*}
	K_T:\left(\mathbb{E}_T^\varphi\cap\overline{B_M},\vertiii{\cdot}_{\mathbb{E}_T}\right)\to\left(\mathbb{E}_T^\varphi,\vertiii{\cdot}_{\mathbb{E}_T}\right)\,, \quad\gamma\mapsto K_T(\gamma):=L_T^{-1}\left(N_{T,1}\left(\gamma\right),N_{T,2}\left(\gamma\right),\varphi\right)
	\end{equation*}
	is a contraction.
\end{cor}
\begin{proof}
	Given $T\in(0,\widetilde{T}\left(\boldsymbol{c},M\right)]$ and $\gamma\in\mathbb{E}_T^\varphi\cap\overline{B_M}$ a direct calculation shows that the admissible initial parametrisation $\varphi$ satisfies the linear compatibility conditions with respect to $N_{T,2}(\gamma)$. In particular, $\left(N_{T,1}\left(\gamma\right),N_{T,2}\left(\gamma\right),\varphi\right)\in\mathbb{F}_T$ and the map $K_T$ is well defined. Moreover, Lemma~\ref{Luniformlybounded} with $T_0\equiv 1$ and Proposition~\ref{N_1 contractive} and~\ref{N_2 contractive} imply for all $\gamma$, $\widetilde{\gamma}\in\mathbb{E}_T^\varphi\cap\overline{B}_M$,
	\begin{align*}
	&\vertiii{ K_T\left(\gamma\right)-K_T\left(\widetilde{\gamma}\right)}_{\mathbb{E}_T}\leq \sup_{T\in(0,1]}\vertiii{L_T^{-1}}_{\mathcal{L}\left(\mathbb{F}_T,\mathbb{E}_T\right)}\vertiii{\left(N_{T,1}(\gamma)-N_{T,1}(\widetilde{\gamma}),N_{T,2}(\gamma)-N_{T,2}(\widetilde{\gamma}),0\right)}_{\mathbb{F}_T}\\
	\leq & C(p)\left\lVert N_{T,1}(\gamma)-N_{T,1}(\widetilde{\gamma})\right\rVert_{L_p\left((0,T);L_p\left((0,1);(\mathbb{R}^2)^3\right)\right)}+ \vertiii{N_{T,2}(\gamma)-N_{T,2}(\widetilde{\gamma})}_{W_p^{\nicefrac{1}{4}-\nicefrac{1}{4p}}\left((0,T);\mathbb{R}^2\right)}\\
	\leq & C\left(\boldsymbol{c},\mathfrak{c},M\right)T^\sigma\vertiii{\gamma-\widetilde{\gamma}}_{\mathbb{E}_T}\,.
	\end{align*}
\end{proof}
\begin{lemma}
	Let $p\in(5,\infty)$ and $T_0$ be positive. There exists a continuous linear operator
	\begin{equation*}
	\mathcal{E}:W_p^{4-\nicefrac{4}{p}}\left((0,1);\mathbb{R}\right)\to W_p^{1,4}\left((0,T_0)\times(0,1);\mathbb{R}\right)
	\end{equation*}
	such that $\left\lVert\mathcal{E}\varphi\right\rVert_{W_p^{1,4}}\leq c(T_0)\left\lVert\varphi \right\rVert_{W_p^{4-\nicefrac{4}{p}}}$ and $\left(\mathcal{E}\varphi\right)_{|t=0} =\varphi$.
\end{lemma}
	\begin{proof}
	Using reflection and a cut-off function we may construct a linear and continuous extension operator 
	\begin{equation*}
	E:W_p^{4-\nicefrac{4}{p}}\left((0,1);\mathbb{R}\right)\to W_p^{4-\nicefrac{4}{p}}\left(\mathbb{R};\mathbb{R}\right)
	\end{equation*}
By~\cite[Corollary~6.1.12]{Prusssimonett} the initial value problem
	\begin{align*}
	\partial_t u(t,x)+ \Delta^2u(t,x)&= 0\,,\qquad\qquad t\in(0,T_0)\,,x\in\mathbb{R}\,,\\
	u(0)&=\psi\,,
	\end{align*}
	admits a unique solution $u:=\mathcal{L}^{-1}\psi\in W_p^{1,4}\left((0,T_0)\times\mathbb{R};\mathbb{R}\right)$ depending linearly and continuously on the initial value $\psi\in W_p^{4-\nicefrac{4}{p}}\left(\mathbb{R};\mathbb{R}\right)$. The result follows setting $\mathcal{E}:=R\circ L^{-1}\circ E$ where $R: W_p^{1,4}\left((0,T_0)\times\mathbb{R};\mathbb{R}\right)\to  W_p^{1,4}\left((0,T_0)\times(0,1);\mathbb{R}\right)$ is the restriction operator,.
	\end{proof}
\begin{prop}\label{selfmapping}
	There exists a positive radius $M$ depending on the norm of $\varphi$ in $W_p^{4-\nicefrac{4}{p}}\left((0,1)\right)$ and a positive time $\widetilde{T}(\boldsymbol{c},\mathfrak{c},M)$ such that for all $T\in (0,\widetilde{T}(\boldsymbol{c},\mathfrak{c},M)]$ we have a well defined map 
	\begin{equation*}
K_T:\mathbb{E}_T^\varphi\cap\overline{B_M}\to\mathbb{E}_T^\varphi\cap\overline{B_M}\,.
\end{equation*}
\end{prop}
\begin{proof}
We let $T_0\equiv 1$ and define
\begin{equation*}
M:=2\max\left\{\sup_{T\in(0,1]}\vertiii{L_T^{-1}}_{\mathcal{L}\left(\mathbb{F}_T,\mathbb{E}_T\right)},1\right\}\max\left\{\vertiii{\mathcal{E}\varphi}_{\mathbb{E}_1},\vertiii{\left(N_{1,1}(\mathcal{E}\varphi),N_{1,2}(\mathcal{E}\varphi),\varphi\right)}_{\mathbb{F}_1}\right\}\,.
\end{equation*}	
In particular, $\mathcal{E}\varphi$ lies in $\mathbb{E}_T\cap\overline{B_M}$ for all $T\in (0,1]$. Moreover, for all $T\in (0,1]$,
\begin{equation*}
\vertiii{K_T\left(\mathcal{E}\varphi\right)}_{\mathbb{E}_T}\leq \sup_{T\in(0,1]}\vertiii{L_T^{-1}}_{\mathcal{L}\left(\mathbb{F}_T,\mathbb{E}_T\right)}\vertiii{\left(N_{T,1}\left(\mathcal{E}\varphi\right),N_{T,2}\left(\mathcal{E}\varphi\right),\varphi\right)}_{\mathbb{F}_T}\leq \nicefrac{M}{2}\,.
\end{equation*}
Let $T\left(\boldsymbol{c},\mathfrak{c},M\right)$ be the corresponding time as in Corollary~\ref{K_T contractive}. Given $T\in (0,T\left(\boldsymbol{c},\mathfrak{c},M\right)]$ and $\gamma\in \mathbb{E}_T^\varphi\cap\overline{B_M}$ we observe that for some $\sigma\in(0,1)$,
\begin{equation*}
\vertiii{ K_T\left(\gamma\right)-K_T\left(\mathcal{E}\varphi\right)}_{\mathbb{E}_T}\leq C\left(\boldsymbol{c},\mathfrak{c},M\right)T^\sigma\vertiii{\gamma-\mathcal{E}\varphi}_{\mathbb{E}_T}\leq C\left(\boldsymbol{c},\mathfrak{c},M\right)T^\sigma 2M\,.
\end{equation*}
We choose $\widetilde{T}(\boldsymbol{c},\mathfrak{c},M)\in (0,T\left(\boldsymbol{c},\mathfrak{c},M\right)]$ so small that $C\left(\boldsymbol{c},\mathfrak{c},M\right)T^\sigma 2M\leq \nicefrac{M}{2}$ for all $T\in (0,\widetilde{T}(\boldsymbol{c},\mathfrak{c},M)]$. Finally, we conclude for all $T\in (0,\widetilde{T}(\boldsymbol{c},\mathfrak{c},M)]$ and $\gamma\in \mathbb{E}_T^\varphi\cap\overline{B_M}$,
\begin{equation*}
\vertiii{K_T(\gamma)}_{\mathbb{E}_T}\leq \vertiii{K_T(\gamma)-K_T(\mathcal{E}\varphi)}_{\mathbb{E}_T}+\vertiii{K_T(\mathcal{E}\varphi)}_{\mathbb{E}_T}\leq\nicefrac{M}{2}+\nicefrac{M}{2}= M\,.
\end{equation*}
\end{proof}
\begin{teo}\label{short time existence}
	Let $p\in(5,\infty)$ and $\varphi$ be an admissible initial parametrisation. There exists a positive time $\widetilde{T}\left(\varphi\right)$ depending on $\min_{i\in\{1,2,3\},x\in[0,1]}\vert\varphi^i_x(x)\vert$, $\vert\varphi_x(0)\vert$ and $\left\lVert\varphi\right\rVert_{W_p^{4-\nicefrac{4}{p}}\left((0,1)\right)}$ such that for all $\boldsymbol{T}\in (0,\widetilde{T}(\varphi)]$ the system~\eqref{TriodC^0} has a solution in 
	\begin{equation*}
		\boldsymbol{E}_{\boldsymbol{T}}=W_p^{1}\left((0,\boldsymbol{T});L_p\left((0,1);(\mathbb{R}^2)^3\right)\right)\cap L_p\left((0,\boldsymbol{T});W_p^4\left((0,1);(\mathbb{R}^2)^3\right)\right)
	\end{equation*}
	which is unique in $\boldsymbol{E}_{\boldsymbol{T}}\cap\overline{B_M}$ where
	\begin{equation*}
M:=2\max\left\{\sup_{T\in(0,1]}\vertiii{L_T^{-1}}_{\mathcal{L}\left(\mathbb{F}_T,\mathbb{E}_T\right)},1\right\}\max\left\{\vertiii{\mathcal{E}\varphi}_{\mathbb{E}_1},\vertiii{\left(N_{1,1}(\mathcal{E}\varphi),N_{1,2}(\mathcal{E}\varphi),\varphi\right)}_{\mathbb{F}_1}\right\}\,.
	\end{equation*}
\end{teo}
\begin{proof}
	Let $M$ and $\widetilde{T}\left(\boldsymbol{c},\mathfrak{c},M\right)$ be the radius and time as in Proposition~\ref{selfmapping} and let $\boldsymbol{T}\in(0,\widetilde{T}\left(\boldsymbol{c},\mathfrak{c},M\right)]$.
	The solutions of the system~\eqref{TriodC^0} in the space $\boldsymbol{E}_{\boldsymbol{T}}\cap\overline{B_M}$
	 are precisely the fixed points of the mapping $K_{\boldsymbol{T}}$ in $\mathbb{E}_{\boldsymbol{T}}^\varphi\cap\overline{B_M}$. As $K_{\boldsymbol{T}}$ is a contraction of the complete metric space $\mathbb{E}_{\boldsymbol{T}}^\varphi\cap\overline{B_M}$, existence and uniqueness of a solution follow from the Contraction Mapping Principle.
\end{proof}

%%%%%%%%%%%%%%%%%%%%%%%%%%%%%%%%%%%%%%%%%%%%%%%%%
%%%%%%%%%%%%%%%%%%%%%%%%%%%%%%%%%%%%%%%%%%%%%%%%%
\section{Parabolic regularisation for the Analytic Problem}
%%%%%%%%%%%%%%%%%%%%%%%%%%%%%%%%%%%%%%%%%%%%%%%%%
%%%%%%%%%%%%%%%%%%%%%%%%%%%%%%%%%%%%%%%%%%%%%%%%%

In this section we show that every solution to the Analytic Problem~\eqref{TriodC^0} is smooth for positive times. To this end we use the classical theory in~\cite{solonnikov2} on solutions to linear parabolic equations in parabolic H\"{o}lder spaces. The definition and properties of these spaces are given in~\cite[\S11, \S13]{solonnikov2}.
\begin{lemma}\label{regularity}
	Let $p\in(5,\infty)$ and $T$ be positive. There exists $\alpha\in (0,1)$ such that for all $\gamma\in \boldsymbol{E}_T$,
	\begin{equation*}
	\gamma\,,\,\gamma_x\,,\,\gamma_{xx}\in C^{\frac{\alpha}{4},\alpha}\left([0,T]\times[0,1];(\mathbb{R}^2)^3\right)
	\end{equation*}
	and
	\begin{equation*}
	t\mapsto\gamma_x(t,0)\in C^{\frac{1+\alpha}{4}}\left([0,T];(\mathbb{R}^2)^3\right)\,.
	\end{equation*}
\end{lemma}
\begin{proof}
	In the proof of Proposition~\ref{hölderspacetime} it is shown that for all $\theta\in(0,1)$,
	\begin{equation*}
	\boldsymbol{E}_T\hookrightarrow C^{(1-\theta)(1-\nicefrac{1}{p})}\left([0,T];W_p^{\theta(4-\nicefrac{4}{p})}\left((0,1);(\mathbb{R}^2)^3\right)\right)\,.
	\end{equation*}
In particular, we obtain for all $\theta\in\left(\frac{1+\nicefrac{1}{p}}{4-\nicefrac{4}{p}},1\right)$ the continuous embedding
\begin{equation*}
\boldsymbol{E}_T\hookrightarrow C^{(1-\theta)(1-\nicefrac{1}{p})}\left([0,T];C^1\left([0,1];(\mathbb{R}^2)^3\right)\right)\,.
\end{equation*}
It is straightforward to verify that this implies for all $\gamma\in\boldsymbol{E}_T$ and all $\alpha\in(0,1)$,
\begin{equation*}
\gamma_x\in C^{\frac{1+\alpha}{4}}\left([0,T];C^0\left([0,1];(\mathbb{R}^2)\right)\right)\,.
\end{equation*}
The Sobolev Embedding Theorem yields further for all $\theta\in\left(\frac{2+\nicefrac{1}{p}}{4-\nicefrac{4}{p}},1\right)$,
\begin{equation*}
\boldsymbol{E}_T\hookrightarrow C^{(1-\theta)(1-\nicefrac{1}{p})}\left([0,T];C^2\left([0,1];(\mathbb{R}^2)^3\right)\right)
\end{equation*}
which implies $\gamma$, $\gamma_{x}$, $\gamma_{xx}\in C^{\frac{\alpha}{4}}\left([0,T];C^0\left([0,1];(\mathbb{R}^2)^3\right)\right)$ for all $\gamma\in\boldsymbol{E}_T$ and all $\alpha\in(0,1)$. Finally, Theorem~\ref{embeddingBUC} gives for all $\alpha\in(0,1)$,
\begin{equation*}
\gamma\,,\gamma_x\,,\gamma_{xx}\in C^0\left([0,T];C^1\left([0,1];(\mathbb{R}^2)^3\right)\right)\hookrightarrow C^0\left([0,T];C^\alpha\left([0,1];(\mathbb{R}^2)^3\right)\right)\,.
\end{equation*}
\end{proof}

\begin{prop}\label{Hölderreguk=1}
	Let $p\in(5,\infty)$, $T$ be positive and $\varphi$ be an admissible initial parametrisation. Suppose that $\gamma\in \boldsymbol{E}_T$ is a solution to the Analytic Problem~\eqref{TriodC^0} in the time interval $[0,T]$ with initial datum $\varphi$ in the sense of Definition~\ref{solutionanalytic}. Then there exists $\alpha\in(0,1)$ such that for all $\varepsilon\in(0,T)$, 
	\begin{equation*}
	\gamma\in
	C^{\frac{4+\alpha}{4}, 4+\alpha}\left([\varepsilon,T]\times[0,1];(\mathbb{R}^2)^3\right)\,.
	\end{equation*}
\end{prop}
\begin{proof}
	Let $\varepsilon\in(0,T)$ and $\eta\in C_0^\infty\left((0,\infty);\mathbb{R}\right)$ be a cut--off function with $\eta\equiv 1$ on $[\varepsilon,T]$. It is straightforward to check that the function $g=\left(g^1,g^2,g^3\right)$ defined by
	\begin{equation*}
	(t,x)\mapsto g^i(t,x):=\eta(t)\gamma^i(t,x)
	\end{equation*}
	lies in $W^{1,4}_p\left((0,T)\times(0,1);(\mathbb{R}^2)^3\right)$ and satisfies a parabolic boundary value problem of the following form: For all $t\in (0,T)$, $x\in(0,1)$, $y\in\{0,1\}$ and $i\in\{1,2,3\}$:
	\begin{equation}\label{systemcutoff}
	\begin{cases}
	\begin{array}{ll}
	g^i_t(t,x)+\frac{2}{\vert\gamma^i_x(t,x)\vert^4}g^i_{xxxx}(t,x)+f\left(\gamma^i_x,\gamma^i_{xx},g^i_x,g^i_{xx},g^i_{xxx}\right)(t,x)&=\eta'(t)\gamma^i(t,x)\,,
	\\
	g^{1}(t,0)-g^{2}(t,0)&=0 \,,\\
	g^{1}(t,0)-g^{3}(t,0)&=0 \,,\\
	g^{i}(t,1)&=\eta(t)\varphi^i(1) \,,\\
	g^i_{xx}(t,y)&=0 \,,\\
	-\sum_{i=1}^3\left(\frac{1}{\vert \gamma^i_x\vert^3}
	\left\langle g^i_{xxx},\nu^i\right\rangle \nu^i\right)(t,0)&=\eta(t)\sum_{i=1}^3h\left(\gamma^i_x(t,0)\right)\,,
\\
	g^{i}(0,x)&=0\,,\\
	\end{array}
	\end{cases}
	\end{equation}
	where $h:\left(\mathbb{R}^2\setminus\{0\}\right)\to\mathbb{R}^2$ is smooth and $\nu^i(t,0):=R\left(\frac{\gamma^i_x(t,0)}{\vert\gamma^i_x(t,0)\vert}\right)$ with $R$ the counter clockwise rotation by $\frac{\pi}{2}$. Moreover, the lower order terms in the motion equation are given by
	\begin{align*}
	f\left(\gamma^i_x,\gamma^i_{xx},g^i_x,g^i_{xx},g^i_{xxx}\right)(t,x)=&-12 \frac{\left\langle\gamma^i_{xx},\gamma^i_x\right\rangle}{\left\lvert\gamma^i_x\right\rvert^6}g^i_{xxx}-8\frac{\gamma^i_{xx}}{\vert\gamma^i_x\vert ^6}\left\langle g^i_{xxx},\gamma^i_x\right\rangle\\	&-\left(5\frac{\lvert\gamma^i_{xx}\rvert ^2}{\lvert\gamma^i_x\rvert^6}-35\frac{\left\langle\gamma^i_{xx},\gamma^i_x\right\rangle^2}{\vert\gamma^i_x\vert^8}+\mu\frac{1}{\vert\gamma^i_x\vert ^2}\right)g^i_{xx}\,.
	\end{align*} 
	The boundary value problem is linear in the components of $g$ and in the highest order term of exactly the same structure as the linear system~\eqref{lyntriod} with time dependent coefficients in the motion equation and the third order condition.
	In particular, the same arguments as in~\cite[\S 3.3.3, \S 4.1.2]{garmenzplu} show that the Lopatinski Shapiro condition is satisfied. Exactly as in~\cite[\S 3.3.2]{garmenzplu} we write the unknown $g$ as $g=(g^1,g^2,g^3)=\left(u^1,v^1,u^2,v^2,u^3,v^,3\right)$ and observe that with the choices $t_j=0$, $s_k=4$ for $j,k\in\{1,\dots,6\}$ the system is parabolic in the sense of Solonnikov (see~\cite[page 18]{solonnikov2}). The complementing condition for the initial value is trivially fulfilled.
	In order to apply the existence result in H\"{o}lder spaces~\cite[Theorem 4.9]{solonnikov2} with $l=4+\alpha$, $\alpha\in(0,1)$, it remains to verify the regularity requirements on the coefficients and the right hand side. We observe that there exists $\boldsymbol{r}>0$ such that for all $i\in\{1,2,3\}$,
	\begin{equation*}
	\inf_{t\in[0,T],x\in[0,1]}\vert\gamma^i_x(t,x)\vert \geq \boldsymbol{r}\,.
	\end{equation*}
	As H\"{o}lder spaces are stable under products and composition with smooth functions, the regularity requirements follow from Lemma~\ref{regularity}. 
	As $\eta(0)=\eta'(0)=0$, the initial datum $0$ satisfies the linear compatibility conditions of order $4$ with respect to the given right hand side.
	By~\cite[Theorem 4.9]{solonnikov2} the problem~\eqref{systemcutoff} has a unique solution $\widetilde{g}\in C^{\frac{4+\alpha}{4}, 4+\alpha}\left([0,T]\times[0,1];(\mathbb{R}^2)^3\right)$.
The function $\widetilde{g}$ solves the system~\eqref{systemcutoff} also in the space $W_p^{1,4}\left((0,T)\times(0,1);(\mathbb{R}^2)^3\right)$. The uniqueness assertion in~\cite[Theorem 5.4]{solonnikov2} implies that $g=\tilde{g}$. This shows the claim as $g$ is equal to $\gamma$ on $[\varepsilon,T]\times[0,1]$.
\end{proof}
\begin{teo}\label{analyticsmoothsol}
	Let $p\in(5,\infty)$, $T$ be positive and $\varphi$ be an admissible initial parametrisation. Suppose that $\gamma\in \boldsymbol{E}_T$ is a solution to the Analytic Problem~\eqref{TriodC^0} in the time interval $[0,T]$ with initial datum $\varphi$ in the sense of Definition~\ref{solutionanalytic}. Then the solution $\gamma$ is smooth for positive times in the sense that
$
\gamma\in C^\infty\left([\varepsilon,T]\times[0,1];(\mathbb{R}^2)^3\right)
$
 for every $\varepsilon\in(0,T)$.
\end{teo}
\begin{proof}
	We show inductively that there exists $\alpha\in(0,1)$ such that for all $k\in\mathbb{N}$ and $\varepsilon\in(0,T)$,
	\begin{equation*}
	\gamma\in C^{\frac{2k+2+\alpha}{4},2k+2+\alpha}\left([\varepsilon,T]\times[0,1];(\mathbb{R}^2)^3\right)\,.
	\end{equation*}
	The case $k=1$ is precisely the statement of Proposition~\ref{Hölderreguk=1}. Now assume that the assertion holds true for some $k\in\mathbb{N}$ and consider any $\varepsilon\in(0,T)$. Let $\eta\in C_0^\infty\left((\nicefrac{\varepsilon}{2},\infty);\mathbb{R}\right)$ be a cut--off function with $\eta\equiv 1$ on $[\varepsilon,T]$. By assumption, $\gamma\in C^{\frac{2k+2+\alpha}{4},2k+2+\alpha}\left(\left[\nicefrac{\varepsilon}{2},T\right]\times[0,1];(\mathbb{R}^2)^3\right)$ and thus 
	\begin{equation*}
	(t,x)\mapsto g(t,x):=\eta(t)\gamma(t,x)\in C^{\frac{2k+2+\alpha}{4},2k+2+\alpha}\left(\left[0,T\right]\times[0,1];(\mathbb{R}^2)^3\right)
	\end{equation*}
	is a solution to system~\eqref{systemcutoff} in $[0,T]\times[0,1]$. The coefficients and the right hand side fulfil the regularity requirements of~\cite[Theorem 4.9]{solonnikov2} in the case $l=2(k+1)+2+\alpha$. As $\eta^{(j)}(0)=0$ for all $j\in\mathbb{N}$, the initial value $0$ satisfies the compatibility conditions of order $2(k+1)$ with respect to the given right hand side. Thus ~\cite[Theorem 4.9]{solonnikov2} yields
	\begin{equation*}
	g\in C^{\frac{2(k+1)+2+\alpha}{4}2(k+1)+2+\alpha}\left([0,T]\times[0,1];(\mathbb{R}^2)^3\right)\,.
	\end{equation*}
	This completes the induction as $g=\gamma$ on $[\varepsilon,T]$. By definition of the parabolic H\"{o}lder spaces (see also~\cite[\S 2.1.2]{garmenzplu}) 
	\begin{equation*}
	\bigcap_{k\in\mathbb{N}}C^{\frac{2k+2+\alpha}{4},2k+2+\alpha}\left([\varepsilon,T]\times[0,1];(\mathbb{R}^2)^3\right)=C^\infty\left([\varepsilon,T]\times[0,1]\right)\,.
	\end{equation*}
\end{proof}
\section{Geometric existence and uniqueness}
\subsection{Geometric existence and uniqueness in Sobolev spaces}
In Theorem~\ref{short time existence} we have shown that given an admissible initial parametrisation there exists a unique solution to the Analytic Problem~\eqref{TriodC^0} in some short time interval. This section is devoted to prove that the \textit{Geometric Flow}, namely the merely geometrical problem~\eqref{Triod0}, possesses a unique solution in the sense of Definition~\ref{geometricsolution1} provided that the initial triod is geometrically admissible.
\begin{teo}[Geometric Existence]\label{geometricexistence}
	Let $p\in(5,\infty)$. Suppose that $\mathbb{T}_0$ is a geometrically admissible initial triod as defined in Definition~\ref{admg1}. Then there exists a positive time $\boldsymbol{T}$ and a family $\left(\mathbb{T}(t)\right)_{t\in[0,\boldsymbol{T}]}$ of triods solving the elastic flow~\eqref{Triod0} in $[0,\boldsymbol{T}]$ with initial datum $\mathbb{T}_0$ in the sense of Definition~\ref{geometricsolution1}. The family of networks $\left(\mathbb{T}_t\right)_{t\in[0,\boldsymbol{T}]}$ is parametrised by one function $\gamma\in\boldsymbol{E}_{\boldsymbol{T}}$.
\end{teo}
\begin{proof}
	Let $P^1$, $P^2$ and $P^3$ denote the fixed endpoints of the triod $\mathbb{T}_0$. Suppose that there exists an admissible initial parametrisation $\varphi=\left(\varphi^1,\varphi^2,\varphi^3\right)$ for system~\eqref{TriodC^0} such that $\varphi$ parametrises $\mathbb{T}_0$ with $\varphi^1(0)=\varphi^2(0)=\varphi^3(0)$ and $\varphi^i(1)=P^i$ for $i\in\{1,2,3\}$. Then by Theorem~\ref{short time existence} there exists a positive time $\boldsymbol{T}$ and a function $\gamma=\left(\gamma^1,\gamma^2,\gamma^3\right)\in\boldsymbol{E}_{\boldsymbol{T}}$ solving system~\eqref{TriodC^0} with initial value $\varphi$. It is straightforward to check that $\mathbb{T}(t):=\gamma\left(t,[0,1]\right)$ with $t\in[0,{\boldsymbol{T}})$ is a family of triods solving problem~\eqref{Triod0} in $[0,\boldsymbol{T})$ with initial network $\mathbb{T}_0$ in the sense of Definition~\ref{geometricsolution1} as each curve $\gamma^i(t)$ is regular according to Lemma~\ref{curveregular}. 
	Thus it is enough to prove that $\mathbb{T}_0$ admits a parametrisation $\varphi=\left(\varphi^1,\varphi^2,\varphi^3\right)$ that is an admissible initial value for system~\eqref{TriodC^0}. We proceed analogously as in the proof of~\cite[Lemma 3.31]{garmenzplu}. Let $\sigma=\left(\sigma^1,\sigma^2,\sigma^3\right)$ be a parametrisation for $\mathbb{T}_0$ such that every $\sigma^i$ is regular and $\sigma^i\in W_p^{4-\nicefrac{4}{p}}\left((0,1);\mathbb{R}^2\right)$ for $i\in\{1,2,3\}$. We may assume that $\sigma^1(0)=\sigma^2(0)=\sigma^3(0)$ and $\sigma^i(1)=P^i$ for $i\in\{1,2,3\}$. 
	Lemma~\ref{constructreparametrisation} implies for every $i\in\{1,2,3\}$ the existence of a smooth function $\theta^i:[0,1]\to[0,1]$ with $\theta^i(0)=0$, $\theta^i(1)=1$, $\theta^i_x(0)=\theta^i_x(1)=1$, $\theta^i_{xx}(0)=-\left\langle\sigma^i_{xx}(0),\sigma^i_x(0)\right\rangle$, $\theta^i_{xx}(1)=-\left\langle\sigma^i_{xx}(1),\sigma^i_x(1)\right\rangle$ and $\theta^i_x\geq \frac{1}{4}$ on $[0,1]$.
	As $\theta^i$ is a smooth diffeomorphism of the interval $[0,1]$, the chain rule for Sobolev functions implies $\varphi^i:=\sigma^i\circ\theta^i\in W_p^{3}\left((0,1);\mathbb{R}^2\right)$ and by direct estimates one obtains also $\varphi^i\in W_p^{4-\nicefrac{4}{p}}\left((0,1);\mathbb{R}^2\right)$ for $i\in\{1,2,3\}$. Moreover, $\varphi^i$ satisfies the concurrency, the third order and the non--degeneracy conditions at $y=0$ and  $\varphi^i_x(x)\neq 0$ for all $x\in[0,1]$.
	Finally, we obtain for $y\in\{0,1\}$,
	\begin{equation*}
	\varphi^i_{xx}(y)=\sigma^i_{xx}(y)+\sigma^i_x(y)\theta^i_{xx}(y)=\left\langle\sigma^i_{xx}(y),\sigma^i_x(y)\right\rangle\sigma^i_x(y)+\sigma^i_x(y)\theta^i_{xx}(y)=0\,.
	\end{equation*}
	This shows that $\varphi=\left(\varphi^1,\varphi^2,\varphi^3\right)$ is an admissible initial parametrisation for~\eqref{TriodC^0}.
\end{proof}
\begin{lemma}\label{constructreparametrisation}
	Let $a$, $b\in\mathbb{R}$. There exists a smooth function $\theta:[0,1]\to[0,1]$ such that $\theta(0)=0$, $\theta(1)=1$, $\theta_x(0)=\theta_x(1)=1$, $\theta_{xx}(0)=a$, $\theta_{xx}(1)=b$ and $\theta_x(x)\geq \frac{1}{4}$ for all $x\in[0,1]$. 
\end{lemma}
\begin{proof}
	 We let $p$ and $q$ be the second order Taylor polynomials on $\mathbb{R}$ determined by the constraints $p(0)=0$, $p_x(0)=1$, $p_{xx}(0)=a$ and $q(1)=1$, $q_x(1)=1$, $q_{xx}(1)=b$ , respectively. Let $\delta\in\left(0,\nicefrac{1}{3}\right)$ be such that for all $x\in[0,3\delta]$, $p(x)\leq \frac{1}{4}$, $p_x(x)\geq \frac{3}{4}$ and for all $x\in[1-3\delta,1]$, $q(x)\geq\frac{3}{4}$, and $q_x(x)\geq \frac{3}{4}$. We define a function $f:[0,1]\to[0,1]$ in $W^1_\infty\left((0,1);\mathbb{R}\right)=C^{0,1}\left([0,1];\mathbb{R}\right)$ by
	 \begin{align*}
	 \begin{array}{ll}
	 f(x):=\begin{cases}
	 p(x) &x\in[0,3\delta]\,,\\
	 p(3\delta)+x\left(q(1-3\delta)-p(3\delta)\right) &x\in[3\delta,1-3\delta]\,,\\
	 q(x)&x\in[1-3\delta,1]\,.
	 \end{cases}
	 \end{array}
	 \end{align*}
	 Let $\left(\psi^{\varepsilon}\right)_{\varepsilon>0}$ be the Standard--Dirac sequence on $\mathbb{R}$. Then for every $\varepsilon>0$ the convolution $f*\psi^{\varepsilon}$ lies in $C^\infty\left(\left(\varepsilon,1-\varepsilon\right);\mathbb{R}\right)$ and $f*\psi^\varepsilon$ converges to $f$ in $W_{loc}^{1,\infty}\left((0,1);\mathbb{R}\right)$ as $\varepsilon\searrow 0$. Let $\varepsilon$ be so small that $\left\lVert f*\psi^\varepsilon-f\right\rVert_{W^{1,\infty}\left((\delta,1-\delta)\right)}<\frac{\delta}{10}$ and let $\eta\in C_0^\infty\left((0,1)\right)$ be a cut-off function with $0\leq\eta\leq 1$ on $[0,1]$, $\eta\equiv 1$ on $[2\delta,1-2\delta]$, $\mathrm{supp}\eta\subset(\delta,1-\delta)$ and $\lvert\partial_x^k\eta\lvert\leq C\delta^{-k}$ for all $k\in\mathbb{N}$. Extending $f*\psi^\varepsilon$ by $0$ to the whole interval $[0,1]$ we define $\theta:[0,1]\to\mathbb{R}$ by
	 \begin{equation*}
	 \theta(x):=\left(1-\eta(x)\right)f(x)+\eta(x)\left(f*\psi^\varepsilon\right)(x)\,.
	 \end{equation*}
	 It follows from the construction that $\theta$ is smooth on $[0,1]$ and satisfies the constraints at the boundary points. For $x\in[0,\delta]\cup[1-\delta,1]$ it holds that $\theta_x(x)=f_x(x)\geq \frac{3}{4}$. By the choice of $\varepsilon$ we have for almost every $x\in(\delta,2\delta)\cup(1-2\delta,1-\delta)$,
	 \begin{equation*}
	 \theta_x(x)=f_x(x)+\eta(x)\left(\left(f*\psi^\varepsilon\right)_x(x)-f_x(x)\right)+\eta_x(x)\left(f*\psi^\varepsilon(x)-f(x)\right)\geq \frac{3}{4}-\frac{\delta}{10}-\frac{1}{10}\geq \frac{1}{4}\,.
	 \end{equation*}
	 Moreover, observe for almost every $x\in[2\delta,1-2\delta]$,
	 \begin{equation*}
	 \theta_x(x)=\left(f*\psi^\varepsilon\right)_x(x)\geq f_x(x)-\left\lvert\left(f*\psi^\varepsilon\right)_x(x)-f_x(x)\right\rvert\geq \frac{1}{2}-\frac{1}{10}\geq\frac{1}{4}\,.
	 \end{equation*}
	 By continuity of $\theta_x$ the estimates hold point wise in the respective sets.
\end{proof}
\begin{lemma}\label{compositionstable}
	Let $T$ be positive, $p\in(5,\infty)$ and $f$, $g\in L_p\left(0,T;W_p^4\left((0,1)\right)\right)\cap W_p^1\left(0,T;L_p\left(0,1\right)\right)$ such that for every $t\in[0,T]$ the function $g_t:[0,1]\to[0,1]$ is a $C^1$--diffeomorphism. Then the map $h(t,x):=f\left(t,g(t,x)\right)$ lies in $L_p\left(0,T;W_p^4\left((0,1)\right)\right)\cap W_p^1\left(0,T;L_p\left(0,1\right)\right)$ and all derivatives can be calculated by the chain rule.
\end{lemma}
\begin{proof}
	By Theorem~\ref{embeddingBUC} both $f$ and $g$ lie in $BUC\left([0,T];C^3\left([0,1]\right)\right)$ which implies
	\begin{equation*}
	h\in BUC\left([0,T];C^3\left([0,1]\right)\right)\hookrightarrow L_p\left((0,T);W_p^3\left((0,1)\right)\right)
	\end{equation*}
	using chain rule and directly estimating the terms. For every $t\in[0,T]$ and $x\in[0,1]$ it holds
	\begin{equation*}
	\partial_x^3\left(h(t)\right)(x)=\left(\partial_x^3f_t\right)\left(g_t(x)\right)\left(\partial_xg_t(x)\right)^3+3\left(\partial_x^2f_t\right)\left(g_t(x)\right)\partial_x^2g_t(x)\partial_xg_t(x)+\partial_x^3g_t(x)\left(\partial_xf_t\right)\left(g_t(x)\right)
	\end{equation*}
	where $f_t:=f(t,\cdot)$ and $g_t:=g(t,\cdot)$. There exists a set $N\subset(0,T)$ of measure $0$ such that for every $t\in(0,T)\setminus N$, the functions $x\mapsto \partial_x^3f_t(x)$ and $x\mapsto \partial_x^3g_t(x)$ lie in $W_p^1\left((0,1)\right)$. Given $t\in(0,T)\setminus N$ the map $g_t$ is a $C^1$--diffeomorphism of $(0,1)$ and thus the chain rule for Sobolev functions implies that also $x\mapsto \left(\partial_x^3f_t\right)\left(g_t(x)\right)$ lies in $W_p\left((0,1)\right)$ with derivative $\left(\partial_x^4f_t\circ g_t\right)\partial_xg_t$. As all remaining terms in the formula for $\partial_x^3h(t)$ lie in $C^1\left([0,1]\right)$, the product rule for Sobolev functions implies $\partial_x^3h(t)\in W_p^1\left((0,1)\right)$ and thus $h(t)\in W_p^4\left((0,1)\right)$ for every $t\in(0,T)\setminus N$. Directly estimating the norms one easily obtains that $t\mapsto \partial_x^3h(t)$ lies in $L_p\left(0,T;W_p^1\left((0,1)\right)\right)$ and hence $h\in L_p\left(0,T;W_p^4\left((0,1)\right)\right)$. 
	In the next step we show that $h$ lies in $W_p^1\left(0,T;L_p\left((0,1)\right)\right)$ with distributional derivative
	\begin{equation*}
	\tilde{h}(t)(x):=\tilde{h}(t,x):=\left(\partial_tf\right)\left(t,g(t,x)\right)+\left(\partial_xf\right)\left(t,g(t,x)\right)\partial_tg(t,x)\in L_p\left((0,T);L_p\left((0,1)\right)\right)\,.
	\end{equation*}
	To this end let $\psi\in C_0^\infty\left((0,T);\mathbb{R}\right)$ be a fixed test function. To conclude that
	\begin{equation*}
	\int_0^T h(t)\psi'(t)\,\mathrm{d}t=-\int_0^T\tilde{h}(t)\psi(t)\,\mathrm{d}t \qquad\qquad\text{ in }L_p\left((0,1)\right)
	\end{equation*}
	it is enough to show that the two integrals are equal for almost every $x\in(0,1)$. Suppose that $\mathrm{supp}\psi=[a,b]\subset(0,T)$. Then for every $x\in[0,1]$,
	\begin{align*}
	\int_a^b h(t,x)\psi'(t)\,\mathrm{d}t&=\int_0^b h(t,x)\lim\limits_{\varepsilon \to 0}\frac{1}{\varepsilon}\left(\psi(t+\varepsilon)-\psi(t)\right)\mathrm{d}t\\
	&=\lim\limits_{\varepsilon\to 0}\int_a^b\frac{1}{\varepsilon}\left(h(t-\varepsilon,x)-h(t,x)\right)\psi(t)\,\mathrm{d}t\,.
	\end{align*}
	Observe that for every $t\in[a,b]$, $\varepsilon\in (0,a)$ and $x\in[0,1]$, 
	\begin{align*}
	\frac{1}{\varepsilon}\left(h(t-\varepsilon,x)-h(t,x)\right)=&\frac{1}{\varepsilon}\left(f\left(t-\varepsilon,g(t-\varepsilon,x)\right)-f\left(t-\varepsilon,g(t,x)\right)\right)\\
	&+\frac{1}{\varepsilon}\left(f\left(t-\varepsilon,g(t,x)\right)-f(t,g(t,x))\right)\,.
	\end{align*}
	Using the fundamental theorem of calculus the first term can be written as
	\begin{equation*}
	\frac{1}{\varepsilon}\int_0^1 \left(\partial_xf\right)\left(t-\varepsilon,\tau g(t-\varepsilon,x)+(1-\tau)g(t,x)\right)\mathrm{d}\tau\, \left(g(t-\varepsilon,x)-g(t,x)\right)\,. 
	\end{equation*}
	There exists a subset $\mathcal{N}\subset(0,1)$ of measure $0$ such that for all $x\in (0,1)\setminus\mathcal{N}$ the functions $g\left(\cdot,x\right)$ and $f\left(\cdot,x\right)$ lie in $W_p^1\left((0,T)\right)$ with distributional derivative $\partial_tg\left(\cdot,x\right)$ and $\partial_tf\left(\cdot,x\right)$, respectively. The difference quotients $t\mapsto\frac{1}{\varepsilon}\left(g(t,x)-g(t-\varepsilon,x)\right)$ converge to $t\mapsto\partial_tg(t,x)$ weakly in $L_p\left((a,b)\right)$  for every $x\in(0,1)\setminus\mathcal{N}$. As $\partial_xf$ and $g$ are uniformly continuous on $[0,T]\times[0,1]$, it is straightforward to show that
	\begin{equation*}
	\lim\limits_{\varepsilon\to 0}\sup_{t\in[a,b],x\in[0,1]}\left\lvert\sup_{\tau\in[0,1]}\left(\partial_xf\right)\left(t-\varepsilon,\tau g(t-\varepsilon,x)+(1-\tau)g(t,x)\right)-\left(\partial_xf\right)(t,g(t,x))\right\rvert =0\,.
	\end{equation*}
	As $t\mapsto \left(\partial_xf\right)(t,g(t,x))\psi(t)$ lies in $L_p\left((a,b)\right)$ for every $x\in[0,1]$, we conclude for every $x\in[0,1]\setminus\mathcal{N}$ in the limit $\varepsilon\to 0$,
	\begin{equation*}
	\int_a^b \frac{1}{\varepsilon}\left(f\left(t-\varepsilon,g(t-\varepsilon,x)\right)-f\left(t-\varepsilon,g(t,x)\right)\right)\psi(t)\mathrm{d}t\to-\int_a^b \left(\partial_xf\right)(t,g(t,x))\partial_tg(t,x)\psi(t)\mathrm{d}t\,.
	\end{equation*}
	To estimate the second term we observe that for every $y\in(0,1)\setminus\mathcal{N}$, the difference quotients $t\mapsto\frac{1}{\varepsilon}\left(f\left(t,y\right)-f\left(t-\varepsilon,y\right)\right)$ converge to $t\mapsto \partial_tf(t,y)$  weakly in $L_p\left((a,b)\right)$. In particular, for every $y\in(0,1)\setminus\mathcal{N}$ it holds
	\begin{equation*}
	\lim\limits_{\varepsilon\to 0}\int_a^b\left(\frac{1}{\varepsilon}\left(f(t-\varepsilon,y)-f(t,y)\right)+\left(\partial_t f\right)(t,y)\right)\psi(t)\left\lvert\left(\partial_xg_t\right)\left(g^{-1}_t(y)\right)\right\rvert^{-1}\mathrm{d}t=0\,.
	\end{equation*}
	Using dominated convergence, Fubini's Theorem and the transformation formula we obtain
	\begin{align*}
	0&=\lim\limits_{\varepsilon\to 0}\int_0^1\int_a^b\left(\frac{1}{\varepsilon}\left(f(t-\varepsilon,y)-f(t,y)\right)+\left(\partial_t f\right)(t,y)\right)\psi(t)\left\lvert\left(\partial_xg_t\right)\left(g^{-1}_t(y)\right)\right\rvert^{-1}\mathrm{d}t\,\mathrm{d}y\\
	&=\lim\limits_{\varepsilon\to 0}\int_0^1\int_a^b\left(\frac{1}{\varepsilon}\left(f(t-\varepsilon,g(t,x))-f(t,g(t,x))\right)+\left(\partial_tf\right)(t,g(t,x))\right)\psi(t)\,\mathrm{d}t\,\mathrm{d}x
	\end{align*}
	and thus for almost every $x\in (0,1)$
	\begin{equation*}
	\lim\limits_{\varepsilon\to 0}\int_a^b\frac{1}{\varepsilon}\left(f(t-\varepsilon,g(t,x))-f(t,g(t,x))\right)\psi(t)\,\mathrm{d}t=-\int_a^b \left(\partial_tf\right)(t,g(t,x))\psi(t)\,\mathrm{d}t\,.
	\end{equation*}
\end{proof}
\begin{teo}[Local Geometric Uniqueness]\label{geometricuniqueness1}
	Let $p\in(5,\infty)$, $\mathbb{T}_0$ be a geometrically admissible initial triod and $T$ be positive.
	Suppose that both $\left(\mathbb{T}(t)\right)$ and $\left(\widetilde{\mathbb{T}}(t)\right)$ are solutions to problem~\eqref{Triod0} in the time interval $[0,T]$ with initial datum $\mathbb{T}_0$ in the sense of Definition~\ref{geometricsolution1}. Then there exists a time $\hat{T}\in (0,T]$ such that for all $t\in [0,\hat{T}]$ the networks $\mathbb{T}(t)$ and $\widetilde{\mathbb{T}}(t)$ coincide.
\end{teo}
\begin{proof}
By Theorem~\ref{geometricexistence} there exists a positive radius $M$, a time $\boldsymbol{T}$ and a function $\gamma\in\boldsymbol{E}_{\boldsymbol{T}}\cap\overline{B_M}$ such that $\gamma$ is a solution to~\eqref{TriodC^0} with $\gamma(0)=\varphi$ where $\varphi$ is an admissible initial value for~\eqref{TriodC^0} parametrising $\mathbb{T}_0$. Moreover, the family of triods $\left(\gamma(t,[0,1])\right)$ is a solution to~\eqref{Triod0} in $[0,T)$ with initial datum $\mathbb{T}_0$ in the sense of Definition~\ref{geometricsolution1}. We show that $\left(\mathbb{T}(t)\right)$ coincides with $\left(\gamma(t,[0,1])\right)$ on a small time interval. There exists $b_0\in (0,\min\{\boldsymbol{T},T\})$ and $\gamma_0\in W_p^{1,4}\left((0,b_0)\times[0,1];(\mathbb{R}^2)^3\right)$ such that for all $t\in(0,b_0)$, $\gamma_0(t)$ is a regular parametrisation of $\mathbb{T}(t)$ and $\gamma_0$ is solution to~\eqref{Triod0} with $\gamma_0(0,[0,1])=\mathbb{T}_0$. We aim to construct a family of re parametrisations $\psi^i:[0,T_0]\times[0,1]\to[0,1]$, $i\in\{1,2,3\}$, with some $T_0\in(0,b_0)$ such that
\begin{equation*}
(t,x)\mapsto\gamma_0^i\left(t,\psi^i(t,x)\right)
\end{equation*}
is a solution to~\eqref{TriodC^0} in $[0,T_0]\times(0,1)$ with initial datum $\varphi$.
As argued in the proof of~\cite[Theorem~3.32]{garmenzplu} 
(formal) differentiation and taking into account the specific tangential velocity in~\eqref{TriodC^0} and the additional boundary condition imply that 
 $\psi=\left(\psi^1,\psi^2,\psi^3\right)$ has to satisfy a boundary value problem of the following shape:
 For $t\in(0,T_0)$, $x\in(0,1)$, $y\in\{0,1\}$ and $i\in\{1,2,3\}$,
\begin{equation}
\begin{cases}
\psi^i_t(t,x)+\frac{2\psi^i_{xxxx}(t,x)}{\vert(\gamma^i_0)_x(t,\psi^i(t,x))\vert^4 \vert \psi^i_x(t,x)\vert^4}
+g(\psi^i_{xxx}, \psi^i_{xx}, \psi^i_x, \psi^i,\gamma^i_0)=0\,,\\
\psi^i(t,y)=y\,,\\
\psi^i_{xx}(t,y)=-\left\langle(\gamma^i_0)_{xx}(t,y),(\gamma^i_0)_x(t,y)\right\rangle\left(\psi_x^i(t,y)\right)^2\,,\\
\gamma^i_0(0,\psi^i(0,x))=\varphi^i(x)\,.
\end{cases}
\end{equation}
Notice that by the implicit function theorem, the initial value lies in $W_p^{4-\nicefrac{4}{p}}\left((0,1);(\mathbb{R}^2)^3\right)$.
As this system has a very similar structure as problem~\eqref{TriodC^0} we studied before, analogous arguments as in the proof of Theorem~\ref{short time existence} allow us to conclude that there exists a time $T_0\in(0,b_0)$ and a function $\psi\in W_p^{1,4}\left((0,T_0)\times(0,1);\mathbb{R}^3\right)$ solving the above system.
The time $T_0$ depends on $\left\lVert\varphi\right\rVert_{X_0}$, $\min_{i\in\{1,2,3\},x\in[0,1]}\left\lvert\varphi^i(x)\right\rvert$, $\left\lvert\varphi_x(0)\right\rvert$ and also on $\left\lVert\gamma_0(0)\right\rVert_{X_0}$, $\min_{i\in\{1,2,3\},x\in[0,1]}\left\lvert(\gamma_0^i)_x(0,x)\right\rvert$ and $\left\lvert(\gamma_0)_x(0,0)\right\rvert$ where $X_0=W_p^{4-\nicefrac{4}{p}}((0,1))$. For every $t\in[0,T_0]$ the continuous function $x\mapsto\psi^i(t,x)$ satisfies $\psi^i(t,0)=0$, $\psi^i(t,1)=1$ and $\psi^i_x(t,x)\neq 0 $ for all $x\in[0,1]$. Thus $x\mapsto\psi^i(t,x)$ is a $C^1$--diffeomorphism of the interval $[0,1]$. Lemma~\ref{compositionstable} implies that 
\begin{equation*}
(t,x)\mapsto\gamma_0^i\left(t,\psi^i(t,x)\right)\in W_p^{1,4}\left((0,T_0)\times(0,1);\mathbb{R}^2\right)
\end{equation*}
and by construction, $(t,x)\mapsto\gamma_0\left(t,\psi(t,x)\right)$ solves the Analytic Problem~\eqref{TriodC^0}. As for $\overline{T}\to 0$,
\begin{equation*}
\vertiii{(t,x)\mapsto\gamma_0\left(t,\psi(t,x)\right)}_{W_p^{1,4}\left((0,\overline{T})\times(0,1);(\mathbb{R}^2)^3\right)}\to \left\lVert\varphi\right\rVert_{W_p^{4-\nicefrac{4}{p}}\left((0,1);\left(\mathbb{R}^2\right)^3\right)}
\end{equation*} 
and $\left\lVert\varphi\right\rVert_{W_p^{4-\nicefrac{4}{p}}\left((0,1);(\mathbb{R}^2)^3\right)}\leq \left\lVert \mathcal{E}\varphi\right\rVert_{\mathbb{E}_1}\leq \nicefrac{M}{2}$,
we may choose $\overline{T}\in(0,T_0)$ small enough such that 
\begin{equation*}
(t,x)\mapsto\gamma_0\left(t,\psi(t,x)\right)\in W_p^{1,4}\left((0,\overline{T})\times(0,1);(\mathbb{R}^2)^3\right)\cap \overline{B_M}\,.
\end{equation*}
The uniqueness assertion in Theorem~\ref{short time existence} implies for all $t\in[0,\overline{T}]$, $x\in[0,1]$ and $i\in\{1,2,3\}$,
\begin{equation*}
\gamma^i(t,x)=\gamma_0^i\left(t,\psi^i(t,x)\right)\,.
\end{equation*}
This proves that $\gamma(t,[0,1])$ and $\mathbb{T}(t)$ coincide for every $t\in[0,T_0]$. The claim follows repeating the same argument for the family of networks $\left(\widetilde{\mathbb{T}}(t)\right)$.
\end{proof}

\begin{teo}[Geometric Uniqueness]\label{geometricuniqueness}
	Let $p\in(5,\infty)$, $\mathbb{T}_0$ be a geometrically admissible initial triod and $T$ be positive. Suppose that $\left(\mathbb{T}(t)\right)$ and $\left(\widetilde{\mathbb{T}}(t)\right)$ are two solutions to~\eqref{Triod0} in the time interval $[0,T]$ with initial network $\mathbb{T}_0$ in the sense of Definition~\ref{geometricsolution1}. Then the networks $\mathbb{T}(t)$ and $\widetilde{\mathbb{T}}(t)$ coincide for all $t\in[0,T]$.
\end{teo}
\begin{proof}
	We prove this statement by contradiction. Suppose that the set \begin{equation*}
	\mathcal{C}:=\left\{t\in (0,T): \mathbb{T}(t)\neq\widetilde{\mathbb{T}}(t)\right\}
	\end{equation*}
	is non empty and let $t^*:=\inf \mathcal{C}$. Then $t^*\in[0,T)$ and as $\mathcal{C}$ is an open subset of $(0,T)$ we conclude that $\mathbb{T}(t^*)=\widetilde{\mathbb{T}}(t^*)$. As $\left(\mathbb{T}(t)\right)$ is a solution to~\eqref{Triod0} in the time interval $[0,T)$, the triod $\mathbb{T}\left(t^*\right)$ is a geometrically admissible initial network to system~\eqref{Triod0}. The two evolutions $\left(\mathbb{T}(t^*+t)\right)$ and $\left(\widetilde{\mathbb{T}}(t^*+t)\right)$ are solutions to~\eqref{Triod0} in the time interval $[0,T-t^*)$ in the sense of Definition~\ref{geometricsolution1} with the same initial network. Theorem~\ref{geometricuniqueness1} implies that there exists a time $T_0\in [0,T-t^*)$ such that for all $t\in[0,T_0]$,
	$\mathbb{T}(t^*+t)=\widetilde{\mathbb{T}}(t^*+t)$. This contradicts the fact that $t^*=\inf\left\{t\in (0,T): \mathbb{T}(t)\neq\widetilde{\mathbb{T}}(t)\right\}$.
\end{proof}

\subsection{Geometric existence and uniqueness of maximal solutions}

\begin{dfnz}[Maximal solution]\label{maximalsolution}
	Let $T\in(0,\infty)\cup\{\infty\}$, $p\in(5,\infty)$ and $\mathbb{T}_0$ be a geometrically admissible initial network.
	A time--dependent family of triods $\left(\mathbb{T}_t\right)_{t\in[0,T)}$
	is a maximal solution to
	the  elastic flow with initial datum $\mathbb{T}_0$ in $[0,T)$
	if it is a 
	%(in the sense of Definition~\ref{geometricsolution1} ) 
	smooth solution in the sense of Definition~\ref{smoothsol} in $(0,\hat{T}]$ for all $\hat{T}<T$ 
	and if there does not exist a smooth solution $\left(\widetilde{\mathbb{T}}(\tau)\right)$ in $(0,\widetilde{T}]$ with $\widetilde{T}\geq T$ and such that $\mathbb{T}=\widetilde{\mathbb{T}}$
		in $(0,T)$.
	In this case the time $T$ is called maximal time of existence and is denoted by $T_{max}$.
\end{dfnz}
\begin{rem}
	If $T=\infty$ in the above definition, $\widetilde{T}\geq T$ is supposed to mean $\widetilde{T}=\infty$.
\end{rem}
\begin{lemma}\label{geomuniqunesssmooth}
	Let $p\in(5,\infty)$, $\mathbb{T}_0$ be a geometrically admissible initial triod and $T$ be positive. Suppose that $\left(\mathbb{T}(t)\right)$ and $\left(\widetilde{\mathbb{T}}(t)\right)$ are smooth solutions in the sense of Definition~\ref{smoothsol} in $(0,T]$ for some positive $T$ with initial datum $\mathbb{T}_0$. Then the networks $\mathbb{T}(t)$ and $\widetilde{\mathbb{T}}(t)$ coincide for all $t\in[0,T]$.
\end{lemma}
\begin{proof}
	By the Definition~\ref{smoothsol} of smooth solution it is ensured that both $\left(\mathbb{T}(t)\right)$ and $\left(\widetilde{\mathbb{T}}(t)\right)$ are solutions to~\eqref{Triod0} in the time interval $[0,T]$ with initial network $\mathbb{T}_0$. This is due to the embedding
	\begin{equation*}
	 C^{\infty}(\overline{I_n}\times[0,1];\mathbb{R}^2)\hookrightarrow  W^1_p(I_n;L_p((0,1);\mathbb{R}^2))\cap L_p(I_n;W^4_p((0,1);\mathbb{R}^2))\,.
	\end{equation*}
	The result now follows from Theorem~\ref{geometricuniqueness}.
\end{proof}
\begin{lemma}[Existence and uniqueness of a maximal solution]\label{exuniquenssmaxsol}
	Let $p\in(5,\infty)$ and $\mathbb{T}_0$ be a geometrically admissible initial network. 
	There exists a 
	maximal solution $\left(\mathbb{T}(t)\right)_{t\in[0,T_{\max})}$ to
	the  elastic flow with initial datum $\mathbb{T}_0$ 
	in the maximal time interval $[0,T_{\max})$ with $T_{\max}\in(0,\infty)\cup\{\infty\}$. It is geometrically unique on finite time intervals in the sense of Lemma~\ref{geomuniqunesssmooth}.
\end{lemma}
\begin{proof}
	Combining Theorem~\ref{short time existence}
	and Theorem~\ref{analyticsmoothsol}  we know that 
	there exists a solution $\gamma\in\boldsymbol{E}_T$ of the Analytic Problem for some positive $T$ which is smooth in 
	$[\varepsilon,T)\times[0,1]$ for all $\varepsilon\in(0,T)$.
	This induces a smooth solution $\left(\mathbb{T}(t)\right)$ to the elastic flow in $(0,T]$ in the sense of Definition~\ref{smoothsol} via $\mathbb{T}(t):=\bigcup_{i=1}^3\gamma^i\left([0,T]\right)$.
	The existence of a maximal solution can be obtained using the Lemma of Zorn.
\end{proof}

%\begin{lemma}
%Suppose $\mathbb{T}(t)_{t\in[0,T_{\max})}$ is a maximal solution
%to the $C^0$--Flow with geometrically admissible initial datum $\mathbb{T}_0$ 
%in the maximal time interval $[0,T_{\max})$.
%Then $\mathbb{T}$ is smooth (in the sense of Remark~\ref{smoothsol}).
%\end{lemma}
%\begin{proof}
%Let 
%\begin{equation*}
%\mathcal{C}:=\left\lbrace t\in (0,T_{\max})\,:\,
%\mathbb{T}\;\text{is a smooth solution in} \;(0,t)\right\rbrace\,, 
%\end{equation*}
%and $t^*:=\inf \mathcal{C}$. 
%We notice that $t^*>0$. 
%Indeed 
%
%Suppose now  by contradiction that $t^*<T_{\max}$. 
%Consider $\varepsilon>0$.
%Integrating~\eqref{estimate}
%between $\varepsilon$ and $t^*$
% we get
%\begin{equation}
%\int_{\mathbb{T}_t^*}\vert\partial^2_s k\vert^{2}\,\mathrm{d}s\leq 
%\int_{\mathbb{T}_{\varepsilon}}\vert\partial^2_s k\vert^{2}\mathrm{d}s
%+C(t^*-\varepsilon)
%\leq C\,.
%\end{equation}
%\end{proof}

\begin{lemma}\label{onepara}
	Let $p\in(5,\infty)$ and $\mathbb{T}_0$ be a geometrically admissible initial network and
	$\left(\mathbb{T}(t)\right)_{t\in[0,T_{\max})}$ a maximal solution
	to the elastic flow with initial datum $\mathbb{T}_0$ 
	in the maximal time interval $[0,T_{\max})$ with $T_{max}\in (0,\infty)\cup\{\infty\}$.
	Then for all $T\in(0,T_{max})$ the evolution $\mathbb{T}$ admits a regular parametrisation
	$\gamma:=(\gamma^1,\gamma^2,\gamma^3)$ in $[0,T]$ that is smooth in $[\varepsilon,T]\times[0,1]$ for all $\varepsilon>0$.
\end{lemma}
\begin{proof}
	Let $T\in(0,T_{max})$ be given.
	As $\left(\mathbb{T}(t)\right)$ is a solution to~\eqref{Triod0} in $[0,T]$, the lengths $\ell^i(t)$ of the curves $\mathbb{T}^i(t)$, $t\in[0,T]$, $i\in\{1,2,3\}$ are uniformly bounded from above and below. Furthermore, for every $\varepsilon>0$ the map $t\mapsto \ell^i(t)$ is smooth on $[\varepsilon,T]$.
	Suppose that $(0,T)=(a_0,b_0)\cup (a_1,b_1)$
	with $0=a_0<a_1<b_0\leq b_1=T_{\max}$ and 
	that $\gamma_0$ and $\gamma_1$
	are regular parametrisations of $\mathbb{T}_t$
	in $(0,b_0)$ and $(a_1,T)$, respectively, as described in Definition~\ref{smoothsol}. In particular, it holds for all $\varepsilon\in(0,b_0)$
	\begin{equation*}
		\gamma^i_0\in  W^1_p((0,b_0);L_p((0,1);\mathbb{R}^2))\cap L_p((0,b_0);W^4_p((0,1);\mathbb{R}^2))\cap C^\infty\left([\varepsilon,b_0]\times[0,1];\mathbb{R}^2\right)
	\end{equation*}
	and
	\begin{equation*}
	\gamma^i_1\in C^\infty\left([a_1,T]\times[0,1];\mathbb{R}^2\right)\,.
	\end{equation*}
	Given $t\in[0,T]$, $j\in\{0,1\}$ and $i\in\{1,2,3\}$ we consider the reparametrisation
	\begin{equation*}
	\sigma^i_j(t):[0,1]\to[0,1]\,,\qquad \sigma^i_j(t)(x):=\sigma^i_j(t,x):=\frac{1}{\ell^i(t)}\int_0^x\left\vert(\gamma^i_j)_\xi(t,\xi)\right\vert\,\mathrm{d}\xi\,.
	\end{equation*}
	For $t\in[0,b_0]$ and $x\in[0,1]$ we define $\widetilde{\gamma}^i_0(t,x):=\gamma^i_0(t,(\sigma^i_0)^{-1}(t,x))$. Analogously, for $t\in[a_1,T]$ and $x\in[0,1]$ we let $\widetilde{\gamma}^i_1(t,x):=\gamma^i_1(t,(\sigma^i_1)^{-1}(t,x))$. Observe that for $t\in[a_1,b_0]$ both $\widetilde{\gamma}^i_0(t)$ and $\widetilde{\gamma}^i_1(t)$ are parametrisations of the curve $\mathbb{T}^i(t)$ with the same speed $\ell^i(t)$. This allows us to conclude for all $t\in[a_1,b_0]$:
	\begin{equation*}
	\widetilde{\gamma}^i_0(t)=\widetilde{\gamma}^i_1(t)\,.
	\end{equation*}
	The desired regular parametrisation $\gamma=\left(\gamma^1,\gamma^2,\gamma^3\right)$ can thus be defined as
	\begin{equation*}
	\gamma^i(t):=\begin{cases}
	\widetilde{\gamma}^i_0(t)\,,\qquad t\in[0,b_0]\,,\\
	\widetilde{\gamma}^i_1(t)\,,\qquad t\in[a_1,T]\,,
	\end{cases}
	\end{equation*}
	which is well defined and smooth in $[\varepsilon,T]\times[0,1]$ for all $\varepsilon>0$. In the case that the interval $(0,T)$ is an arbitrary finite union of intervals $(a_n,b_n)$, the proof follows using this procedure on every intersection.
\end{proof}

%%%%%%%%%%%%%%%%%%%%%%%%%%%%%%%%%%%%%%%%%%%%%%%%%%
%%%%%%%%%%%%%%%%%%%%%%%%%%%%%%%%%%%%%%%%%%%%%%%%%%
\section{Evolution of geometric quantities}\label{bounds}
%%%%%%%%%%%%%%%%%%%%%%%%%%%%%%%%%%%%%%%%%%%%%%%%%%
%%%%%%%%%%%%%%%%%%%%%%%%%%%%%%%%%%%%%%%%%%%%%%%%%%

The aim of this section is to find \emph{a priori} estimates 
for \emph{geometric} quantities related to the flow. Let $p\in(5,\infty)$ and $\mathbb{T}_0$ be a geometrically admissible initial network.
We consider a maximal solution $\left(\mathbb{T}(t)\right)_{t\in [0,T)}$ to the elastic flow starting in $\mathbb{T}_0$ in the maximal time interval $[0,T_{max})$ with $T_{max}\in(0,\infty)\cup\{\infty\}$. Notice that by Lemma~\ref{exuniquenssmaxsol} such a solution exists and is unique on finite time intervals. Furthermore, Lemma~\ref{onepara} implies that in every finite time interval $[0,T]\subset [0,T_{max})$ it can be parametrised by one parametrisation that is smooth away from $t=0$. Thus the arclength parameter is smooth on $[\varepsilon,T]\times[0,1]$ for all $T\in(0,T_{max})$ and all $\varepsilon\in(0,T)$.
Hence all the geometric quantities involved in the 
following computations are \emph{smooth functions}
depending on the time variable $t\in(0,T_{max})$ and the space variable
$s$ (the arclength parameter).

%%%%%%%%%%%%%%%%%%%%%%%%%%%%%%%%%%%%%%%%%%%%%%%%%%
\subsection{Notation and preliminaries}
%%%%%%%%%%%%%%%%%%%%%%%%%%%%%%%%%%%%%%%%%%%%%%%%%%

We introduce here some notation (the same as defined in~\cite{mannovtor}) which will be helpful in the following arguments. 

\begin{dfnz}
We denote by $\pol_\sigma(\ders^h k)$ a  polynomial in $k,\dots,\ders^h k$ with constant 
coefficients in $\mathbb{R}$ such that every monomial it contains is of the form
\begin{equation*}
C \prod_{l=0}^h	(\ders^lk)^{\beta_l}\quad\text{ with} \quad \sum_{l=0}^h(l+1)\beta_l = \sigma\,,
\end{equation*}
where $\beta_l\in\mathbb{N}_0$ for $l\in\{0,\dots,h-1\}$ and $\beta_h\in\mathbb{N}$ 
for at least one monomial.
\end{dfnz}

\begin{dfnz}
We denote by $\mathfrak{q}_\sigma\left(\dert^jT,\ders^hk\right)$ a polynomial in 
$T,\dots,\dert^jT$, $k,\dots,\ders^hk$ with constant coefficients in $\mathbb{R}$ 
such that every monomial it contains is of the form
\begin{equation*}
C \prod_{l=0}^j\left(\dert^lT\right)^{\alpha_l}\prod_{l=0}^h\left(\ders^lk\right)^{\beta_l}\quad
\text{ with} \quad \sum_{l=0}^{j}(4l+1)\alpha_l+\sum_{l=0}^h(l+1)\beta_l = \sigma\,,
\end{equation*}
with $\alpha_l\in\mathbb{N}_0$ for $l\in\{0,\dots,j\}$, $\beta_l\in\mathbb{N}_0$ 
for $l\in\{0,\dots,h\}$. We demand that there are (possibly different) monomials 
that satisfy $\alpha_j\in\mathbb{N}$ and $\beta_h\in\mathbb{N}$, respectively.
\end{dfnz}

\begin{dfnz}
We write $\mathfrak{q}_\sigma (\vert \dert^jT\vert,\vert\ders^hk\vert )$ 
to denote a finite sum of terms of the form
\begin{equation*}
C \prod_{l=0}^j\left\lvert\dert^lT\right\rvert^{\alpha_l}
\prod_{l=0}^h\left\lvert\ders^lk\right\rvert^{\beta_l}\quad
\text{ with} \quad \sum_{l=0}^{j}(4l+1)\alpha_l+\sum_{l=0}^h(l+1)\beta_l = \sigma\,,
\end{equation*}
with $\alpha_l\in\mathbb{N}_0$ for $l\in\{0,\dots,j\}$, $\beta_l\in\mathbb{N}_0$ f
or $l\in\{0,\dots,h\}$. 
Again we demand that there are (possibly different) monomials 
that satisfy $\alpha_j\in\mathbb{N}$ and $\beta_h\in\mathbb{N}$, respectively.
The polynomials $\pol_\sigma(\vert\ders^h k\vert)$  are defined in the same manner.
\end{dfnz}

We notice that 
\begin{align}
\partial_s\left(\pol_\sigma(\ders^h k)\right)&=\pol_{\sigma+1}(\ders^{h+1} k)\,,\nonumber\\
\mathfrak{p}_{\sigma_1}(\partial_s^{h_1}k)\mathfrak{p}_{\sigma_2}
(\partial_s^{h_2}k)&=\mathfrak{p}_{\sigma_1+\sigma_2}(\partial_s^{\max\{h_1,h_2\}}k)\,.
\label{calcpol}
\end{align}

\medskip

\emph{Young's inequality}\\
We will use Young's inequality in the following form:
\begin{equation}
ab\leq \varepsilon a^p+ C(\varepsilon,p,p')b^{p'}
  \end{equation}
with $a,b,\varepsilon>0$, $p,p'\in(1,\infty)$ and $\frac{1}{p}+\frac{1}{p'}=1$.

\medskip
\medskip

We adopt the following convention
to calculate the evolution of a certain geometric quantity $f=(f^1,f^2,f^3)$ 
integrated along the network $\mathbb{T}$ composed of the curves $\mathbb{T}^i$:
\begin{equation*}
\int_{\mathbb{T}_t} f\mathrm{d}s:=\sum_{i=1}^{3}\int_{\mathbb{T}_t^i}^{}f^i\,\mathrm{d}s^i\,.
\end{equation*}
We remind that 
the arclength parameter $s$ varies in $[0,\mathrm{L}(\mathbb{T}^i)]$ with $L\left(\mathbb{T}^i\right)$ denoting the length of the curve $\mathbb{T}^i$
and that
the curves $\mathbb{T}^i$  can be also parametrised by
functions $\gamma^i$ defined on the fixed interval $[0,1]$.
Then 
\begin{equation*}
\sum_{i=1}^{3}\int_{\mathbb{T}_t^i}^{}f^i\,\mathrm{d}s^i
=\sum_{i=1}^3 \int_0^1f\,\vert\gamma^i_x\vert\,\mathrm{d}x\,,
\end{equation*}
as the arclength measure
is given by $\mathrm{d}s=\vert\gamma^i_x\vert\,\mathrm{d}x$ 
on the curve parametrised by $\gamma^i$.\\

\medskip

\emph{$L^p$--norms}\\
In the sequel we will prove that the length of each curve of a triod $\mathbb{T}_t$
evolving by the elastic flow is bounded (see Remark~\ref{boundL})
and we will require that it is also uniformly bounded away from zero.
That is 
\begin{equation}
0<c\leq \mathrm{L}(\mathbb{T}^i_t)\leq C<\infty\,.
\end{equation}

Hence for any
fixed  time $t\in (0,T)$ the interval $[0,\mathrm{L}(\mathbb{T}^i_t)]$ is positive and bounded.
For all $p\in [1,\infty)$ we will write
\begin{equation*}
\lVert f^i\rVert_{L^p(0,\mathrm{L}(\mathbb{T}^i_t))}^p:=\int_{\mathbb{T}^i}\lvert  
f^i\rvert^p\mathrm{d}s\,\quad \text{and}\quad
\lVert f\rVert_{L^p(\mathbb{T}_t)}:=\sum_{i=1}^{3}\lVert f^i\rVert_{L^p(0,\mathrm{L}(\mathbb{T}^i_t))}\,.
\end{equation*}
We will also use the $L^\infty$--norm
\begin{equation*}
\lVert f^i\rVert_{L^\infty(0,\mathrm{L}(\mathbb{T}^i_t))}:=
\mathrm{ess\; sup}_{L^\infty(0,\mathrm{L}(\mathbb{T}^i_t))} \,\vert f^i\vert
\,\quad \text{and}\quad
\lVert f\rVert_{L^\infty(\mathbb{T}_t)}:=\sum_{i=1}^{3}\lVert f^i\rVert_{L^\infty(0,\mathrm{L}(\mathbb{T}^i_t))}\,.
\end{equation*}
Whenever we are considering continuous functions, we identify the supremum norm with the $L^\infty$ norm and denote it by $\left\lVert\cdot\right\rVert_\infty$.

\medskip

We underline here that for sake of notation we will simply write $\Vert\cdot\Vert_{L^p}$
instead of $\lVert \cdot \rVert_{L^p(0,\mathrm{L}(\mathbb{T}^i_t))}$
both for $p\in (0,\infty)$ and $p=\infty$.

\medskip

\emph{Gagliardo--Nirenberg Inequality}\\
Let $\eta$ be a smooth regular curve in $\mathbb{R}^2$ 
with finite length $\mathrm{L}$ and let $u$ be a smooth function defined on $\eta$. 
Then for every $j\geq 1$, $p\in [2,\infty]$ and $n\in\{0,\dots,j-1\}$ we have the estimates
\begin{equation*}
\lVert \partial_s^nu\rVert_{L^p}\leq C_{n,j,p}\lVert \partial_s^ju\rVert_{L^2}^\sigma\lVert u\rVert_{L^2}^{1-\sigma}+\frac{B_{n,j,p}}{\mathrm{L}^{j\sigma}}\lVert u\rVert_{L^2}
\end{equation*}
where 
\begin{equation*}
\sigma=\frac{n+1/2-1/p}{j}
\end{equation*}
and the constants $C_{n,j,p}$ and $B_{n,j,p}$ are independent of $\eta$.
In particular, if  $p=+\infty$, 
\begin{equation}\label{int2}     
{\Vert\partial_s^n u\Vert}_{L^\infty}     \leq C_{n,m}  
{\Vert\partial_s^m  u\Vert}_{L^2}^{\sigma}       
{\Vert u\Vert}_{L^2}^{1-\sigma}+      
 \frac{B_{n,m}}{{\mathrm L}^{m\sigma}}{\Vert         
u\Vert}_{L^2}\qquad\text{ { with} }\quad \text{ $\sigma=\frac{n+1/2}{m}$.}  
\end{equation}

We notice that in the case of a family of curves with length equibounded from below 
by some positive value, the Gagliardo--Nirenberg inequality holds with uniform constants.

%%%%%%%%%%%%%%%%%%%%%%%%%%%%%%%%%%%%%%%%%%%%%%%%%%
\subsection{Basic evolution formulas}
%%%%%%%%%%%%%%%%%%%%%%%%%%%%%%%%%%%%%%%%%%%%%%%%%%

\begin{lemma}[Commutation rules]\label{commrule}
If $\gamma$ moves according to~\eqref{motion} the commutation rule
\begin{equation*}
\partial_{t}\partial_{s}=\partial_{s}\partial_{t}+\left(T_{s}-kA\right)\partial_{s}\,
\end{equation*}
holds. The measure $\mathrm{d}s$ evolves as
\begin{equation*}
\partial_t(\mathrm{d}s)=\left(kA-T_s\right)\mathrm{d}s\,.
\end{equation*}
\end{lemma}
\begin{proof}
The proof follows by straightforward computations.
\end{proof}

\begin{lemma}\label{basicformulae}
The tangent, unit normal and curvature of a curve moving by~\eqref{motion} satisfy
\begin{align}
\partial_{t}\tau&=-\left(A_{s}+T k\right)\nu\,,\nonumber\\
\partial_{t}\nu&=\left(A_{s}+T k\right)\tau\,,\nonumber\\
\partial_tk&=\left\langle \partial_{t}\boldsymbol{\kappa},\nu\right\rangle 
=-A_{ss}-T\partial_sk-k^{2}A\,\nonumber\\
&=-2\partial_{s}^{4}k-5k^{2}\partial_{s}^{2}k-6k\left(\partial_{s}k\right)^{2}
-T\partial_{s}k-k^{5}+\mu \left( \partial_{s}^{2}k+k^3\right)\,.\label{kt}
\end{align}
Moreover  if $\gamma$ moves according to~\eqref{motion} 
the following formula holds for any $j\in\mathbb{N}$:
\begin{align}
\partial_{t}\partial_{s}^{j}k=&-2\partial_{s}^{j+4}k	
+\mathfrak{p}_{j+5}\left(\partial_{s}^{j+2}k\right)-T\partial_{s}^{j+1}k+\mu\,\mathfrak{p}_{j+3}
(\partial_s^{j+2}k)\label{derivativekt}
\end{align}
\end{lemma}
\begin{proof}
The proof follows by direct computations.
\end{proof}

%%%%%%%%%%%%%%%%%%%%%%%%%%%%%%%%%%%%%%%%%%%%%%%%%%
\subsection{Bounds on curvature and length} 
%%%%%%%%%%%%%%%%%%%%%%%%%%%%%%%%%%%%%%%%%%%%%%%%%%

\begin{lemma}\label{decreasingenergy}
For every $t\in (0,T_{max})$ it holds
\begin{align*}
\frac{\mathrm{d}}{\mathrm{d}t}\int_{\mathbb{T}(t)}k^{2}+\mu\,
\mathrm{d}s&=-\int_{\mathbb{T}(t)}^{} A^2\,\mathrm{d}s\,\qquad\text{and}\qquad 
E_{\mu}(\mathbb{T}(t))\leq E_{\mu}(\mathbb{T}_0)\,.
\end{align*}
\end{lemma}
\begin{proof}
This result follows from
the gradient flow structure, see~\cite{bargarnu}.
\end{proof}

\begin{rem}\label{boundL}
As a consequence for every $t\in (0,T_{max})$ and for every $\mu> 0$
\begin{equation}\label{boundlength}
\int_{\mathbb{T}(t)} k^2\,\mathrm{d}s\leq E_{\mu} (\mathbb{T}_0)\,,\qquad\text{and}
\qquad 
L(\mathbb{T}(t))\leq \frac{1}{\mu} E_{\mu} (\mathbb{T}_0)<\infty\,.
\end{equation}

Notice that the global length of the evolving triod is bounded
from below away from zero by the value of the length
shortest path connecting the three points 
$P^1,P^2$ and $P^3$. 
Unfortunately this does not give a bound on the length of the single curve, 
the length of (at most) one curve 
can go to zero during the evolution.
\end{rem}

%%%%%%%%%%%%%%%%%%%%%%%%%%%%%%%%%%%%%%%%%%%%%%%%%%
\subsection{Bound on $\partial_s^2k$}\label{kssbounded}
%%%%%%%%%%%%%%%%%%%%%%%%%%%%%%%%%%%%%%%%%%%%%%%%%%

\begin{lemma}\label{espressioneduedevk}
For $t\in (0,T_{max})$ it holds
\begin{align*}
\frac{\mathrm{d}}{\mathrm{d}t}\int_{\mathbb{T}_t}\vert\partial^2_s k\vert^{2}\,\mathrm{d}s
&=\int _{\mathbb{T}_t}
 -\vert 2\partial_s^4k \vert^2 -2\mu\vert\partial_s^3k \vert^2 +\mathfrak{p}_{10}\left(\partial_{s}^{3}k\right)
 +\mathfrak{p}_{8}\left(\partial_{s}^{2}k\right)\,\mathrm{d}s\\
&+\sum_{i=1}^3\left. 
\mathfrak{q}_7\left(T^i,\partial_s^3k^i\right)
+\mathfrak{q}_5\left(T^i,\partial_sk^i\right)
%+\mathfrak{p}_9\left(\partial_s^3k^i\right)
+\mathfrak{p}_5\left(\partial_s^2k^i\right)
\right|_{3\text{--point}}\,,
\end{align*}
where $T^i$ is appearing in the polynomials with power $1$.
\end{lemma}
\begin{proof}
By direct computation 
\begin{align*}
&\frac{\mathrm{d}}{\mathrm{d}t}\int_{\mathbb{T}_t}\vert\partial_s^2 k\vert ^{2}\,\mathrm{d}s\\
&=\int_{\mathbb{T}_t}\partial_s^2k\{2\partial_t\partial_s^2k
+\partial_s^2k\left(kA-T_s\right) \}\,\mathrm{d}s\\
&=\int_{\mathbb{T}_t}\partial_s^2k\left\lbrace -4\partial_s^6k
-10k^2\partial_s^4 k+2\mu\,\partial_s^4k-36\left(\partial_sk\right)^2\partial_s^2k
+\mathfrak{p}_7(\partial_s^3 k)+\mathfrak{p}_5(\partial_s^2 k)\right.\\
& \left.
-2T\partial_s^3 k-T_s\partial^2_s k\right\rbrace\,\mathrm{d}s\,.
\end{align*}
Integrating by parts once the term
$\int \partial^2_s k\left(-10k^2\partial_s^4 k
-36\left(\partial_sk\right)^2\partial_s^2k+2\mu\, \partial_s^4k\right) \,\mathrm{d}s$
and twice
$\int -4\partial^2_s k \partial_s^6k  \,\mathrm{d}s$
we get
\begin{align*}
\frac{\mathrm{d}}{\mathrm{d}t}&\int_{\mathbb{T}_t}\vert\partial^2_s k\vert^{2}\,\mathrm{d}s
=  \int _{\mathbb{T}_t}
 - \vert 2\partial_s^4k \vert^2 
 -2\mu\vert\partial_s^3k \vert^2
 +\mathfrak{p}_{10}\left(\partial_{s}^{3}k\right)
 +\mathfrak{p}_{8}\left(\partial_{s}^{2}k\right)
\,\mathrm{d}s\\
+& \left.\left\lbrace  
4 \partial^3_s k^i\partial_s^4k^i 
-4\partial^2_s k^i \partial_s^5k^i 
-12\left(\partial_sk^i\right)^3\partial^2_sk^i
-(10(k^i)^2-2\mu)\partial^2_s k^i \partial_s^3 k-T^i\left(\partial_s^2 k^i\right)^2
\right\rbrace \right|_{bdry}\,.\nonumber
\end{align*}
We focus now on the boundary terms.
It is easy to see that at the fixed end--points 
the contribution is zero.
Indeed the curvature is zero, the velocity is zero (hence the second derivative of the curvature
is zero and $T^i$ is zero) 
and using~\eqref{kt} one notices that also the fourth derivative of the curvature is zero.
Hence (using the fact that $k^i(0)=0$) it remains to  deal with
\begin{align}\label{bondarypart}
 \sum_{i=1}^3 
 \left( 4\partial^3_s k^i\partial_s^4k^i 
  -4 \partial^2_s k^i\partial_s^5k^i
    -12\left(\partial_sk^i\right)^3 \partial^2_sk^i
+2\mu\,\partial^2_s k^i \partial_s^3 k^i
-T^i\left(\partial_s^2 k^i\right)^2\right)\,,
\end{align}
where for sake of notation we omitted the dependence on $x=0$.
Differentiating in time the curvature condition $k^i=0$ for $i\in\{1,2,3\}$ at the triple junction we get
\begin{align*}
0&= \partial_t k^i
=-2\partial_{s}^{4}k^i-5(k^i)^{2}\partial_{s}^{2}k^i-6k^i\left(\partial_{s}k^i\right)^{2}
	-T^i\partial_{s}k^i-\left(k^i\right)^{5}+\mu \left( \partial_{s}^{2}k^i+\left(k^i\right)^3\right)\\
&=-2\partial_{s}^{4}k^i
	-T^i\partial_{s}k^i+\mu  \partial_{s}^{2}k^i\,.
\end{align*}
Thus, at the triple junction we have
\begin{equation}\label{consequencecurvcond}
\sum_{i=1}^3
2  \partial_t k^i \partial^3_s k^i
=\sum_{i=1}^3
\left(4\partial^3_s k^i\partial_{s}^{4}k^i
+2T^i\partial_{s}k^i\partial^3_s k^i-2\mu\, \partial_{s}^{2}k^i\partial^3_s k^i\right) =0\,.
\end{equation}
Moreover,
differentiating in time both the concurrency condition
and the  third order condition  at the 
triple junction, we get
$A^i\nu^i+T^i\tau^i=A^j\nu^j+T^j\tau^j$  and 
$\partial_t( \sum 2\partial_s k^i\nu^i-\mu \tau^i)=0$ for $i,j\in\{1,2,3\}$ which implies
\begin{align*}
0&=\left\langle A^1\nu^1+T^1\tau^1\,,\, \partial_t  \left(\sum_{i=1}^3
2\partial_s k^i\nu^i-\mu \tau^i\right)\right\rangle\\
&=\sum_{i=1}^3 \left\langle A^i\nu^i+T^i\tau^i\,,\, 
 \partial_t(2\partial_s k^i\nu^i-\mu \tau^i)\right\rangle\\
&= \sum_{i=1}^3 \left\langle A^i\nu^i+T^i\tau^i , 
(  2\partial_t\partial_sk^i+\mu(A_s^i+T^ik^i))\nu^i
+2\partial_sk^i(A_s^i+T^ik^i)\tau^i
\right\rangle\\
&=\sum_{i=1}^3\left( \left(2\partial_t\partial_s k^i
+\mu A^i_s\right)A^i+ 2\partial_sk^i A^i_sT^i\right)\\
&=\sum_{i=1}^3
2\left(-4\partial_s^5 k^i
-12\left(\partial_s k^i\right)^3
-2T^i\partial_s^2k^i+4\mu\,\partial_s^3k^i-\mu^2\partial_sk^i\right)\partial_s^2k^i
+4T^i\partial_sk^i\partial_s^3k^i-2\mu\,T^i(\partial_sk^i)^2 \,,
\end{align*}
that is 
\begin{align*}
\sum_{i=1}^3
&\left(-4\partial_s^2k^i\partial_s^5 k^i
-12\left(\partial_s k^i\right)^3\partial_s^2k^i+4\mu\, \partial_s^2k^i\partial_s^3k^i 
-2T^i(\partial_s^2k^i)^2\right.\\
&\left.-\mu^2\partial_sk^i\partial_s^2k^i+2T^i\partial_sk^i\partial_s^3k^i-\mu T^i(\partial_sk^i)^2\right)=0\,.
\end{align*}
Combined with~\eqref{consequencecurvcond} this yields
\begin{align*}
\sum_{i=1}^3
&\left(4\partial^3_s k^i\partial_{s}^{4}k^i-4\partial_s^2k^i\partial_s^5 k^i
-12\left(\partial_s k^i\right)^3\partial_s^2k^i
+2\mu\, \partial_s^3k^i \partial_s^2k^i
-2T^i(\partial_s^2k^i)^2\right.\\
&\left.-\mu^2\partial_sk^i\partial_s^2k^i+4T^i\partial_sk^i\partial_s^3k^i-\mu T^i(\partial_sk^i)^2\right)
=0\,.
\end{align*}
Hence we can express the sum~\eqref{bondarypart} as
\begin{align*}
\sum_{i=1}^3
\left(T^i(\partial_s^2k^i)^2
+\mu^2\partial_sk^i\partial_s^2k^i-4T^i\partial_sk^i\partial_s^3k^i+\mu T^i(\partial_sk^i)^2\right)\,.
\end{align*}
Combined with the previous computations this gives the desired result.
\end{proof}

We have obtained an explicit expression for 
$\frac{\mathrm{d}}{\mathrm{d}t}\Vert\partial_s^2k\Vert_2$.
Differently from $\frac{\mathrm{d}}{\mathrm{d}t}\Vert k\Vert_2$,
the sign of the expression is not clearly determined, hence we cannot
easily say that $\Vert\partial_s^2k\Vert_2$ is decreasing during the evolution.
Our aim is to estimate the polynomials involved in the formula in order to get
at least a bound for $\Vert\partial_s^2k\Vert_2$.

We underline that from now on the constant $C$ may vary from line to line. 

\begin{lemma}\label{boundint}
Let 
$\int_{\mathbb{T}_t}^{}\lvert\mathfrak{p}_{10}\left(\partial_s^{3}k\right)\rvert
+\lvert\mathfrak{p}_{8}\left(\partial_s^{2}k\right)\rvert\mathrm{d}s$
be the integral of the two polynomials appearing in Lemma~\ref{espressioneduedevk}.
Suppose that the lengths of the three curves of the triod $\mathbb{T}_t$ are
uniformly  bounded
away from zero for all $t\in [0,T_{max})$. 
Then the following estimates hold
for all $t\in (0,T_{max})$:
\begin{align*}
\int_{\mathbb{T}_t}^{} \mathfrak{p}_{10}\left(\partial_s^{3}k\right)\mathrm{d}s
\leq \lVert\partial_s^{4}k\rVert_{L^2}^2+C\lVert k\rVert_{L^2}^2+C\lVert k\rVert^{18}_{L^2}\,,\\
\int_{\mathbb{T}_t}^{}\mathfrak{p}_{8}\left(\partial_s^{2}k\right)\mathrm{d}s
\leq \frac{\mu}{2}\lVert\partial_s^{3}k\rVert_{L^2}^2
+C\lVert k\rVert_{L^2}^2+C\lVert k\rVert^{14}_{L^2}\,.
\end{align*}
\end{lemma}
\begin{rem}
The constants $1$ and $\mu/2$ in front of $\lVert\partial_s^{4}k\rVert_{L^2}^2$
and $\lVert\partial_s^{3}k\rVert_{L^2}^2$ will play a special role later in   
Lemma~\ref{estimatetogether}. In this Lemma they can be chosen arbitrarily small.
\end{rem}
\begin{proof}
To obtain the desired estimates we adapt~\cite[pag 260--261]{mannovtor} to our situation.

Let $m\in\{1,2\}$.
Every monomial of $\mathfrak{p}_{2m+6}\left(\partial_s^{m+1}k\right)$ is of the shape
\begin{equation*}
C\prod_{l=0}^{m+1}\left(\partial_s^lk\right)^{\alpha_l} 
\end{equation*}
with $\alpha_l\in\mathbb{N}_0$ and $\sum_{l=0}^{m+1}\alpha_l(l+1)=2m+6$. 
We define $J:=\{l\in\{0,\dots,m+1\}:\alpha_l\neq 0\}$
and for every $l\in J$  we set
$$
\beta_l:=\frac{2m+6}{(l+1)\alpha_l}\,.
$$
We observe that $\sum_{l\in J}^{}\frac{1}{\beta_l}=1$ and $\alpha_l\beta_l>2$
for every $l\in\{0,\ldots,j\}$ such that $\alpha_l\neq 0$. Thus the H\"older inequality implies
\begin{equation*}
C\int_{\mathbb{T}_t}\prod_{l\in J}^{} (\partial_s^lk)^{\alpha_l}\mathrm{d}s\leq C\prod_{l\in J}^{}\left(\int_{\mathbb{T}_t}\lvert\partial_s^lk\rvert^{\alpha_l\beta_l}\mathrm{d}s\right)^{\frac{1}{\beta_l}}=C\prod_{l\in J}^{}\lVert\partial_s^lk\rVert^{\alpha_l}_{L^{\alpha_l\beta_l}}\,.
\end{equation*}
Applying the Gagliardo--Nirenberg inequality for every $l\in J$ yields for every $i\in\{1,2,3\}$
\begin{equation*}
\lVert\partial_s^lk^i\rVert_{L^{\alpha_l\beta_l}}\leq C_{l,m,\alpha_l,\beta_l}\lVert\partial_s^{m+2}k^i\rVert_{L^2}^{\sigma_l}\lVert k^i\rVert_{L^2}^{1-\sigma_l}+\frac{B_{l,m,\alpha_l,\beta_l}}{L^{(m+2)\sigma_l}}\lVert k^i\rVert_{L^2}\,
\end{equation*}
where for all $l\in J$ the coefficient $\sigma_l$ is given by
$$
\sigma_l=\frac{l+1/2-1/(\alpha_l\beta_l)}{m+2}\,.
$$
We may choose 
$$
C=\max\left\{C_{l,m,\alpha_l,\beta_l},\frac{B_{l,m,\alpha_l,\beta_l}}{L^{(m+2)\sigma_l}}:l\in J \right\}\,.
$$ 
Since 
the polynomial $\mathfrak{p}_{2m+6}\left(\partial_s^{m+1}k\right)$
consists of finitely many monomials (whose number depends on $m$)
of the above type with coefficients independent 
of time and the points on the curve, we can write
\begin{align*}
\int_{\mathbb{T}_t}\mathfrak{p}_{2m+6}\left(\partial_s^{m+1}k\right)\,\mathrm{d}s
&\leq C\int_{\mathbb{T}_t}^{}\prod_{l\in J}^{}\lvert\partial_s^lk\rvert^{\alpha_l}\mathrm{d}s
\leq C\prod_{l\in J }^{}\lVert\partial_s^lk\rVert^{\alpha_l}_{L^{\alpha_l\beta_l}}\\
&\leq C(L_t)\prod_{l\in J}^{}\lVert k\rVert^{(1-\sigma_l)\alpha_l}_{L^2}\left(\lVert\partial_s^{m+2}k\rVert_{L^2}+\lVert k\rVert_{L^2}\right)^{\sigma_l\alpha_l}_{L^2}\\
&= C(L_t)\lVert k\rVert^{\sum_{l\in J}(1-\sigma_l)\alpha_l}_{L^2}\left(\lVert\partial_s^{m+2}k\rVert_{L^2}+\lVert k\rVert_{L^2}\right)^{\sum_{l\in J}\sigma_l\alpha_l}_{L^2}
\end{align*}
for every $t\in(0,T_{max})$ such that the flow exists. Here the constant $C(L_t)$ depends on the lengths of \emph{each} curve at time $t$. Moreover we have
\begin{align*}
\sum_{l\in J}\sigma_l\alpha_l&
%%=\sum_{l\in J}\sigma_l
%=\frac{1}{m+2}\sum_{l\in J}\alpha_l(l+1/2)-\frac{1}{\beta_l}=\frac{1}{m+2}\left(\sum_{l\in J}\alpha_l(l+1)-\sum_{l\in J}\frac{1}{2} \alpha_l -1\right)\\
%&=\frac{1}{m+2}\left( 2m+5-\sum_{l\in J}\frac{1}{2} \alpha_l\right)\leq \frac{2m+5}{m+2}-\frac{1}{2(m+2)}\sum_{l\in J}^{}\alpha_l\frac{l+1}{m+2}\\
%&=\frac{2m+5}{m+2}-\frac{1}{2(m+2)^2}\sum_{l\in J}^{}\alpha_l(l+1)=\frac{2m+5}{m+2}-\frac{2m+6}{2(m+2)^2}=\frac{2m+5}{m+2}-\frac{m+3}{\left(m+2\right)^2}\\
%&=\frac{2m+4}{m+2}-\frac{1}{(m+2)^2}
=2-\frac{1}{(m+2)^2}<2\,.
\end{align*}
Applying Young's inequality with $p:=\frac{2}{\sum_{l\in J}\sigma_l\alpha_l}$ and $q:=\frac{2}{2-\sum_{l\in J}^{}\sigma_l\alpha_l}$ we obtain
\begin{align*}
C\int_{\mathbb{T}_t}^{}\prod_{l\in J}^{}\lvert\partial_s^lk\rvert^{\alpha_l}\mathrm{d}s&\leq \frac{C(L_t)}{\varepsilon}\lVert k\rVert^{2\frac{\sum_{l\in J}(1-\sigma_l)\alpha_l}{2-\sum_{l\in J}^{}\sigma_l\alpha_l}}_{L^2}+\varepsilon C(L_t) \left(\lVert\partial_s^{m+2}k\rVert_{L^2}+\lVert k\rVert_{L^2}\right)^{2}_{L^2}
\end{align*}
where
\begin{align*}
2\frac{\sum_{l\in J}(1-\sigma_l)\alpha_l}{2-\sum_{l\in J}^{}\sigma_l\alpha_l}&
%=2\frac{\sum_{l\in J}\alpha_l-\frac{1}{m+2}\sum_{l\in J}\left(\alpha_l(l+1)-\frac{1}{2}\alpha_l -\frac{1}{\beta_l}\right)}{2-\frac{1}{m+2}\sum_{l\in J}\left(\alpha_l(l+1)-\frac{1}{2}\alpha_l -\frac{1}{\beta_l}\right)}\\
%&=2\frac{(m+2)\sum_{l\in J}\alpha_l-(2m+5)+\frac{1}{2}\sum_{l\in J}\alpha_l}{2(m+2)-2m-5+\frac{1}{2}\sum_{l\in J}\alpha_l}\\
%&=\frac{\left(2m+5\right)\left(\frac{1}{2}\sum_{l\in J}\alpha_l-1\right)}{\frac{1}{2}\sum_{l\in J}\alpha_l-1}
=2(2m+5)\,.
\end{align*}
As $C(L_t)$ depends only on $m$ and the length of each curve of the solution at time $t$ and as the single lengths are bounded from below by hypothesis, we get choosing $\varepsilon$ small enough 
\begin{align*}
\int_{\mathbb{T}_t}\mathfrak{p}_{2m+6}\left(\partial_s^{m+1}k\right)\,\mathrm{d}s
&\leq\varepsilon \left(\lVert\partial_s^{m+2}k\rVert_{L^2}+\lVert k\rVert_{L^2}\right)^{2}_{L^2}+\frac{C}{\varepsilon}\lVert k\rVert^{2(2m+5)}_{L^2}
%\\
%&\leq \frac{1}{4}\lVert\partial_s^{m+2}k\rVert_{L^2}^2+C\lVert k\rVert_{L^2}^2+C\lVert k\rVert^{2(2m+5)}_{L^2}
\,.
\end{align*}
To conclude in our case it is enough to take  $m\in\{1,2\}$ and 
choose a suitable $\varepsilon>0$.
\end{proof}
In the following Lemma we will express the tangential velocity at the triple junction in terms of the normal velocity similarly as in~\cite{garckenov}.
\begin{lemma}\label{tangentialvelocityboundary}
	Given $t\in[0,T_{max})$ we let $\tau^i(t):=\tau^i(t,0)$ be the unit tangent vector to the curve $\mathbb{T}^i(t)$ at the triple junction and $\alpha^1(t)$, $\alpha^2(t),\alpha^3(t)$ be the angle at the triple junction
	between$\tau^2(t)$ and  $\tau^3(t)$, $\tau^3(t)$ and $\tau^1(t)$, and $\tau^1(t)$ and  
	$\tau^2(t)$, respectively. Suppose that there exists $\rho>0$ such that
	\begin{equation}\label{nondegeneracy}
	\inf_{t\in[0,T_{max})}\max\left\{\left\vert\sin\alpha^1(t)\right\vert,\left\vert\sin\alpha^2(t)\right\vert,\left\vert\sin\alpha^3(t)\right\vert\right\}\geq \rho\,.
	\end{equation}
	Then for every $t\in[0,T_{max})$ the tangential velocities $T^i(t):=T^i(t,0)$ at the triple junction are linear combinations of the normal velocities $A^i(t):=A^i(t,0)$ with coefficients uniformly bounded in time.
\end{lemma}
\begin{rem}
	The above condition~\eqref{nondegeneracy} means that for all $t\in[0,T_{max})$ the network $\mathbb{T}(t)$ is non degenerate in the sense that $\mathrm{span}\left\{\nu^1(t,0),\nu^2(t,0),\nu^3(t,0)\right\}=\mathbb{R}^2$. Notice that this condition appears in the Definition~\ref{adm} of geometrically admissible initial networks as it is needed to prove the validity of the Lopatinskii--Shapiro condition, see~\cite[Lemma 3.14]{garmenzplu}. We will refer to~\eqref{nondegeneracy} as the\textit{ uniform non--degeneracy condition}.
\end{rem}
\begin{proof}
	Given $t\in[0,T_{max})$ differentiating the concurrency condition in time yields at the triple junction
	\begin{equation*}
	A^1(t)\nu^1(t)+T^1(t)\tau^1(t)=A^2(t)\nu^2(t)+T^2(t)\tau^2(t)=A^3(t)\nu^3(t)+T^3(t)\tau^3(t)\,.
	\end{equation*}
	Testing these identities with $\tau^1(t)$, $\tau^2(t)$, $\tau^3(t)$ implies
	\begin{equation*}
	\begin{pmatrix}
	-\left\langle\tau^1(t),\tau^2(t)\right\rangle& 1 &0\\
	0 & - \left\langle \tau^2(t),\tau^3(t)\right\rangle &1 \\
	1 & 0 &-\left\langle \tau^3(t),\tau^1(t)\right\rangle  &
	\end{pmatrix}
	\begin{pmatrix}
	T_1(t) \\ T_2(t) \\ T_3(t)
	\end{pmatrix}
	=\begin{pmatrix}
	\left\langle \nu^1 (t),\tau^2(t)\right\rangle A_1(t) \\ \left\langle \nu^2(t), \tau^3(t)\right\rangle A_2(t) \\ \left\langle \nu^3(t),\tau^1 (t)\right\rangle A_3(t)\,.
	\end{pmatrix}
	\end{equation*}
	The $3\times 3$--matrix on the left hand side will be denoted by $\mathcal{M}(t)$ in the following. Its determinant is given by
	\begin{equation*}
	\det\mathcal{M}(t)=1- \left\langle\tau^2(t),\tau^1(t)\right\rangle \left\langle \tau^3(t),\tau^2(t)\right\rangle \left\langle \tau^1(t),\tau^3(t)\right\rangle\,.
	\end{equation*}
	By Cramer's rule each component $T^i(t)$ of the unique solution $T(t)$ of the above system can be expressed as a linear combination of $A^1(t)$, $A^2(t)$, $A^3(t)$ with coefficients that are polynomials in the entries of $\mathcal{M}(t)$, $\left\langle \nu^1 (t),\tau^2(t)\right\rangle $, $\left\langle \nu^2(t), \tau^3(t)\right\rangle $, $\left\langle \nu^3(t),\tau^1 (t)\right\rangle$ and $\left(\det\mathcal{M}(t)\right)^{-1}$. The condition~\eqref{nondegeneracy} ensures that these coefficients are uniformly bounded in $[0,T_{max})$. Indeed, notice that
	\begin{equation*}
	\inf_{t\in[0,T_{max})}\det \mathcal{M}(t)= 1- \cos\left(\alpha^1(t)\right)\cos\left(\alpha^2(t)\right)\cos\left(\alpha^3(t)\right)\geq 1-\sqrt{1-\rho^2}>0\,.
	\end{equation*}
\end{proof}
\begin{lemma}\label{boundboundary} 
	Let 
\begin{align*}
\sum_{i=1}^3\left.\left( 
\mathfrak{q}_7\left(T^i,\partial_s^3k^i\right)
+\mathfrak{q}_5\left(T^i,\partial_sk^i\right)
+\mathfrak{p}_5\left(\partial_s^2k^i\right)
\right)\right|_{3\text{--point}}
\end{align*}
be the boundary terms appearing in Lemma~\ref{espressioneduedevk}.
Suppose that the length of the three curves of the triod $\mathbb{T}_t$
are uniformly bounded away from zero for all $t\in(0,T_{max})$.
Moreover suppose that the uniform non--degeneracy condition~\eqref{nondegeneracy} is satisfied.
Then the following estimates hold
for all $t\in (0,T_{max})$:
\begin{align*}
\sum_{i=1}^3\left. 
\mathfrak{q}_7\left(T^i,\partial_s^3k^i\right)\right|_{3\text{--point}}
&\leq \lVert\partial_s^{4}k\rVert_{L^2}^2+C\lVert k\rVert_{L^2}^2+C\lVert k\rVert^{18}_{L^2}\,,\\
\sum_{i=1}^3\left. \mathfrak{q}_5\left(T^i,\partial_sk^i\right)\right|_{3\text{--point}}
&\leq \frac{\mu}{4}\lVert\partial_s^{3}k\rVert_{L^2}^2+C\lVert k\rVert_{L^2}^2+C\lVert k\rVert^{14}_{L^2}\,,\\
\sum_{i=1}^3\left.  \mathfrak{p}_5\left(\partial_s^2k^i\right)\right|_{3\text{--point}}
&\leq \frac{\mu}{4}\lVert\partial_s^{3}k\rVert_{L^2}^2+C\lVert k\rVert_{L^2}^2+C\lVert k\rVert^{a_*}_{L^2}\,.
\end{align*}
\end{lemma}
\begin{rem}
As before 
the constants $1$ and $\mu/4$ in front of $\lVert\partial_s^{4}k\rVert_{L^2}^2$
and $\lVert\partial_s^{3}k\rVert_{L^2}^2$  in this Lemma can be chosen arbitrarily small.
\end{rem}
\begin{proof}
Since $T^i$ is appearing with power one in the polynomials 
$\mathfrak{q}_{7}(T^i,\partial_s^3 k^i)$ and $\mathfrak{q}_{5}(T^i,\partial_s k^i)$,
we can write
\begin{equation*}
\mathfrak{q}_{7}(T^i,\partial_s^3 k^i)
=T^i\left(\mathfrak{p}_{6}(\partial_s^3 k^i)\right)
\qquad\text{and}\qquad
\mathfrak{q}_{5}(T^i,\partial_s k^i)
=T^i\left(\mathfrak{p}_{4}(\partial_s k^i)\right)\,.
\end{equation*}
By Lemma~\ref{tangentialvelocityboundary} for every $t\in (0,T_{max})$ and $i\in\{1,2,3\}$ it holds at the triple junction
\begin{align*}
\vert T^i\vert
&\leq 
C \left(\vert A^1\vert
+\vert A^2\vert
+\vert A^3\vert
\right)=C \left(\vert \partial_s^2k^{1}\vert
+\vert \partial_s^2k^{2}\vert
+\vert \partial_s^2k^{3}\vert
\right)
\leq C\sum_{i=1}^3\Vert \partial_s^2k^i\Vert_{\infty}\\
&\leq C\Vert \partial_s^2k\Vert_{\infty} 
= C \Vert \partial_s^2k\Vert_{L^\infty}\,,
\end{align*}
with a constant $C$ independent of $t$ where the equality $\Vert \partial_s^2k\Vert_{\infty}=\Vert \partial_s^2k\Vert_{L^\infty}$
holds because all $\partial_s^2k^i$
are smooth functions.
Hence 
$\sum_{i=1}^3  \left. \mathfrak{q}_{7}(T^i,\partial_s^3 k^i)\right|_{3\text{--point}}
\leq \sum_{i=1}^3\left. \vert T^i\vert
\left(\mathfrak{p}_{6}(\vert\partial_s^3 k^i\vert)\right)\right|_{3\text{--point}}$
can be controlled with a sum of terms like 
$C\prod_{l=0}^{3}\Vert\ders^{l}k\Vert_{L^\infty}^{\alpha_l}$
with $\sum_{l=0}^3(l+1)\alpha_l=9$.
Similarly
$\sum_{i=1}^3  \left. \mathfrak{q}_{5}(T^i,\partial_s k^i)\right|_{3\text{--point}}$
can be controlled by a sum of terms of type
$C\prod_{l=0}^{2}\Vert\ders^{l}k\Vert_{L^\infty}^{\alpha_l}$
with 
$\sum_{l=0}^2(l+1)\alpha_l=7$
and also
$\sum_{i=1}^3 \left.  \mathfrak{p}_{j}(\partial_s^m k)\right|_{3\text{--point}}$
can be controlled
by 
$C\prod_{l=0}^{2}\Vert\ders^{l}k\Vert_{L^\infty}^{\alpha_l}$
with  $\sum_{l=0}^2(l+1)\alpha_l=5$.

Again we follow~\cite[pag 261--262]{mannovtor}.
%We introduce the notation
%$\mathfrak{p}_k(\partial_s^m k)$ with the couple $(k,m)\in\{(9,3),(7,2),(5,2)\}$.
We use interpolation inequalities with $p=+\infty$,
\begin{equation}
\Vert\ders^lk\Vert_{L^\infty}\leq C_l\left(
    {\Vert\partial_s^{m+1}  k\Vert}_{L^2}^{\sigma_l}
    {\Vert k\Vert}_{L^2}^{1-\sigma_l}+
    {\Vert k\Vert}_{L^2}\right)
\end{equation}
with $\sigma_l=\frac{l+1/2}{m+1}$, hence
\begin{align*}
C\prod_{l=0}^{m}\Vert\ders^{l}k\Vert_{L^\infty}^{\alpha_l}
\leq&\, C \prod_{l=0}^{m}\left(
    {\Vert\partial_s^{m+1}  k\Vert}_{L^2}
    + {\Vert k\Vert}_{L^2}\right)^{\sigma_l\alpha_l}
{\Vert k\Vert}_{L^2}^{(1-\sigma_l)\alpha_l}\\
\leq&\, C \left(    {\Vert\partial_s^{m+1}  k\Vert}_{L^2}
    + {\Vert k\Vert}_{L^2}\right)^{\sum_{l=0}^j\sigma_l\alpha_l}
{\Vert k\Vert}_{L^2}^{\sum_{l=0}^j(1-\sigma_l)\alpha_l}
\end{align*}
and
\begin{align*}
\sum_{l=0}^m\sigma_l\alpha_l
%&=\,\sum_{l=0}^m\alpha_l\frac{l+1/2}{m+1}
%=\,\sum_{l=0}^{m}\frac{(l+1)\alpha_l}{m+1} - \frac12 \sum_{l=0}^{m}\frac{\alpha_l}{m+1}\\
%&\leq \,\sum_{l=0}^{m}\frac{(l+1)\alpha_l}{m+1} 
%- \frac{1}{2(m+1)}\sum_{l=0}^m\frac{\alpha_l(l+1)}{m+1}
%=\left(\frac{2m+1}{2m+2}\right)\sum_{l=0}^{m}\frac{(l+1)\alpha_l}{m+1} 
<2\,,
\end{align*}
as it takes the values $\frac{63}{32}$, $\frac{35}{18}$ and $\frac{25}{18}$.
By Young's inequality
$$
\left(    {\Vert\partial_s^{4}  k\Vert}_{L^2}
    + {\Vert k\Vert}_{L^2}\right)^{\sum_{l=0}^3\sigma_l\alpha_l}
{\Vert k\Vert}_{L^2}^{\sum_{l=0}^3(1-\sigma_l)\alpha_l}
\leq \varepsilon \left(    {\Vert\partial_s^{4}  k\Vert}_{L^2}
    + {\Vert k\Vert}_{L^2}\right)^2
+ C {\Vert  k\Vert}_{L^2}^{a_*}
$$
with 
$$
a_*=
2\frac{\sum_{l=0}^3(1-\sigma_l)\alpha_l}
{2-\sum_{l=0}^3\sigma_l\alpha_l}\,.
$$
%When $(k.m)=(9,3)$ and $(7,2)$, then  $a_*$ is equal to $18$ and $14$ respectively. 
For a suitable choice of 
 $\varepsilon>0$ we get the result.
\end{proof}

\begin{prop}\label{estimatetogether}
Let $\left(\mathbb{T}(t)\right)$ be a maximal solution to the elastic flow with initial datum $\mathbb{T}_0$ in the maximal time interval $[0,T_{max})$ with $T_{max}\in(0,\infty)\cup\{\infty\}$ and let $E_\mu(\mathbb{T}_0)$ be the elastic energy
of the initial network. 
Suppose that for $t\in(0,T_{max})$ the lengths of the three curves of the triod $\mathbb{T}_t$
are uniformly bounded away from zero and 
that the uniform non--degeneracy condition~\eqref{nondegeneracy} is satisfied.
Then for all $t\in (0,T_{max})$ it holds
\begin{equation}\label{estimate}
\frac{\mathrm{d}}{\mathrm{d}t}\int_{\mathbb{T}_t}
\left\lvert\partial_s^2k\right\rvert^2 \,\mathrm{d}s
\leq C(E_\mu(\mathbb{T}_0))\,.
\end{equation}
\end{prop}
\begin{proof}
Combining Lemma~\ref{espressioneduedevk},~\ref{boundint}
and~\ref{boundboundary} with the
bound on the curvature we get
\begin{align*}
\frac{\mathrm{d}}{\mathrm{d}t}\int_{\mathbb{T}_t}\vert\partial^2_s k\vert^{2}\,\mathrm{d}s
&\leq \int _{\mathbb{T}_t}
 -2\vert \partial_s^4k \vert^2 -\mu\vert\partial_s^3k \vert^2 \,\mathrm{d}s
+ C(E_\mu(\mathbb{T}_0))
 \leq C(E_\mu(\mathbb{T}_0))\,.
\end{align*}
\end{proof}

\section{Long time behaviour}

\subsection{Long time behaviour of the elastic flow of triods}

\begin{teo}\label{longtimetriod}
Let $p\in(5,10)$ and $\mathbb{T}_0$ be a geometrically admissible initial network. Suppose that $\left(\mathbb{T}(t)\right)_{t\in[0,T_{\max})}$ is a 
maximal solution to
the elastic flow with initial datum $\mathbb{T}_0$ 
in the maximal time interval $[0,T_{\max})$ with $T_{\max}\in (0,\infty)\cup\{\infty\}$ in the sense of Definitions~\ref{smoothsol} and~\ref{maximalsolution}. Then $$T_{max}=\infty$$ or at least one of the following happens:
\begin{itemize}
\item[(i)] the inferior limit of the length of one curve of $\mathbb{T}(t)$ is zero as $t\nearrow T_{max}$.
\item[(ii)] $\liminf_{t\nearrow T_{max}}\max\left\{\left\vert\sin\alpha^1(t)\right\vert,\left\vert\sin\alpha^2(t)\right\vert,\left\vert\sin\alpha^3(t)\right\vert\right\}=0$, where $\alpha^1(t)$, $\alpha^2(t)$ and $\alpha^3(t)$ are the angles at the triple junction.
\end{itemize}
\end{teo}

\begin{proof}
Let $\mathbb{T}$ be a maximal solution of the elastic flow
in $[0,T_{\max})$. Suppose that the two assertions $(i)$ and $(ii)$ are not fulfilled and that $T_{max}$ is finite. Then the lengths of the three curves of $\mathbb{T}(t)$ are uniformly bounded away from zero on $[0,T_{max})$ and the uniform non--degeneracy condition~\eqref{nondegeneracy} is satisfied.
Observe that by hypothesis, smoothness of the flow on $[\varepsilon,T]$ for all positive $\varepsilon$ and all $T\in(\varepsilon,T_{max})$ and the short time existence result, the lengths $\ell^i(t)$ of the curves $\mathbb{T}^i(t)$ are uniformly bounded from below on $[0,T_{max})$. 
Remark~\ref{boundL} implies further that $\ell^i(t)$ are uniformly bounded from above on $[0,T_{max})$ and that $k\in L^{\infty}((0,T_{\max});L^2(\mathbb{T}_t))$. Let $\varepsilon\in \left(0,\frac{T_{max}}{100}\right)$ and $\delta\in\left(0,\frac{T_{max}}{4}\right)$ be fixed. Integrating~\eqref{estimate} on the interval $\left(\varepsilon,T_{max}-\delta\right)$ gives
\begin{equation*}
\int_{\mathbb{T}_{T_{max}-\delta}}^{}\vert\partial^2_s k\vert^{2}\,\mathrm{d}s\leq \int_{\mathbb{T}_{\varepsilon}}^{}\vert\partial^2_s k\vert^{2}\,\mathrm{d}s+C\left(E_\mu\left(\mathbb{T}_0\right)\right)T_{max}
\end{equation*}
which implies $\partial_s^2k\in L^\infty\left((\varepsilon,T_{max}-\delta);L^2\left(\mathbb{T}_t\right)\right)$.

By interpolation we obtain $k$, $\partial sk\in L^\infty\left((\varepsilon,T_{max}-\delta);L^\infty\left(\mathbb{T}_t\right)\right)$.
Proposition~\ref{onepara} implies that the evolution $\left(\mathbb{T}(t)\right)$ can be parametrised by one map $\gamma=(\gamma^1,\gamma^2,\gamma^3)$ which is smooth on $[\varepsilon,T_{max}-\delta]\times[0,1]$. By construction of this map, $\vert\gamma^i_x(t,x)\vert=\ell^i(t)$ for all $x\in[0,1]$ and all $t\in[\varepsilon,T_{max}-\delta]$ which implies in particular
\begin{equation*}
0<c\leq\sup_{t\in[\varepsilon,T_{max}-\delta],x\in[0,1]}\vert\gamma^i(t,x)\vert \leq C<\infty\,.
\end{equation*}
By direct computation, we observe the following identities for the curvature vectors
\begin{equation}\label{choiceofpara}
\vec{k}^i(t,x)=\frac{\gamma^i_{xx}(t,x)}{\left(\ell^i(t)\right)^2}\,,\qquad
\vec{k}_s^i(t,x)=\frac{\gamma^i_{xxx}(t,x)}{\left(\ell^i(t)\right)^3}\,,\qquad
\vec{k}_{ss}^i(t,x)=\frac{\gamma^i_{xxxx}(t,x)}{\left(\ell^i(t)\right)^4}\,.
\end{equation}
Combining~\eqref{choiceofpara} with Remark~\ref{boundL}
we obtain for all $t\in[0,T_{max})$
\begin{equation*}
\int_{0}^1 \frac{\vert \gamma^i_{xx}(t,x) \vert^2}{(\ell^i(t))^3}\,\mathrm{d}x
=\int_{\mathbb{T}^i_t}\vert \vec{k}^i(t,s)\vert^2\,\mathrm{d}s
=\int_{\mathbb{T}^i_t} k^i(t,s) ^2\,\mathrm{d}s \leq E_{\mu}(\mathbb{T}_0)\,, 
\end{equation*}
hence 
\begin{equation*}
\sup_{t\in[\varepsilon,T_{\max}-\delta)}\int_0^1\vert \gamma^i_{xx}(t,x)\vert^2\,\mathrm{d}x\leq C^3 E_{\mu}(\mathbb{T}_0)<\infty\,.
\end{equation*}
As 
\begin{equation*}
\vec{k}_{ss}=\left(k_{ss}-k^3\right)\nu-3kk_s\tau
\end{equation*} the previous observations immediately give $\vec{k}_{ss}\in L^\infty\left((\varepsilon,T_{max}-\delta);L^2\left(\mathbb{T}_t\right)\right)$. This implies for every $t\in(\varepsilon,T_{max}-\delta)$ and a constant $\tilde{C}$ not depending on $\delta$
\begin{equation*}
\int_{0}^1 \frac{\vert \gamma^i_{xxxx}(t,x) \vert^2}{(\ell^i(t))^7}\,\mathrm{d}x
=\int_{\mathbb{T}^i_t}\vert \vec{k}_{ss}^i(t,s)\vert^2\,\mathrm{d}s
\leq \tilde{C}\,,
\end{equation*}
which yields
\begin{equation*}
\sup_{t\in(\varepsilon,T_{\max}-\delta)}\int_0^1\vert \gamma^i_{xxxx}(t,x) \vert^2\,\mathrm{d}x\leq \tilde{C}C^7<\infty\,.
\end{equation*}
As a consequence we obtain 
\begin{equation*}
\gamma^i_{xx}\in L^{\infty}((\varepsilon,T_{max}-\delta), L^2(0,1))\qquad\text{and}\qquad
\gamma^i_{xxxx}\in L^{\infty}((\varepsilon,T_{max}-\delta), L^2(0,1))\,.
\end{equation*}
A uniform bound in time and space on the third derivative of $\gamma^i$ can be obtained by interpolation. This bound is independent of $\delta$. Further, $\gamma^i_x\in L^\infty\left((\varepsilon,T_{max}-\delta);L^\infty((0,1))\right)$ as $\gamma^i$ is parametrised with constant speed equal to the length.
As one endpoint of each curve $\mathbb{T}^i$ is fixed during the evolution and as the lengths $\ell^i(t)$ of the curves are bounded from above uniformly in time on $[0,T_{max})$, the networks $\mathbb{T}(t)$ remain inside a ball $B_R(0)$ for all $t\in[0,T_{max})$ and a suitable choice of $R$. This allows us to conclude for all $\delta$ and $\varepsilon$ as above
\begin{equation*}
\gamma\in L^\infty\left((\varepsilon,T_{max}-\delta);W^4_2\left((0,1);(\mathbb{R}^2)^3\right)\right)
\end{equation*}
where the norm is bounded by a constant independent of $\delta$. The Sobolev Embedding Theorem implies for all $p\in(5,10)$
\begin{equation*}
\gamma\in L^\infty\left((\varepsilon,T_{max}-\delta);W^{4-\nicefrac{4}{p}}_p\left((0,1);(\mathbb{R}^2)^3\right)\right)
\end{equation*}
where the norm is bounded by a constant $\boldsymbol{C}$ not depending on $\delta$. Notice that $\gamma(T_{max}-\delta)$ is an admissible initial parametrisation for all $\delta\in\left(0,\frac{T_{max}}{4}\right)$ in the sense of Definition~\ref{adm}. By Theorem~\ref{short time existence} there exists a uniform time $\boldsymbol{T}$ of existence depending on $\boldsymbol{C}$ for all initial values $\gamma(T_{max}-\delta)$. Let $\delta:=\frac{\boldsymbol{T}}{2}$. Then Theorem~\ref{short time existence} implies the existence of a regular solution
\begin{equation*}
\eta\in W_p^1\left((T_{max}-\delta, T_{max}+\delta);L_p\left((0,1);(\mathbb{R}^2)^3\right)\right)\cap L_p\left((T_{max}-\delta, T_{max}+\delta);W_p^4\left((0,1);(\mathbb{R}^2)^3\right)\right)
\end{equation*}
to the system~\eqref{TriodC^0} with $\eta\left(T_{max}-\delta\right)=\gamma\left(T_{max}-\delta\right)$. By Theorem~\ref{analyticsmoothsol} we obtain 
\begin{equation*}
\eta\in C^\infty\left([T_{max}-\frac{\delta}{2},T_{max}+\delta]\times[0,1];(\mathbb{R}^2)^3\right)\,.
\end{equation*} 
The two parametrisations $\gamma$ and $\eta$ defined on $(0,T_{max}-\frac{\delta}{3})$ and $\left(T_{max}-\frac{\delta}{2},T_{max}+\delta\right)$, respectively, define a smooth solution $\left(\widetilde{\mathbb{T}}(t)\right)$ to the elastic flow on the time interval $(0,T_{max}+\delta]$ with initial datum $\mathbb{T}_0$ in the sense of Definition~\ref{smoothsol} coinciding with $\mathbb{T}$ on $(0,T_{max})$. This contradicts the maximality of $T_{max}$.
\end{proof}

%\begin{rem}
%	Notice that the above proof works for any initial network composed of curves that meet in triple junctions with at least one endpoint of the network fixed in the plane.
%\end{rem}

\subsection{A remark on the definition of maximal solutions}

The aim of this section is to show that the assumption of smoothness in  Definition~\ref{maximalsolution} is not needed.

\begin{dfnz}[Sobolev maximal solution]
Let $T\in(0,\infty)\cup\{\infty\}$, $p\in(5,\infty)$ and $\mathbb{T}_0$ 
be a geometrically admissible initial network.
A time--dependent family of triods $\left(\mathbb{T}_t\right)_{t\in[0,T)}$
is a Sobolev maximal solution to
the elastic flow with initial datum $\mathbb{T}_0$ in $[0,T)$
if it is a solution 
(in the sense of Definition~\ref{geometricsolution1}) 
%	smooth solution in the sense of Definition~\ref{smoothsol} 
in $(0,\hat{T}]$ for all $\hat{T}<T$ 
and if there does not exist a solution 
$\left(\widetilde{\mathbb{T}}(\tau)\right)$ in $(0,\widetilde{T}]$ with $\widetilde{T}\geq T$ 
and such that $\mathbb{T}=\widetilde{\mathbb{T}}$ in $(0,T)$.
\end{dfnz}

The existence and uniqueness of a Sobolev maximal solution is easy to prove.

\medskip

Suppose that $\mathbb{T}_0$ is a geometrically admissible initial network,
$\left(\mathbb{T}_t\right)_{t\in[0,T)}$ is a Sobolev maximal solution
to the elastic flow with initial datum $\mathbb{T}_0$ in $[0,T)$
and $(\widetilde{\mathbb{T}}_t)_{t\in[0,\widetilde{T})}$ is a maximal solution
to the elastic flow with initial datum $\mathbb{T}_0$ in $[0,\widetilde{T})$
in the sense of Definition~\ref{maximalsolution}.
Then Theorem~\ref{geometricuniqueness} implies that
$\mathbb{T}$ and $\widetilde{\mathbb{T}}$ coincide in $[0,\min\{T,\widetilde{T}\})$.

\medskip

A priori it is possible that $\widetilde{T}<T$ and hence 
$\left(\mathbb{T}_t\right)_{t\in[0,T)}$ is a Sobolev maximal solution which is smooth
until $\widetilde{T}$ but has a sudden loss of regularity for $t>\widetilde{T}$.
We show that this can not be the case.

Suppose by contradiction that $\widetilde{T}<T$ and 
\begin{equation}\label{lossofreg}
\widetilde{T}=\sup\{t\in (0,T)\;:\; \mathbb{T}_t\;\text{is smooth in the sense of 
Definition~\ref{smoothsol}}\,\}\,.
\end{equation}
Since $\mathbb{T}$ is a solution to the elastic flow in $[0,T)$ with $T>\widetilde{T}$,
for all $t\in[0,\widetilde{T}]$ the length of all curves of the triod are uniformly bounded
away from zero and the uniform non--degeneracy condition~\eqref{nondegeneracy} is fulfilled.
Moreover by Lemma~\ref{onepara},
$\mathbb{T}$ admits a regular smooth parametrisation
$\gamma$ in $[\varepsilon,\widetilde{T}-\delta]\times[0,1]$ for all $\varepsilon,\delta >0$.
We can hence apply the same arguments as in the proof of Theorem~\ref{longtimetriod}
to obtain a smooth extension of $\mathbb{T}$ in the time interval 
$[0,\widetilde{T}+\delta]$, a contradiction to~\eqref{lossofreg}. 

We summarize this result in the following:
\begin{lemma}
	Let $p\in(5,\infty)$ and $\mathbb{T}_0$ be a geometrically admissible initial network. 
	There exists a 
	maximal solution $\left(\mathbb{T}(t)\right)_{t\in[0,T_{\max})}$ to
	the  elastic flow with initial datum $\mathbb{T}_0$ 
	in the maximal time interval $[0,T_{\max})$ with $T_{\max}\in(0,\infty)\cup\{\infty\}$. It is
		smooth in the sense of Definition~\ref{maximalsolution},
		 geometrically unique on finite time intervals in the sense of Lemma~\ref{geomuniqunesssmooth},
and for all $T\in(0,T_{max})$ the evolution $\mathbb{T}$ admits a regular parametrisation
	$\gamma:=(\gamma^1,\gamma^2,\gamma^3)$ in $[0,T]$ that is smooth in $[\varepsilon,T]\times[0,1]$ for all $\varepsilon>0$.
\end{lemma}

\subsection{Long time behaviour of the elastic flow of Theta--networks}
In this section we show that all the above results hold true also in the case of Theta--networks
(see Definition~\ref{deftheta}).
%\begin{dfnz}
%	A Theta--network $\Theta$ is a network in $\mathbb{R}^2$ composed by three regular 
%	curves that intersect each other at their endpoints in two triple junctions.
%\end{dfnz}
It is straightforward to adapt the proofs of (geometric) short time existence and geometric uniqueness to the case of Theta--networks. The calculations that were done to treat the boundary terms at the triple junction of the triod are precisely the ones needed for both triple junctions of the Theta. In particular the elastic flow of Theta--networks satisfies the a priori estimates. The difficulty lies in the proof of the long time existence result. In contrast to the elastic flow of triods no points of the Theta--network are fixed during the evolution. The presence of fixed endpoints was used in Theorem~\ref{longtimetriod} to find a uniform in time and space $L^\infty$--bound on the parametrisations. As these arguments are no longer possible in the Theta--case, we prove a refinement of the short time existence result, namely that the time interval within which the Analytic Problem is well posed does not depend on the $L^p$--norm of the initial parametrisation (see Theorem~\ref{shorttimethetarefined}).

Analogously to the elastic flow of triods we use the following notion of geometric solution.

\begin{dfnz}[Elastic flow of a Theta--network]\label{geometricsolution1theta}
	Let $p\in (5,\infty)$ and $T>0$.
	Let $\Theta_0$ be a geometrically admissible initial Theta--network.
	A time dependent family of Theta--networks $\left(\Theta(t)\right)$ is a solution to
	the elastic flow with initial datum $\Theta_0$ 
	in  $[0,T]$ 
	if and only if there exists a collection of time dependent
	parametrisations
	\begin{equation*}
	\gamma^i_n\in  W^1_p(I_n;L_p((0,1);\mathbb{R}^2))\cap L_p(I_n;W^4_p((0,1);\mathbb{R}^2))\,,
	\end{equation*}
	with $n\in\{0,\dots, N\}$ for some $N\in\mathbb{N}$, $I_n:=(a_n,b_n)\subset \mathbb{R}$, $a_n\leq a_{n+1}$, $b_n\leq b_{n+1}$, $a_n<b_n$
	and $\bigcup_n (a_n,b_n)=(0,T)$ such that for all $n\in\{0,\dots,N\}$ and $t\in I_n$, $\gamma_n(t)=\left(\gamma^1(t),\gamma^2(t),\gamma^3(t)\right)$ is a regular parametrisation of $\Theta(t)$.
	Moreover each 
	$\gamma_n$
	needs to satisfy the following system
	\begin{equation}\label{Thetaelastic}
	\begin{cases}
	\begin{array}{lll}
	\left(\left\langle \gamma_t^{i}, \nu^i\right\rangle  \nu^i\right)(t,x)
	=-\left((2k_{ss}^{i}+\left(k^{i}\right)^{3}-\mu k^{i})\nu^{i}\right)(t,x)& &\text{motion,}\\
	\gamma^{1}\left(t,y\right)=\gamma^{2}\left(t,y\right)=\gamma^{3}\left(t,y\right)&
	&\text{concurrency condition,}\\
	k^{i}(t,y)=0 & &\text{curvature condition,}\\
	\sum_{i=1}^{3}\left(2k_{s}^{i}\nu^{i}-\mu\tau^{i}\right)(t,y)=0& &\text{third order condition,}\\
	\end{array}
	\end{cases}
	\end{equation}
	for every  $t\in  I_n,x\in\left(0,1\right)$, $y\in\{0,1\}$ and for $i\in\{1,2,3\}$. Finally we ask that $\gamma_n(0,[0,1])=\Theta_0$ whenever $a_n=0$. 
	
	The family $\left(\Theta(t)\right)$ is a \textit{smooth} solution to the elastic flow with initial datum $\Theta_0$ in $(0,T]$ if there exists a collection $\gamma_n$, $n\in\{0,\dots,N\}$, satisfying all requirements as above such that additionally $a_1>0$, $\gamma_n^i\in C^{\infty}(\overline{I_n}\times[0,1];\mathbb{R}^2)$ for all $n\in\{1,\dots,N\}$ and for all $\varepsilon\in(0,b_0)$, $\gamma^i_0\in C^\infty\left([\varepsilon,b_0]\times[0,1];\mathbb{R}^2\right)$.
\end{dfnz}

\begin{dfnz}\label{admtheta}
	Let $p\in(5,\infty)$. A Theta--network $\Theta_0$ is a geometrically admissible initial network for 
	system~\eqref{Thetaelastic}  if
	\begin{itemize}
		\item[-] there exists a parametrisation $\sigma=(\sigma^1,\sigma^2,\sigma^3)$ 
		of $\Theta_0$ 
		such that every curve $\sigma^i$ is regular and
		\begin{equation*}
		\sigma^i\in W_p^{4-\nicefrac{4}{p}}((0,1);\mathbb{R}^2)\,.
		\end{equation*}	
		\item[-] at both triple junctions $k^i_0=0$ ,
		$\sum_{i=1}^{3}\left(2k_{0,s}^{i}\nu_0^{i}-\mu\tau^{i}_0\right)=0$, and 
		$\mathrm{span}\{\nu^1_0,\nu^2_0,\nu^3_0\}=\mathbb{R}^2$.
	\end{itemize}
\end{dfnz}
As in the case of triods we transform the geometric problem~\eqref{Thetaelastic} to a parabolic quasilinear system of PDEs.

\begin{dfnz}\label{admptheta}
	Let $p\in(5,\infty)$. A parametrisation $\varphi=(\varphi^1,\varphi^2,\varphi^3)$
	of an initial Theta--network $\Theta_0$
	is an admissible initial parametrisation for system~\eqref{thetapde} if each curve $\varphi^i$ is regular, $\varphi^i\in W_p^{4-\nicefrac{4}{p}}((0,1);\mathbb{R}^2)$ and for $y\in\{0,1\}$ the parametrisations $\varphi^i$
	satisfy the concurrency, second and third order condition
	and $\mathrm{span}\{\nu^1_0(y),\nu^2_0(y),\nu^3_0(y)\}=\mathbb{R}^2$.
\end{dfnz}

\begin{dfnz}\label{thetanalytic}
	Let $T>0$ and $p\in(5,\infty)$. Given an admissible initial parametrisation 
	$\varphi$
	the time dependent parametrisation $\gamma=(\gamma^1,\gamma^2,\gamma^3)\in\boldsymbol{E}_{{T}}$
	is a solution of the Analytic Problem for Theta--networks with initial value $\varphi$ in $[0,T]$ if
	the curve $\gamma^i(t)$ is regular for all $t\in[0,T]$ and the following system is satisfied 
	for almost every $t\in\left(0,T\right),x\in\left(0,1\right)$ and for $i\in\{1,2,3\}$, $y\in\{0,1\}$:
	\begin{equation}\label{thetapde}
	\begin{cases}
	\begin{array}{lll}
	\gamma_t^{i}(t,x)=-A^{i}(t,x)\nu^{i}(t,x)-T^{i}(t,x)\tau^{i}(t,x)& &\text{motion,}\\
	\gamma^{1}\left(t,y\right)=\gamma^{2}\left(t,y\right)=\gamma^{3}\left(t,y\right)&
	&\text{concurrency condition,}\\
	\gamma_{xx}^{i}(t,y)=0 & &\text{second order condition,}\\
	\sum_{i=1}^{3}\left(2k_{s}^{i}\nu^{i}-\mu\tau^{i}\right)(t,y)=0 &  &
	\text{third order condition,}\\
	\gamma^i(0,x)=\varphi^i(x)&  &
	\text{initial condition}\,,
	\end{array}
	\end{cases}
	\end{equation}
	where the tangential velocity is defined in~\eqref{Tang}.
\end{dfnz}
Repeating precisely the same arguments as in \S~\ref{shorttimeexistence} we obtain the following result.
\begin{teo}\label{shorttimetheta}
	Let $p\in(5,\infty)$ and $\varphi$ be an admissible initial parametrisation. There exists a positive time $\widetilde{T}\left(\varphi\right)$ depending on $\min_{i\in\{1,2,3\},x\in[0,1]}\vert\varphi^i_x(x)\vert$ and $\left\lVert\varphi\right\rVert_{W_p^{4-\nicefrac{4}{p}}\left((0,1)\right)}$ such that for all $\boldsymbol{T}\in (0,\widetilde{T}(\varphi)]$ the system~\eqref{thetanalytic} has a solution in 
	\begin{equation*}
	\boldsymbol{E}_{\boldsymbol{T}}=W_p^{1}\left((0,\boldsymbol{T});L_p\left((0,1);(\mathbb{R}^2)^3\right)\right)\cap L_p\left((0,\boldsymbol{T});W_p^4\left((0,1);(\mathbb{R}^2)^3\right)\right)
	\end{equation*}
	which is unique in $\boldsymbol{E}_{\boldsymbol{T}}\cap\overline{B_M}$ where
	\begin{equation*}
	M:=2\max\left\{\sup_{T\in(0,1]}\vertiii{L_T^{-1}}_{\mathcal{L}\left(\mathbb{F}_T,\mathbb{E}_T\right)},1\right\}\max\left\{\vertiii{\mathcal{E}\varphi}_{\mathbb{E}_1},\vertiii{\left(N_{1,1}(\mathcal{E}\varphi),N_{1,2}(\mathcal{E}\varphi),\varphi\right)}_{\mathbb{F}_1}\right\}\,.
	\end{equation*}
\end{teo}
To prove a refinement of the above theorem we introduce the following notation.
\begin{dfnz}
	Given $p\in(5,\infty)$ and $\eta\in W_p^{4-\nicefrac{4}{p}}\left((0,1);(\mathbb{R}^2)^3\right)$ we let
	\begin{equation*}
	\left\lvert\eta\right\rvert_{W_p^{4-\nicefrac{4}{p}}((0,1);(\mathbb{R}^2)^3)}:=\sum_{j=1}^3\left\lVert\partial_x^j\eta\right\rVert_{L_p\left((0,1);(\mathbb{R}^2)^3\right)}+\left[\partial_x^3\eta\right]_{1-\frac{4}{p},p}\,.
	\end{equation*}
\end{dfnz}
We will now show that the existence time depends on $\left\lVert\varphi\right\rVert_{W_p^{4-\nicefrac{4}{p}}((0,1))}$ only via $\left\lvert\varphi\right\rvert_{W_p^{4-\nicefrac{4}{p}}((0,1))}$. This is due to the fact that the problem for the Theta--network is translationally invariant.
\begin{teo}\label{shorttimethetarefined}
		Let $p\in(5,\infty)$ and $\varphi$ be an admissible initial parametrisation. There exists a time $T\left(\varphi\right)\in(0,1]$ depending on $\min_{i\in\{1,2,3\},x\in[0,1]}\vert\varphi^i_x(x)\vert$ and $\left\vert\varphi\right\vert_{W_p^{4-\nicefrac{4}{p}}\left((0,1)\right)}$ such that for all $\boldsymbol{T}\in (0,T(\varphi)]$ the system~\eqref{thetanalytic} has a solution in 
		\begin{equation*}
		\boldsymbol{E}_{\boldsymbol{T}}=W_p^{1}\left((0,\boldsymbol{T});L_p\left((0,1);(\mathbb{R}^2)^3\right)\right)\cap L_p\left((0,\boldsymbol{T});W_p^4\left((0,1);(\mathbb{R}^2)^3\right)\right)\,.
		\end{equation*}
\end{teo}
\begin{proof}
	Given $p\in(5,\infty)$ and an initial parametrisation $\varphi$ we let $\widetilde{T}(\varphi)$ be the time of existence as in the statement of Theorem~\ref{shorttimetheta}. Let further $v\in\mathbb{R}^2$ be fixed and $(\varphi^v)^i:[0,1]\to\mathbb{R}^2$ be defined by $(\varphi^v)^i(x):=\varphi^i(x)+v$. Then $\varphi^v:=\left((\varphi^v)^1,(\varphi^v)^2,(\varphi^v)^3\right)$ is an admissible initial parametrisation to system~\eqref{thetapde}. Observe that all derivatives of $\varphi$ and $\varphi^v$ of order one or higher coincide. Thus by Theorem~\ref{shorttimetheta} there exists a time $\widetilde{T}(\varphi^v)$ depending on $\min_{i\in\{1,2,3\},x\in[0,1]}\vert\varphi^i_x(x)\vert$, $\left\vert\varphi\right\rvert_{W_p^{4-\nicefrac{4}{p}}\left((0,1)\right)}$ and $\left\lVert\varphi^v\right\rVert_{L_p((0,1))}$ such that for all $T\in(0,\widetilde{T}(\varphi^v)]$ the system~\eqref{thetapde} with initial datum $\varphi^v$ has a solution in $\boldsymbol{E}_{T}$. We want to show that $\widetilde{T}(\varphi)$ and $\widetilde{T}(\varphi^v)$ can be replaced by $\max\{\widetilde{T}(\varphi),\widetilde{T}(\varphi^v)\}$. Let $T\in (0,\widetilde{T}(\varphi)]$ and $\gamma\in\boldsymbol{E}_T$ be a solution to system~\eqref{thetapde} with initial datum $\varphi$. Then $\gamma^v:=\left((\gamma^v)^1,(\gamma^v)^2,(\gamma^v)^3\right)$ with $(\gamma^v)^i(t,x):=\gamma^i(t,x)+v$ for $t\in[0,T]$, $x\in[0,1]$, lies in $\boldsymbol{E}_T$ and is a solution to~\eqref{thetapde} with initial datum $\varphi^v$. Conversely, given $T\in(0,\widetilde{T}(\varphi^v)]$ and a solution $\eta$ to~\eqref{thetapde} with initial datum $\varphi^v$, the function $\eta^{-v}:=\left((\eta^{-v})^1,(\eta^{-v})^2,(\eta^{-v})^3\right)$ with $(\eta^{-v})^i(t,x):=\eta^i(t,x)-v$ for $t\in[0,T]$, $x\in[0,1]$, is a solution to~\eqref{thetapde} in $\boldsymbol{E}_T$ with initial datum $\varphi$. 
	With the particular choice of $v:=-\varphi(0)$ we observe that the shifted network $\varphi^v$ has one triple junction in the origin and satisfies for all $x\in[0,1]$, 
	\begin{equation*}
	\varphi^v(x)=\varphi^v(0)+\int_0^x\left\lvert\varphi_x(y)\right\rvert\mathrm{d}y\leq \left\lVert\varphi_x\right\rVert_{L^p\left((0,1);(\mathbb{R}^2)^3\right)}
	\end{equation*}
	and thus $\left\lVert\varphi^v\right\rVert_{L^p\left((0,1);(\mathbb{R}^2)^3\right)}\leq \left\lVert\varphi_x\right\rVert_{L^p\left((0,1);(\mathbb{R}^2)^3\right)}$. This shows that the existence time $$ T(\varphi):=\max\left\{\widetilde{T}(\varphi),\widetilde{T}\left(\varphi^{-\varphi(0)}\right)\right\} $$ 
	depends on $\left\lVert\varphi\right\rVert_{W_p^{4-\nicefrac{4}{p}}((0,1))}$ only via $\left\lvert\varphi\right\rvert_{W_p^{4-\nicefrac{4}{p}}((0,1))}$.
\end{proof}

As in the case of triods we obtain parabolic regularisation for system~\eqref{thetapde}, geometric existence and uniqueness, and the bounds established in \S~\ref{bounds}. Maintaining the notion of maximal solution introduced in Definition~\ref{maximalsolution} the existence and uniqueness of a maximal solution follows with the same arguments as in the case of triods, see Lemma~\ref{exuniquenssmaxsol}. Using suitable reparametrisations such that the curves of the evolving network are parametrised with constant speed equal to the length of the curve, we deduce as in Lemma~\ref{onepara} that on compact subintervals of $[0,T_{max})$ the maximal solution can be described by one parametrisation $\gamma$. Given a small $\varepsilon$ the arguments in the proof of Theorem~\ref{longtimetriod} imply
\begin{equation*}
\sup_{\delta\in (0,\frac{T_{max}}{4})}\sup_{t\in(\varepsilon,T_{max}-\delta)}\left\vert\gamma(t)\right\vert_{W_p^{4-\nicefrac{4}{p}}}\left((0,1);(\mathbb{R}^2)^3\right)\leq \boldsymbol{C}\,.
\end{equation*} 
Thanks to the refined short time existence result~\ref{shorttimethetarefined} we obtain the following Theorem:

\begin{teo}\label{longtimetheta}
	Let $p\in(5,10)$ and $\Theta_0$ be a geometrically admissible initial network. Suppose that $\left(\Theta(t)\right)_{t\in[0,T_{\max})}$ is a 
	maximal solution to
	the elastic flow with initial datum $\Theta_0$ 
	in the maximal time interval $[0,T_{\max})$ with $T_{\max}\in (0,\infty)\cup\{\infty\}$ in the sense of Definitions~\ref{smoothsol} and~\ref{maximalsolution}. Then $$T_{max}=\infty$$ or at least one of the following happens:
	\begin{itemize}
		\item[(i)]  the inferior limit of the length of at least one curve of  $\Theta(t)$ is zero as $t\nearrow T_{max}$.
		\item[(ii)] at one of the triple junctions $\liminf_{t\nearrow T_{max}}\max\left\{\left\vert\sin\alpha^1(t)\right\vert,\left\vert\sin\alpha^2(t)\right\vert,\left\vert\sin\alpha^3(t)\right\vert\right\}=0$, where $\alpha^1(t)$, $\alpha^2(t)$ and $\alpha^3(t)$ are the angles at the respective triple junction.
	\end{itemize}
\end{teo}

\subsection{Long time behaviour of general networks}

Now that we have completely described the long time behaviour 
of triods and Theta--networks, 
we are in the position to prove Theorem~\ref{main}.

First of all we
fix the following convention.

Let $\mathcal{N}$ be a  network
composed of $j\in\mathbb{N}$ curves $\mathcal{N}^1,\ldots,\mathcal{N}^j$ with $m\in\mathbb{N}$
triple junctions $O^1,\ldots,O^m$ and (if present) $l\in\mathbb{N}_0$ end--points 
$P^1, P^2,\dots, P^l$
and let $\gamma(t)=\left(\gamma^1(t),\ldots,\gamma^j(t)\right)$ 
be a regular parametrisation of $\mathcal{N}(t)$.
We assume (up to possibly reordering the family of
curves and ``inverting'' their parametrisation) that for every $r\in\{0,\ldots,l\}$
the end point $P^r$ of the network is given by $\gamma^r(t,1)$. 
Moreover, with an abuse of notation, 
we denote by $\gamma^{p,1},\gamma^{p,2},\gamma^{p,3}$
the three curves concurring at the triple junction $O^p$
and for $y\in\{1,2,3\}$ we denote by
$\tau^{py}(t,O^p), \nu^{py}(t,O^p), k^{py}(t,O^p), k_s^{py}(t,O^p)$
the respective exterior unit tangent vectors, unit normal vector, curvature and 
first arclength derivative of the curvature at $O^p$ of the three curves
$\gamma^{py}(t,\cdot)$.

\begin{dfnz}[Elastic flow of networks]\label{geometricsolutiongeneral}
Let $p\in (5,\infty)$ and $T>0$.
Let $\mathcal{N}_0$ be a geometrically admissible initial network
composed of $j$ curves $\mathcal{N}_0^1,\dots,\mathcal{N}_0^j$ with $m$
triple junctions $O^1,\ldots,O^m$ and (if present) $l$ end--points $P^1, \dots, P^l$.
A time dependent family of homeomorphic  networks $\left(\mathcal{N}(t)\right)$ 
is a solution to the elastic flow with initial datum $\mathcal{N}_0$ in  $[0,T]$ 
if and only if there exists a collection of time dependent parametrisations
	\begin{equation*}
	\gamma^i_n\in  W^1_p(I_n;L_p((0,1);\mathbb{R}^2))\cap L_p(I_n;W^4_p((0,1);\mathbb{R}^2))\,,
	\end{equation*}
	with $i\in\{1,\dots,j\}$ and $n\in\{0,\dots, N\}$ for some $N\in\mathbb{N}$, $I_n:=(a_n,b_n)\subset \mathbb{R}$, $a_n\leq a_{n+1}$, $b_n\leq b_{n+1}$, $a_n<b_n$
	and $\bigcup_n (a_n,b_n)=(0,T)$ such that for all $n\in\{0,\dots,N\}$ and $t\in I_n$, $\gamma_n(t)=\left(\gamma^1_n(t),\ldots,\gamma^j_n(t)\right)$ is a regular parametrisation of $\mathcal{N}(t)$.
	Moreover each 
	$\gamma_n$
	needs to satisfy the following system
	\begin{equation}\label{networkelastic}
	\begin{cases}
	\begin{array}{lll}
	\left(\left\langle \gamma_t^{i}, \nu^i\right\rangle  \nu^i\right)(t,x)
	=-\left((2k_{ss}^{i}+\left(k^{i}\right)^{3}-\mu k^{i})\nu^{i}\right)(t,x)& &\text{motion,}\\
	\gamma^{p,1}\left(t,O^p\right)=\gamma^{p,2}\left(t,O^p\right)=\gamma^{p,3}\left(t,O^p\right)&
	&\text{concurrency condition,}\\
	k^{p,y}(t,O^p)=0 & &\text{curvature condition,}\\
	\sum_{y=1}^{3}\left(2k_{s}^{p,y}\nu^{p,y}-\mu\tau^{p,y}\right)(t,O^p)=0& &\text{third order condition,}\\
    \gamma^r(t,1)=P^r& &\text{fixed endpoints,}\\
	k^{r}(t,1)=0 & &\text{curvature condition,}
	\end{array}
	\end{cases}
	\end{equation}
	for every  $t\in  I_n,x\in\left(0,1\right)$, $y\in\{1,2,3\}$, $p\in\{1,\ldots,m\}$, $r\in\{0,\ldots,l\}$ and for $i\in\{1,\ldots,j\}$. Finally we ask that $\gamma_n(0,[0,1])=\mathcal{N}_0$ whenever $a_n=0$. 
%	The family $\left(\Theta(t)\right)$ is a \textit{smooth} solution to the elastic flow with initial datum $\Theta_0$ in $(0,T]$ if there exists a collection $\gamma_n$, $n\in\{0,\dots,N\}$, satisfying all requirements as above such that additionally $a_1>0$, $\gamma_n^i\in C^{\infty}(\overline{I_n}\times[0,1];\mathbb{R}^2)$ for all $n\in\{1,\dots,N\}$ and for all $\varepsilon\in(0,b_0)$, $\gamma^i_0\in C^\infty\left([\varepsilon,b_0]\times[0,1];\mathbb{R}^2\right)$.
\end{dfnz}

\begin{dfnz}[Geometrically admissible initial network]
	Let $p\in(5,\infty)$. A network $\mathcal{N}_0$ composed of $j\in\mathbb{N}$ curves $\mathcal{N}_0^1,\dots,\mathcal{N}_0^j$
	with $m\in\mathbb{N}$
	triple junctions $O^1,\ldots,O^m$ and (if present) $l\in\mathbb{N}_0$ end--points 
	$P^1, P^2,\dots, P^l$
	is geometrically admissible for 
	system~\eqref{networkelastic} if
	\begin{itemize}
		\item[-] there exists a regular parametrisation $\sigma=(\sigma^1,\dots,\sigma^j)$ 
		of $\mathcal{N}_0$ 
		such that 
		\begin{equation*}
		\sigma^i\in W_p^{4-\nicefrac{4}{p}}((0,1);\mathbb{R}^2)\,.
		\end{equation*}	
		\item[-]  at each triple junction $O^p$ the three concurring curves $\sigma^{p,1},\sigma^{p,2},\sigma^{p,3}$ satisfy $k^{py}(O^p)=0$ and
		$\sum_{y=1}^{3}\left(2k^{py}_{s}\nu^{py}-\mu\tau^{py}\right)(O^p)=0$ and at least two curves form a strictly positive angle;
		\item[-] at each fixed end point $P^r$ it holds $k^r(1)=0$.
	\end{itemize}
\end{dfnz}

The notions of smooth solution (Definition~\ref{smoothsol}) and maximal solution 
(Definition~\ref{maximalsolution}) of the elastic flow of a triod
can be easily adapted to the general case.

\medskip

\textit{Proof of Theorem~\ref{main}}

The proof of geometric existence and uniqueness and parabolic smoothing presented in the
previous sections extends to the case of
a geometrically admissible initial network $\mathcal{N}_0$.
Indeed the results 
rely on the uniform parabolicity of the system and on the fact that the Lopatinskii--Shapiro
and compatibility conditions are satisfied.
All the estimates of Section~\ref{bounds} hold true for a more general network
provided that none of the lengths of the curves goes to zero and the uniform non--degeneracy condition
is satisfied.
If the network $\mathcal{N}_0$ has at least one fixed end point, then we obtain the 
desired result as a corollary of Theorem~\ref{longtimetriod}. If instead 
the network $\mathcal{N}_0$ has only triple junctions but no fixed end points, the theorem
follows using the refined version of the short time existence (Theorem~\ref{shorttimethetarefined}) 
that allows to conclude as in 
Theorem~\ref{longtimetheta}.
\qed

\medskip

\medskip

We conclude the paper with some observations.

\medskip

None of the possibilities listed in Theorem~\ref{main} excludes the others. Indeed it is possible that as the time approaches $T_{max}$, both the lengths of one or several curves of the network and the angles between the curves at one or more triple junctions tend to zero, regardless of $T_{\max}$ being finite or infinite.
To convince the reader of the high chances of this phenomenon
we say a few words about the minimization problem in the class
of Theta--networks naturally associated to the flow, namely
$$
\inf \{E_{\mu}(\Theta)\,\vert \, \Theta\; \mbox{is a Theta--network}\}\,.
$$
The infimum is zero and it is not a minimum. 
An example of a minimizing sequence is the following:
two arcs of a circle of radius $1$ and of length $\varepsilon$
that meet with a segment (of length $\sqrt{2}\sqrt{1-\cos\varepsilon}\sim\varepsilon$)
forming  angles of $\frac\varepsilon2$. Then 
$E_{\mu}(\Theta)=2\varepsilon+\mu(2\varepsilon+\sqrt{2}\sqrt{1-\cos\varepsilon})$.
Letting $\varepsilon\to 0$, the lengths of all curves and the angles at both triple junctions tend to zero and $E_{\mu}(\Theta)\to 0$.

\medskip

If the network $\mathcal{N}_0$ has no fixed end points (as in the case of the Theta--network)
we are not able to exclude that as $t\to T_{\max}$ the entire
configuration $\mathcal{N}_t$ ``escapes" to infinity.
In the case of networks with at least one fixed end point $P^1$ in $\mathbb{R}^2$ the global length of $\mathcal{N}_t$ is bounded by
$\frac{1}{\mu}E_{\mu}(\mathcal{N}_0)=:R$ (see Remark~\ref{boundL}). Hence
as $t\to T_{\max}$ the entire $\mathcal{N}_t$ remains in a ball of center $P^1$
and radius $R$. 
\medskip

There is a slight difference in the point $(i)$ of Theorem~\ref{longtimetriod}
with respect to Theorem~\ref{longtimetheta}
and Theorem~\ref{main}.
The global length of a triod is bounded from below away from zero by the value of the length of the
shortest path connecting the three distinct fixed end points $P^1$ , $P^2$ and $P^3$.
Unfortunately this does not give a bound on the length of the single curves,
but clearly the length of \emph{at most} one curve can go to zero during the evolution.

\medskip

Consider now the case of the Theta: as $t\to T_{\max}$  
the length of more than one curve can go to zero if 
the angles go to zero.
Suppose by contradiction that the lengths of the curves 
$\gamma^1$ and $\gamma^2$ go to zero and 
that all the angles are uniformly bounded
away from zero.
We can see the union  of $\gamma^1$ and $\gamma^2$ as a 
 closed curve with two angles $\alpha,\beta$. 
Then a consequence of~\cite[Theorem~A.1]{danovplu}
is that if the lengths of both $\gamma^1$ and $\gamma^2$ 
go to zero but the angles are uniformly bounded away from zero,
the $L^2$--norm of the scalar curvature blows up, 
a contradiction to the bound in~\ref{boundL}.

\medskip

\medskip

\textbf{Acknowledgements}

The authors gratefully acknowledge the support by the Deutsche Forschungsgemeinschaft (DFG) via the GRK 1692 “Curvature, Cycles, and Cohomology”.

\bibliographystyle{amsplain}

\bibliography{Elastic-flow-II_final2.bib}

\providecommand{\bysame}{\leavevmode\hbox to3em{\hrulefill}\thinspace}
\providecommand{\MR}{\relax\ifhmode\unskip\space\fi MR }
% \MRhref is called by the amsart/book/proc definition of \MR.
\providecommand{\MRhref}[2]{%
  \href{http://www.ams.org/mathscinet-getitem?mr=#1}{#2}
}
\providecommand{\href}[2]{#2}
\begin{thebibliography}{10}

\bibitem{AdamsFournier}
R.~A. Adams and J.~Fournier, \emph{Sobolev spaces}, second ed., Pure and
  Applied Mathematics (Amsterdam), vol. 140, Elsevier/Academic Press,
  Amsterdam, 2003. \MR{2424078}

\bibitem{Amann}
H.~Amann, \emph{Linear and quasilinear parabolic problems. {V}ol. {I}},
  Monographs in Mathematics, vol.~89, Birkh\"{a}user Boston, Inc., Boston, MA,
  1995, Abstract linear theory. \MR{1345385}

\bibitem{bargarnu}
J.~W. Barrett, H.~Garcke, and R.~N\"urnberg, \emph{Elastic flow with junctions:
  variational approximation and applications to nonlinear splines}, Math.
  Models Methods Appl. Sci. \textbf{22} (2012), no.~11, 1250037, 57.
  \MR{2974175}

\bibitem{Bronsardreitich}
L.~Bronsard and F.~Reitich, \emph{On three-phase boundary motion and the
  singular limit of a vector-valued ginzburg-landau equation}, Archive for
  Rational Mechanics and Analysis \textbf{124} (1993), no.~4, 355--379.

\bibitem{lin2}
A.~Dall'Acqua, C.~C. Lin, and P.~Pozzi, \emph{Evolution of open elastic curves
  in {$\Bbb{R}^n$} subject to fixed length and natural boundary conditions},
  Analysis (Berlin) \textbf{34} (2014), no.~2, 209--222. \MR{3213535}

\bibitem{dallacqualinpozzi}
\bysame, \emph{Flow of elastic networks: long-time existence result},
  arXiv:1812.11367, 2018.

\bibitem{danovplu}
A.~Dall'Acqua, M.~Novaga, and A.~Pluda, \emph{Minimal elastic networks}, to
  appear: Indiana Univ. Math. J.

\bibitem{dalpoz}
A.~Dall'Acqua and P.~Pozzi, \emph{A {W}illmore-{H}elfrich {$L^2$}-flow of
  curves with natural boundary conditions}, Comm. Anal. Geom. \textbf{22}
  (2014), no.~4, 617--669. \MR{3263933}

\bibitem{DepnerGarckeKohsaka}
D.~Depner, H.~Garcke, and Y.~Kohsaka, \emph{Mean curvature flow with triple
  junctions in higher space dimensions}, Arch. Ration. Mech. Anal. \textbf{211}
  (2014), no.~1, 301--334. \MR{3182482}

\bibitem{dziukkuwertsch}
G.~Dziuk, E.~Kuwert, and R.~Sch\"atzle, \emph{Evolution of elastic curves in
  {$\Bbb R^n$}: existence and computation}, SIAM J. Math. Anal. \textbf{33}
  (2002), no.~5, 1228--1245. \MR{1897710}

\bibitem{Freire}
A.~Freire, \emph{Mean curvature motion of triple junctions of graphs in two
  dimensions}, Comm. Partial Differential Equations \textbf{35} (2010), no.~2,
  302--327. \MR{2748626}

\bibitem{GarckeItoKohsaka}
H.~Garcke, K.~Ito, and Y.~Kohsaka, \emph{Surface diffusion with triple
  junctions: a stability criterion for stationary solutions}, Adv. Differential
  Equations \textbf{15} (2010), no.~5-6, 437--472. \MR{2643231}

\bibitem{garmenzplu}
H.~Garcke, J.~Menzel, and A.~Pluda, \emph{Willmore flow of planar networks},
  Journal of Differential Equations \textbf{266} (2019), no.~4, 2019 -- 2051.

\bibitem{garckenov}
H.~Garcke and A.~Novick-Cohen, \emph{A singular limit for a system of
  degenerate {C}ahn-{H}illiard equations}, Adv. Differential Equations
  \textbf{5} (2000), no.~4-6, 401--434. \MR{1750107}

\bibitem{michaelgoesswein}
M.~G\"{o}\ss{}wein, \emph{Surface diffusion flow of triple junction clusters in
  higher space dimensions}, Ph.D. thesis, Universit\"{a}t Regensburg, 2019.

\bibitem{Lee}
J.~M. Lee, \emph{Introduction to smooth manifolds}, second ed., Graduate Texts
  in Mathematics, vol. 218, Springer, New York, 2013. \MR{2954043}

\bibitem{lin1}
C.~C. Lin, \emph{{$L^2$}-flow of elastic curves with clamped boundary
  conditions}, J. Differential Equations \textbf{252} (2012), no.~12,
  6414--6428. \MR{2911840}

\bibitem{Carlo2}
A.~Magni, C.~Mantegazza, and M.~Novaga, \emph{Motion by curvature of planar
  networks {II}}, Ann. Sc. Norm. Sup. Pisa \textbf{15} (2016), 117--144.

\bibitem{Carlo3}
C.~Mantegazza, M.~Novaga, A.~Pluda, and F.~Schulze, \emph{Evolution of networks
  with multiple junctions}, preprint 2016.

\bibitem{mannovtor}
C.~Mantegazza, M.~Novaga, and V.~M. Tortorelli, \emph{Motion by curvature of
  planar networks}, Ann. Sc. Norm. Sup. Pisa \textbf{3 (5)} (2004), 235--324.

\bibitem{novok}
M.~Novaga and S.~Okabe, \emph{Curve shortening-straightening flow for
  non-closed planar curves with infinite length}, J. Differential Equations
  \textbf{256} (2014), no.~3, 1093--1132. \MR{3128933}

\bibitem{okabe}
S.~Okabe, \emph{The existence and convergence of the shortening-straightening
  flow for non-closed planar curves with fixed boundary}, International
  {S}ymposium on {C}omputational {S}cience 2011, GAKUTO Internat. Ser. Math.
  Sci. Appl., vol.~34, Gakk\=otosho, Tokyo, 2011, pp.~1--23. \MR{3223122}

\bibitem{pluda}
A.~Pluda, \emph{Evolution of spoon--shaped networks}, Network and Heterogeneus
  Media \textbf{11} (2016), no.~3, 509--526.

\bibitem{poldenthesis}
A.~Polden, \emph{Curves and surfaces of least total curvature and fourth-order
  flows}, Ph.D. thesis, Universit\"{a}t T\"{u}bingen, 1996.

\bibitem{Prusssimonett}
J.~Pr\"{u}ss and G.~Simonett, \emph{Moving interfaces and quasilinear parabolic
  evolution equations}, Monographs in Mathematics, vol. 105,
  Birkh\"{a}user/Springer, [Cham], 2016. \MR{3524106}

\bibitem{schulzewhite}
F.~Schulze and B.~White, \emph{A local regularity theorem for mean curvature
  flow with triple edges}, ArXiv Preprint Server -- http:/$\!\!$/$\!$arxiv.org,
  to appear in J. Reine Angew. Math., 2016.

\bibitem{Simon}
J.~Simon, \emph{Sobolev, {B}esov and {N}ikolskii fractional spaces: imbeddings
  and comparisons for vector valued spaces on an interval}, Ann. Mat. Pura
  Appl. (4) \textbf{157} (1990), 117--148. \MR{1108473}

\bibitem{solonnikov2}
V.~A. Solonnikov, \emph{Boundary value problems of mathematical physics. {III},
  proceedings of the steklov institute of mathematics, no. 83 (1965)}, Amer.
  Math. Soc., Providence, R.I., 1967.

\bibitem{glen}
G.~Wheeler, \emph{Global analysis of the generalised helfrich flow of closed
  curves immersed in $ \mathbb{R}^n$}, Trans. Amer. Math. Soc. \textbf{367}
  (2015), 2263--2300.

\end{thebibliography}
\end{document}